\documentclass[a4paper,10pt]{amsart}
\usepackage[colorlinks,linkcolor=blue,citecolor=blue]{hyperref}
\usepackage{amssymb,stmaryrd}
\usepackage{amsfonts}
\usepackage{amstext}
\usepackage{algorithmic}
\usepackage{algorithm}
\usepackage{young}
\usepackage{graphicx}
\usepackage{epstopdf}
\usepackage[all]{xy}
\usepackage{enumerate}
\usepackage{color}
\usepackage{epic}

\usepackage{tikz}
\usetikzlibrary{matrix,arrows,decorations.pathmorphing}

\numberwithin{equation}{section}
\newtheorem{theorem}{Theorem}[section]
\newtheorem{lemma}[theorem]{Lemma}
\newtheorem{proposition}[theorem]{Proposition}
\newtheorem{corollary}[theorem]{Corollary}

\theoremstyle{definition}
\newtheorem{definition}[theorem]{Definition}
\newtheorem{example}[theorem]{Example}

\theoremstyle{remark}
\newtheorem{remark}[theorem]{\bf{Remark}}

\newcommand{\C}{{\mathbb{C}}}
\newcommand{\D}{{\Delta}}
\newcommand{\e}{{\epsilon}}
\newcommand{\Z}{{\mathbb{Z}}}

\newcommand{\<}{{\langle}}
\renewcommand{\>}{{\rangle}}

\newcommand{\CO}{{\mathcal{O}}}
\newcommand{\CC}{{\mathcal{C}}}

\newcommand{\isom}{{\cong}}
\newcommand{\Ad}{{\rm Ad}}
\newcommand{\End}{{\rm End}}
\newcommand{\Sh}{{\rm Sh}}

\renewcommand{\ker}{{\rm{ker}}}

\newcommand{\ver}{{\rm{ver}}}
\newcommand{\im}{{\rm{Im}}}
\newcommand{\kk}{{k}}

\newcommand{\tens}{\otimes}
\newcommand{\id}{{\rm id}}
\newcommand{\bo}{{}^{(1)}}
\newcommand{\bt}{{}^{(2)}}
\renewcommand{\o}{{}_{(1)}}
\renewcommand{\t}{{}_{(2)}}
\renewcommand{\th}{{}_{(3)}}
\newcommand{\four}{{}_{(4)}}
\newcommand{\z}{{}_{(0)}}
\newcommand{\mo}{{}_{(-1)}}
\newcommand{\mt}{{}_{(-2)}}
\newcommand{\mth}{{}_{(-3)}}
\newcommand{\extd}{{\rm d}}
\newcommand{\del}{{\partial}}
\newcommand{\eps}{\epsilon}

\newcommand{\la}{{\triangleright}}
\newcommand{\ra}{{\triangleleft}}

\newcommand{\lbiprod}{{>\!\!\!\triangleleft\kern-.33em\cdot}}
\newcommand{\rbiprod}{{\cdot\kern-.33em\triangleright\!\!\!<}}

\newcommand{\rcross}{{\triangleright\!\!\!<}}

\begin{document}

\title[Generalised noncommutative geometry on quivers]{Generalised noncommutative geometry on finite groups and Hopf quivers}


\author[S. Majid]{Shahn Majid}
\address{Queen Mary University of London\\
School of Mathematical Sciences, Mile End Rd, London E1 4NS, UK}
\email{s.majid@qmul.ac.uk}
\author[W.-Q. Tao]{Wenqing Tao\textsuperscript{$\dagger$}}
\thanks{\textsuperscript{$\dagger$}The second author was supported by the China Scholarship Council and funded by 2016YXMS006 and NSFC11601167.}
\address{Huazhong University of Science and Technology\\
School of Mathematics and Statistics, Wuhan, Hubei 430074, P.R. China}
\email{wqtao@hust.edu.cn}

\date{}

\begin{abstract}
We explore the differential geometry of finite sets where the differential structure is given by a quiver
rather than as more usual by a graph. In the finite group case we show that the data for such a differential calculus is described by certain Hopf quiver data as familiar in the context of path algebras. We explore a duality between geometry on the function algebra vs geometry on the group algebra, i.e. on the dual Hopf algebra, illustrated by the noncommutative Riemannian geometry of the group algebra of $S_3$. We show how quiver geometries arise naturally in the context of quantum principal bundles. We provide a formulation 
of bimodule Riemannian geometry for quantum metrics on a quiver, with a fully worked example on 2 points; in the quiver case, metric data assigns matrices not real numbers to the edges of a graph. The paper builds on the general theory in our previous work\cite{MaTao1}. 
\end{abstract}

\maketitle


\section{Introduction}

Noncommutative differential geometry is an extension of geometry to the case where the coordinate algebra $A$ may be noncommutative or `quantum'. The starting point is a `differential structure' on $A$ defined as a pair $(\Omega^1,\extd)$ where the space of `1-forms' $\Omega^1$ is an $A$-$A$-bimodule (so one can multiply by the algebra from the left or the right) and $\extd:A\to \Omega^1$ obeys the Leibniz rule. One usually requires that the map $\phi:A\tens A\to \Omega^1$ given by $\phi(a\tens b)=a\extd b$, $a,b\in A$ is surjective. Recently in \cite{MaTao1} we began a study of differentials in noncommutative geometry where this condition is dropped. One still has a standard differential calculus $\bar\Omega^1=\phi(A\tens A)$ given by the image of this map, but there are many situations where this image is not easy to describe and where a larger $\Omega^1$ is the much more natural object. This also links in to the wider use of differential graded algebras in other contexts where such surjectivity need not be assumed. As for usual calculi, $(\Omega^1,\extd)$ is inner if the derivative $\extd$ can be written as a commutator
\[\extd a=[\theta, a],\quad\forall\, a\in A,\] 
for some $\theta\in\Omega^1$, but this is now more general as we do not require $\theta\in\bar\Omega^1$. Indeed, it is observed in~\cite{MaTao1} that any generalised differential calculus can be embedded into an inner one, thus we will be particularly interested in inner calculi.

We begin by briefly recalling some results from \cite{MaTao1} that we will be specialising to the finite case, in the Preliminaries section. The main difference is that whereas a bicovariant calculus is usually, following \cite{Wor}, thought about in terms of Ad-stable right ideals in the augmentation ideal $A^+$ (the kernel of the counit, or in classical geometry the functions that vanish at the identity), in the generalised case what we are interested in is just  all crossed-module morphisms $A^+\to \Lambda^1$ from a given fixed crossed module (or Drinfeld-Radford-Yetter module) structure on $A^+$. Another key idea is to augment the notion of a bicovariant calculus by a degree $-1$ map $i$ going the other way. This augmented notion is then self-dual in the finite-dimensional case, i.e. dualisation gives us the same type of data but on the dual Hopf algebra.

Section~3.1 specialises the theory of \cite{MaTao1} to the case of $A=k(G)$, the functions on a finite group and dually the codifferentials on a group algebra $A=kG$. The former generalises the well-known case of ordinary noncommutative bicovariant differentials given by ad-stable subsets, now given by Hopf quivers, and connects with the quiver path (co)algebra in the sense of \cite{Rosso,Hua}. By contrast, Section~3.2 covers the dual notion of bicovariant differentials on $kG$ and codifferentials on $k(G)$, which turn out  to correspond simply to group 1-cocycles $\zeta\in Z^1(G,\Lambda^1{}^*)$. Such cocycles are principally obtained from characters or matrix elements of irreducible representations of the group and, remarkably, again turn out to be characterised by certain Hopf quiver data. 

Section~4 turns to some elements of noncommutative differential geometry outlined in general terms in Section~4.1. In Section~4.2, working with ordinary differential structures, we explore our duality notions in detail on the group $S_3$ of permutations of 3 elements. Calculi on its group algebra  are given by irreducible representations while eigenvectors for an associated Laplacian are given by ad-stable subsets, the exact reverse of the situation in finite group function algebras (where calculi are given by ad-stable subsets as mentioned above while eigenfunctions of the laplacian are given by matrix elements of irreducible representations). We also find a natural quantum metric $g\in\Omega^1\tens_A\Omega^1$ for $A=\C S_3$ with its 4D calculus associated to the 2D irreducible of $S_3$ and we find its moduli of bimodule torsion free and cotorsion free or `weak quantum Levi-Civita' connections to complement what is known on $\C(S_3)$ in  \cite{Ma:rieq}. In Section~4.3 we show how quiver geometries arise  naturally in the context of quantum principal bundles as invariants $\Omega(X)^G$ of standard (graph) calculi on $X$ and give a geometric description of the quiver under certain conditions. Finally, in Section~4.4 we turn to quiver noncommutative geometry itself in the form of noncommutative Riemannian geometry with generalised differentials, including two  fully worked examples. The second of these is a quiver calculus on $\C(\Z_2)$ and  we are able to find a full 4-functional parameter moduli of quantum Levi-Civita connections for a given quiver quantum metric given by a matrix at each of the two points. 

\section{Preliminaries}

We work over a field $k$ of characterisitic not 2. Here we briefly recall and slightly improve some basic results in \cite{MaTao1} that will be needed. We also then recall some background on Hopf quivers. 

\subsection{Generalised differential structures} 

Throughout the paper, unless otherwise specified, a differential calculus will normally mean a generalised one in the sense of a pair $(\Omega^1,\extd)$ where $\Omega^1$ is a bimodule and $\extd (ab)=a\extd b+(\extd a)b$ for all $a,b\in A$. If the surjectivity of the map $\phi$ holds as discussed above, we speak of a {\em standard} calculus and in particular of the standard subcalculus $\bar\Omega^1\subseteq\Omega^1$ associated to any possibly generalised one.  Two first order differential (generalised) calculi $(\Omega^1,\extd)$ and $(\Omega'{}^1,\extd')$ on $A$ are said to be isomorphic if there exists a bimodule isomorphism $\varphi:\Omega^1\to\Omega'{}^1$ such that $\varphi(\extd a)=\extd' a$ for any $a\in A.$ A calculus is {\em connected} if $\ker\extd=k.1$. 

We recall that a Hopf algebra is an algebra $A$ equipped with a coalgebra structure $(\Delta,\eps)$ where $\Delta:A\to A\tens A$, $\eps:A\to k$  are algebra maps and where we are further provided with a `linearised inverse' or antipode $S$ obeying $(S a\o)a\t=a\o S a\t=1\eps(a)$ where we use the Sweedler notation $\Delta a=a\o\tens a\t$ (summation of such terms understood).  We recall that an $A$-crossed module is a vector space $V$ which is both a module and a comodule and these are compatible in the sense (here we work in the right-handed theory)
\[ \Delta_R(v\ra a)=v\z\ra a\t \tens (Sa\o)v\o a\th,\quad\forall a\in A,\ v\in V,\]
where $\Delta_Rv=v\z\tens v\o$ is a notation (summation understood) for the coaction. A coaction is like an action but with all arrows reversed.  We let $A^+$ denote the augmentation ideal, defined as the kernel of the counit. This forms an $A$-crossed module with
\begin{equation}\label{Aplus1} 
a\ra b=ab,\quad \Delta_R(a)=a\t\tens Sa\o a\th ,\quad\forall a\in A^+,\ b\in A
\end{equation}
i.e., the right regular action and adjoint coaction respectively. We let $\pi:A\to A^+$,  $\pi (a)=a-\eps(a)$ be the counit projection.

In the Hopf algebra case we say that a calculus $\Omega^1$ is \textit{right-covariant} if it is equipped with a right coaction such that $\extd:A\to \Omega^1$ intertwines this with the right regular coaction of $A$ on itself given by the coproduct $\Delta$ and where $\Delta_R$ is a bimodule map. $\Omega^1$ is {\em bicovariant} if this holds and if it is similarly \textit{left-covariant} under a coaction $\Delta_L$, and the two coactions commute. It is shown in \cite[Theorem 2.5]{MaTao1} that bicovariant generalised differential calculi on a Hopf algebra $A$ are isomorphic to ones of the form 
\begin{equation}\label{fdiff} \Omega^1=A\tens\Lambda^1,\quad \extd a=a\o\tens\omega\circ\pi (a\t),\quad \omega:A^+\to \Lambda^1\end{equation}
 given by any data $(\Lambda^1,\omega)$ where $\Lambda^1$ is a right $A$-crossed module and $\omega$ is a morphism in the category of right crossed modules. Left covariant ones correspond to $\Lambda^1$ merely a right $A$-module and $\omega:A^+\to \Lambda^1$  a right $A$-module map. The image $\bar\Lambda^1=\omega(A^+)$ and $\omega$ give the standard sub-calculus. Also, it is easy to see that a left covariant (generalised) calculus is inner if and only if here exists $\theta\in \Lambda^1$ such that $\omega(a)=\theta\ra\,a$ for any $a\in A^+$. This inner calculus is bicovariant if and only if we have $\Delta_R$ making $\Lambda^1$ a crossed module with 
\[ \theta\z\ra a\t\tens Sa\o \theta\o a\th-\theta\ra a\t\tens Sa\o a\th= \eps(a)(\Delta_R\theta-\theta\tens 1),\quad\forall a\in A.\]

Next we recall the notion of a {\em differential graded algebra} over an algebra $A$, meaning a graded algebra $\Omega=\oplus_{n\ge 0}\Omega^n$ with $\Omega^0=A$ equipped with  map $\extd:\Omega\to \Omega$ obeying the super-Leibniz rule $\extd(uv)=(\extd u)v+(-1)^nu\extd v$ for all $u\in\Omega^n,v\in\Omega$ and such that $\extd^2=0$. The differential graded algebra is \textit{inner} if $\extd=[\theta, \ \}$ for some $\theta\in\Omega^1$, where $[\theta,u\}=\theta u-(-1)^n u\theta$ for any $u\in\Omega^n$. Clearly the degree 1 component is a (generalised) differential calculus as above. The standard case is with $\Omega^1$ standard (the surjectivity assumption) and $\Omega$ algebraically generated by $A,\Omega^1$. It is a crucial question whether a given (inner) first order calculus extends to higher orders. When $A$ is a Hopf algebra, notions of left, right and bicovariance extend in the obvious way: we require graded coactions $\Delta_L,\Delta_R$ now on all degrees with analogous properties. We slightly extend a result in \cite{MaTao1}:

\begin{proposition}
Any left covariant (or bicovariant) differential graded algebra $(\Omega,\extd)$ on a Hopf algebra $A$ can be embedded into an inner left covariant (or bicovariant) differential graded algebra $(\widehat\Omega,\widehat\extd)$ such that $\widehat\extd|_{\Omega}=\extd.$ If $(\Omega,\extd)$ is standard, then $\Omega=\bar{\widehat\Omega},$ the standard sub-calculus of the extended calculus.
\end{proposition}
\proof
From \cite[Proposition 3.3]{MaTao1}, we know such $(\Omega,\extd)$ corresponds to $(\Lambda,\omega,\delta)$ where $\Lambda$ is a graded right $A$-module algebra with $\Lambda^0=k,$ $\omega:A^+\to\Lambda^1$ is a right $A$-module map and $\delta:\Lambda\to\Lambda$ is degree $1$ super-derivation ($\delta(\xi\eta)=(\delta\xi)\eta+(-1)^{|\xi|}\xi(\delta\eta)$ and $\delta^2=0$) such that 
\begin{gather}
(\delta\eta)\ra a-\delta(\eta\ra a)=\omega\pi(a\o)(\eta\ra a\t)-(-1)^{|\eta|}(\eta\ra a\o)\omega\pi(a\t),\label{deltaraa}\\
\delta(\omega\pi(a))+\omega\pi(a\o)\omega\pi(a\t)=0,\label{deltaomega}
\end{gather}
for all $\xi,\eta\in\Lambda,a\in A.$ Without loss of generality, we can assume $\Omega=A\rcross\Lambda$ with differential $\extd(a\tens\eta)=a\o\tens\omega\pi(a\t)\eta+a\tens\delta\eta.$

Now let $\widehat\Lambda=\Lambda\tens\Lambda(k\theta)=\Lambda\tens 1\oplus\Lambda\tens\theta$ with product $\wedge:\widehat\Lambda\tens\widehat\Lambda\to\widehat\Lambda$ given by
\begin{gather*}
(\xi\tens\theta)\wedge\eta=\xi\delta\eta+(-1)^{|\eta|}\xi\eta\tens\theta,\\
\xi\wedge(\eta\tens\theta)=\xi\eta\tens\theta,\\
(\xi\tens\theta)\wedge(\eta\tens\theta)=\xi\delta\eta\tens\theta,
\end{gather*}
for all $\xi,\eta\in\Lambda.$ Clearly, $\widehat\Lambda$ is graded with $\deg \theta=1,$ and hence $\widehat\Lambda^0=k,$ $\widehat\Lambda^1=\Lambda^1\tens 1\oplus 1\tens k\theta,$ and $\widehat\Lambda^n=\Lambda^n\tens 1\oplus\Lambda^{n-1}\tens\theta$ for $n\ge 2$. Then we define
\begin{equation*}
(\eta\tens\theta)\ra a=(\eta\ra a\o)\omega\pi(a\t)+(\eta\ra a)\tens\theta.\quad\forall\,a\in A.
\end{equation*}
We are making this construction so that $(\widehat\Lambda,\wedge)$ is a graded right $A$-module algebra. One can easily see it is true by checking $(\theta\wedge\eta)\ra a=(\theta\ra a\o)\wedge(\eta\ra a\t).$ The left-hand side is 
\begin{align*}
(\theta\wedge\eta)\ra a&=(\delta\eta)\ra a+(-1)^{|\eta|}(\eta\tens\theta)\ra a,\\
&=(\delta\eta)\ra a+(-1)^{|\eta|}(\eta\ra a\o)\omega\pi(a\t)+(-1)^{|\eta|}(\eta\ra a)\tens\theta),
\end{align*}
while the right-hand side is 
\begin{align*}
(\theta\ra a\o)\wedge(\eta\ra a\t)&=(\omega\pi(a\o)+\epsilon(a\o)\theta)\wedge(\eta\ra a\t)\\
&=\omega\pi(a\o)(\eta\ra a\t)+\theta\wedge(\eta\ra a)\\
&=\omega\pi(a\o)(\eta\ra a\t)+\delta(\eta\ra a)+(-1)^{|\eta|}(\eta\ra a)\tens\theta.
\end{align*}
Two sides agree with each other by (\ref{deltaraa}). Also $(\theta\ra a\o)\wedge(\theta\ra a\t)=0=(\theta\wedge\theta)\ra a$ follows from (\ref{deltaomega}). Obviously $\iota:\Lambda\to\widehat{\Lambda}$ that maps $\eta$ to $\eta\tens 1$ is an injective algebra map that respects grading and commutes with right $A$-action.

Consider maps $\widehat{\omega}:A^+\to\widehat{\Lambda}^1$ and $\widehat{\delta}:\widehat\Lambda\to \widehat\Lambda$ that are defined by 
\[\widehat\omega(a)=\omega(a),\quad\widehat\delta\eta=\delta\eta,\quad\widehat{\delta}(\eta\tens\theta)=\delta\eta\tens\theta,\quad\forall\,a\in A^+,\eta\in\Lambda.\]
Clearly $\widehat\omega$ is a right $A$-module map and 
$\widehat{\extd}a=a\o\tens\widehat\omega\pi(a\t)=a\o\tens\theta\ra a\t-a\theta=[\theta,a]$ while for all $\eta\in \Lambda$, 
\begin{gather*} \widehat\delta\eta=\delta\eta=\theta\wedge\eta-(-1)^{|\eta|}\eta\tens\theta=[\theta,\eta\},\\
\widehat\delta(\eta\tens\theta)=\delta\eta\tens\theta=\theta\wedge(\eta\tens\theta)+(-1)^{|\eta|}(\eta\tens\theta)\wedge\theta=[\theta,\eta\tens\theta\}.
\end{gather*}
So we know from \cite[Proposition 3.5]{MaTao1} that $(\widehat\Lambda,\widehat\omega, \widehat\delta=[\theta,\ \})$ defines an inner left covariant differential graded algebra  $(\widehat{\Omega}=A\rcross\widehat\Lambda,\widehat\extd)$ as $\theta\wedge\theta=0$ in $\widehat\Lambda.$ In particular, $\widehat{\extd}(a\tens\eta)=a\o\tens\widehat\omega\pi(a\t)\wedge\eta+a\tens\widehat\delta\eta
=a\o\tens\omega\pi(a\t)\eta+a\tens\delta\eta=\extd(a\tens\eta)$ shows that the embedding $\iota:\Omega\to\widehat{\Omega}$ (extending $\iota:\Lambda\to\widehat{\Lambda}$) commutes with
differential structures and thus becomes a morphism of left covariant differential graded algebra. In the bicovariant case, one can simply take the right $A$-action on $\theta$ to be trivial and check $\widehat\Lambda$ is also a right $A$-comodule algebra. Again, from \cite[Proposition 3.5]{MaTao1}, the corresponding $(\widehat\Omega,\widehat\extd)$ is bicovariant as the only additional condition $[\Delta_R\theta-\theta\tens 1,\Delta_R\}=0$ holds automatically.
\endproof

It is known that standard first order bicovariant differential calculi have a `minimal' extension to a bicovariant differential exterior algebra, due to Woronowicz\cite{Wor}, and which is known\cite{Brz} to be a super-Hopf algebra.  This motivates the following definition for generalised calculi:

\begin{definition}[\cite{MaTao1}]\label{sdga} We say that a differential graded algebra $(\Omega,\extd)$ over a Hopf algebra $A=\Omega^0$ is {\em strongly bicovariant} if $\Omega$ is a $\mathrm{N}_0$-graded super-Hopf algebra with odd/even part given by the parity of the grading and the super-derivation $\extd$ is also a `\textit{super-coderivation}' in the sense that
\begin{equation*}
\Delta\circ\extd(w)=(\extd\otimes\mathrm{id}+(-1)^{|\ |}\otimes\extd)\Delta(w),\quad \forall w\in\Omega,
\end{equation*} where $\Delta=(\ )\o\otimes(\ )\t$ is the graded-super coproduct of $\Omega$ and $(-1)^{|\ |}w=(-1)^{|w|}w$ according to the degree.
\end{definition}
By definition the coproduct respects the grading so that $\Delta(\Omega^n)\subseteq\oplus_{i,j=0}^{i+j=n}\Omega^i\otimes \Omega^j.$ The super-coderivation condition in Definition~\ref{sdga} is a new observation even in the standard case and is key to what follows.  Our terminology is justified by the fact that any strongly bicovariant differential exterior algebra is bicovariant~\cite{MaTao1}. Given data $(\Lambda^1,\omega)$ (or $(\Lambda^1,\theta)$) on $A$, we recall in particular three constructions of strongly bicovariant differential graded algebras on $A$ in~\cite[Props. 3.10, 3.11 and 3.13]{MaTao1}:
\begin{itemize}
\item[d1)] The `universal inner' calculus $\Omega_\theta(A)=A\rbiprod\Lambda_{\theta}(\Lambda^1)$ with $\extd=[\theta,\ \},$ where $\Lambda_\theta(\Lambda^1)=T_-\Lambda^1/\<\theta^2\ra a, [\theta^2,\eta]\ |\ a\in A^+,\ \eta\in \Lambda^1\>$ and $\Delta_R\theta=\theta\tens1.$ Any inner strongly bicovariant differential graded algebra on $A$ with $\theta$ right invariant that is generated by its degree $0,1$ is a quotient of such $\Omega_\theta(A)$;
\item[d2)] The `braided-exterior calculus' $\Omega_w(A)=A\rbiprod B_-(\Lambda^1)$, $\extd=[\theta,\ \}$ provided\begin{equation*}\label{wor-con}
\Delta_R\theta-\theta\tens 1\in \Lambda^1\square A,\quad \Psi(\eta\tens\theta)=\theta\tens\eta, \quad\{\Delta_R-\theta\tens1,\Delta_R\eta\}=0,
\end{equation*}
for all $\eta\in\Lambda^1.$ Here $\Lambda^1\square A=\{\sum\eta_i\tens a_i\in\Lambda^1\tens A\,|\,\sum\eta_i\ra b\tens a_i=\sum\eta_i\tens b\o a_i Sb\t,\forall\,b\in A\}$ and $\Psi$ is the crossed module (pre)braiding on $\Lambda^1$ defined by $\Psi(\xi\tens\eta)=\eta\z\tens\xi\ra\eta\o$ and  $B_-(\Lambda^1)$ is the braided exterior algebra of $\Lambda^1$ associated to this (with relations given by braided-antisymmetrization). This extends the Woronowicz construction to the generalised case using the braided-Hopf algebra point of view introduced for this in
\cite{Ma:dcalc};
\item[d3)] The `shuffle calculus' $\Omega_{sh}(A)=A\rbiprod \Sh_-(\Lambda^1)$, where $\Sh_-(\Lambda^1)$ is the super-braided shuffle algebra of $\Lambda^1$ in the category of right $A$-crossed modules and $\extd$ is determined recursively from degree $1$ as a super-coderivation with $\delta\eta=-\eta\z\tens\omega\pi(\eta\o)$ so that
\[\extd (a\tens\eta)=a\o\tens\omega\pi(a\t)\tens\eta-a\o\tens\eta\z\tens\omega\pi(a\t\eta\o),\quad\forall\,\eta\in\Lambda^1.\] Here  $\Omega_{sh}(A)$ is inner for some element $\theta\in \Lambda^1$ if and only if $\theta$ makes $(\Lambda^1,\omega)$ inner and $\Psi(\eta\tens\theta)=\theta\tens \eta$ for any $\eta\in \Lambda^1.$
\end{itemize}

\subsection{Codifferential calculi}

 A \textit{first order codifferential structure} on a coalgebra $H$ is \cite{MaTao1} a pair $(\Omega^1,i),$ where $\Omega^1$ is a $H$-$H$-bicomodule and $i:\Omega^1\to H$ is a linear map such that
\begin{equation}\label{coLebniz} \Delta\circ i=(i\tens\id)\Delta_R+ (\id\tens i)\Delta_L.
\end{equation}
A codifferential calculus is said to be \textit{coinner} if the coderivation $i$ is given by
\[i(\eta)=\<\vartheta,\eta_{(0)}\>\eta_{(1)}-\eta_{(-1)}\<\vartheta,\eta_{(0)}\>,\quad \forall\,\eta\in\Omega^1,\]
for some element $\vartheta\in \Omega^{1*}.$ Here $\Delta_R=(\ )\z\tens (\ )\o$ and $\Delta_L=(\ )\mo\tens (\ )\z$ denote the right and left coaction respectively, and $\<\ ,\ \>$ the evaluation pairing. 

When $H$ is a Hopf algebra, we have of course the notion of left, right and bi-covariant codifferential calculi with respect to left and right actions of $H$. Thus, a first order codifferential structure on $H$ is \textit{bicovariant} clearly means an $H$-Hopf bimodule $\Omega^1$ together with a bimodule map $i:\Omega^1\to H$ such that $\Delta\circ i=(i\tens\id)\Delta_R+ (\id\tens i)\Delta_L.$ We also note that any Hopf algebra $H$ is canonically a right $H$-crossed module in a different way from (\ref{Aplus1}), namely by the right adjoint action and right regular coaction (given by the coproduct). This projects down to a second $H$-crossed module structure on $H^+$,
\begin{equation}\label{Aplus2} g\ra h=Sh\o g h\t,\quad \Delta_R=\Delta-1\tens\id,\quad \forall g\in H^+,\,h\in H.
\end{equation}
First order bicovariant codifferential calculi over $H$ are isomorphic to ones of the form 
\begin{equation}\label{fcodiff}\Omega^1=H\tens\Lambda^1,\quad i:\Lambda^1\to H^+\end{equation}
given by data $(\Lambda^1,i)$  where $\Lambda^1$ is a right $H$-crossed module and $i$ is morphism of right $H$-crossed modules (and extended to $\Omega^1$ as a left $H$-module map). The calculus is coinner if and only if there exists $\vartheta\in \Lambda^{1*}$ such that $i(\eta)=\<\vartheta,\eta\z\>\eta\o-\<\vartheta,\eta\>1_H$ for all $\eta\in\Lambda^1.$

Similarly, a  \textit{codifferential graded coalgebra} on a coalgebra $H$ is defined in~\cite{MaTao1} to be a $\mathbb{N}_0$-graded coalgebra $\Omega=\oplus_{n\ge 0}\Omega^n$ with $\Omega^0=H$ equipped with $i:\Omega\to \Omega$ of degree $-1$ with $i^2=0$ and obeying the super-coderivation property \[ \Delta\circ i(\eta)=(i\otimes\mathrm{id}+(-1)^{|\ |}\otimes i)\circ\Delta\eta,\quad \forall \eta\in\Omega,\] where $(-1)^{|\ |}\eta=(-1)^{|\eta|}\eta.$
One necessarily has $i=0$ when restricted to $H$. We say $(\Omega, i)$ is {\em coinner} if there exists an element $\vartheta\in \Omega^{1*}$ such that
$$i(\eta)=\<\vartheta,\eta\o\>\eta\t+(-1)^{|\eta|} \eta\o\<\vartheta,\eta\t\>$$
for any $\eta\in\Omega.$ Here $\Delta=(\ )\o\tens (\ )\t$ denotes the coproduct of the underlying coalgebra of $\Omega$.
Dual to Definition~\ref{sdga}, a {\em strongly bicovariant codifferential graded algebra}~\cite{MaTao1} on a Hopf algebra $H$ is an $\mathbb{N}_0$-graded super-Hopf algebra extending $H$ in degree 0 and equipped with a degree $-1$, square zero super-derivation and super-coderivation. If $(\Lambda^1, i)$ (or $(\Lambda^1,\vartheta)$) defines a first order bicovariant codifferential calculus on a Hopf algebra $H$ then we have following  constructions of strongly bicovariant codifferential graded algebras on $H$ dual to those in the preceding subsection~\cite[Proposition 4.2, Proposition 4.3]{MaTao1}:
\begin{itemize}
\item[c1)] The `maximal' coinner one $\Omega_\vartheta(H)=H\rbiprod B_\vartheta(\Lambda^1),$ where $B_\vartheta(\Lambda^1)$ is a sub-braided-super Hopf algebra of $\Sh_-(\Lambda^1)$ and $\vartheta\in\Lambda^{1*}$ is such that $\<\vartheta,\eta\ra h\>=\epsilon_H(h)\<\vartheta,\eta\>$ for all $\eta\in\Lambda^1,h\in H.$ In fact, any coinner strongly bicovariant codifferential graded algebra on $H$ with coinner data right-invariant that is `cogenerated by its degree $0$ and $1$' can be embedded into such $\Omega_\vartheta(H)$ as a sub-object;
\item[c2)] The `Woronowicz' coinner one $\Omega_w(H)=H\rbiprod B_-(\Lambda^1)$ subject to the dual conditions of (\ref{wor-con});
\item[c3)] The tensor one $\Omega_{tens}(H)=H\rbiprod T_-\Lambda^1,$ where $T_-\Lambda^1$ is the braided-super tensor algebra of $\Lambda^1$ in the category of right $H$-crossed modules and codifferential is extended from $i$ as a super-derivation. Moreover, $\Omega_{tens}(H)$ is coinner with $\vartheta\in \Lambda^{1*}$ if and only if its first order is coinner with the same $\vartheta\in\Lambda^{1*}$ such that $\<\vartheta,\xi\>\eta=\eta\z\<\vartheta,\xi\ra \eta\o\>$ for any $\xi,\eta\in\Lambda^1.$
\end{itemize}

We will be interested in the self-dual case when we have both a differential and a codifferential structure at the same time. In the  bicovariant case this corresponds to triple $(\Lambda^1,\omega,i)$ where $\omega:A^+\to\Lambda^1$ and $i:\Lambda^1\to A^+$ are morphisms for the appropriate right $A$-crossed module structures on $A^+$ by superposing (\ref{fdiff}) and (\ref{fcodiff}) with $H=A$. We say that the calculus corresponding to $(\Lambda^1,\omega)$ is {\em augmented} by $i$. A strongly bicovariant differential graded algebra is similarly augmented if it admits a strongly bicovariant codifferential structure, etc. In general, the first order data $(\Omega^1,\extd,i)$ do not extend to higher orders automatically. We have \cite[Propositions 4.5 and 4.6]{MaTao1}, 
\begin{itemize}
\item[a1)]  The `universal' inner construction d1) is augmented by $i$ if and only if $\theta\ra i(\theta)$ is graded central in $\Omega_{\theta}(A);$
\item[a2)] For the `Woronowicz' or braided exterior algebra construction d2), an augmentation $i:\Lambda^1\to A^+$ extends to degree two if $\eta\ra i(\zeta)=0$ for all $\sum\eta\tens\zeta\in \ker(\id-\Psi);$
\item[a3)] The `maximal' coinner  $\Omega_{\vartheta}(H)=H\rbiprod B_\vartheta(\Lambda^1)$ in c1) forms an augmented differential calculus by $\omega:H^+\to\Lambda^1$ if and only if 
\begin{align*}
\<\vartheta\tens\vartheta,\xi\z\tens &\tilde{\omega}(\xi\o)\>\xi\t=\<\vartheta\tens\vartheta,\xi\z\tens \tilde{\omega}(\xi\o)\>,\nonumber\\
\<\vartheta\tens\vartheta,\eta_1\z\tens &\tilde{\omega}(\eta_1\o)\> \eta_2\tens\cdots\tens \eta_n\\
&{}\quad=(-1)^{n-1}\eta_1\tens\cdots\tens \eta_{n-1}\<\vartheta\tens\vartheta,\eta_n\z\tens \tilde{\omega}(\eta_n\o)\>,\nonumber
\end{align*}
for all $\xi\in \Lambda^1\subseteq B_{\vartheta}(\Lambda^1)$ and $\eta_1\tens\cdots\tens \eta_n\in B_{\vartheta}(\Lambda^1).$
\end{itemize}

In the finite-dimensional case, it is clear that $(\Omega^1,\extd,i)$ is an augmented first order bicovariant differential calculus over $A$ if and only if $(\Omega^{1*},i^*,\extd^*)$ is augmented over $A^*.$ If $(\Omega,\extd,i)$ is augmented strongly bicovariant differential graded algebra on $\Omega^0=A$, then the graded dual $(\Omega^{gr},i^*,\extd^*)$ is augmented on $A^*$. it should be clear that our differential and codifferential constructions are similarly related by duality provided the relevant Hopf algebras, actions and coactions are related by duality. Details are in \cite[Lemma 4.10, Corollary~4.13]{MaTao1}.

\subsection{Hopf quivers}

A quiver is a quadruple $Q=(Q_0,Q_1,s,t),$ where $Q_0$ is the set of vertices, $Q_1$ is the set of arrows, and $s,t: Q_1 \to Q_0$ are two maps assigning respectively the source and the target for each arrow. A path of length $l \ge 1$ in the quiver $Q$ is a finitely ordered sequence of $l$ arrows $\alpha_1 \cdots \alpha_l$ such that $s(\alpha_{i+1})=t(\alpha_i)$ for $1 \le i \le l-1.$ By convention a vertex is said to be a trivial path of length $0,$ an arrow is a path of length $1,$ and an arrow $\alpha$ is called a loop if $t(\alpha)=s(\alpha).$ We call a quiver is a directed graph or digraph if it has no loops or multiple arrows. A quiver $Q$ is called finite if both $Q_0$ and $Q_1$ are finite sets. Let ${}^{x}Q_1{}^{y}$ denotes the set of all the arrows from $x$ to $y.$ 

The path coalgebra  denoted by $kQ$ is the $k$-space spanned by the paths of $Q$ with comultiplication and counit defined by $\D(x)=x \otimes x,\ \e(x)=1$ for each $x \in Q_0,$ and for each non-trivial path $p=\alpha_1 \cdots \alpha_n,$
\begin{equation*}
\D(p)=s(\alpha_1) \tens p + \sum_{i=1}^{n-1}\alpha_1 \cdots \alpha_{i} \tens
\alpha_{i+1} \cdots \alpha_n + p \tens t(\alpha_n),\quad\e(p)=0.
\end{equation*}
The length of paths gives a natural gradation to the path coalgebra.  Let $Q_n$ denote the set of paths of length $n$ in $Q,$ then $k Q=\oplus_{n \ge 0} k Q_n$ and $\D(k Q_n) \subseteq \oplus_{n=i+j}k Q_i \otimes k Q_j.$ Let ${}^{x}kQ_1{}^{y}=k{}^xQ_1{}^y$ denotes the subspace of  $kQ_1$ spanned by all the arrows from $x$ to $y.$

The path algebra also denoted by $kQ$ when the context is clear has the same underlying vector space as the path coalgebra and the multiplication is defined by concatenation of paths, i.e.
\[pq=\begin{cases}
                         \alpha_1\cdots\alpha_l\beta_1\cdots\beta_m,   &  \text{if } t(\alpha_l)=s(\beta_1), \\
                          0,   &  \text{otherwise,}           
          \end{cases}\]
for paths $p=\alpha_1\cdots\alpha_l$ and $q=\beta_1\cdots\beta_m.$ Hence $k Q$ as an algebra is length-graded associated algebra, as $kQ_i\cdot k Q_j\subseteq kQ_{i+j}$ for any $i,j\ge 0.$ In fact, if $Q$ is a finite quiver, then the graded dual of $kQ$ as the path coalgebra is exactly $k Q$ as the path algebra $kQ$. This justifies our notations. Clearly, $k{}^xQ_1{}^y$ is also identified as its own dual space. For brevity, in view of this, we still use arrows to denote the basis of $k{}^xQ_1{}^y$ when considered in the path algebra.

Finally, a quiver $Q$ is said to be a {\em Hopf quiver} if the corresponding path coalgebra $kQ$ admits a length-graded Hopf algebra structure~\cite{cr2}. Hopf quivers can be determined by ramification datum of groups. 
Let $G$ be a group, $\mathfrak{C}$ the set of
conjugacy classes. A ramification datum $R$ of the group $G$ is a
formal sum $\sum_{C \in \mathfrak{C}}R_CC$ of conjugacy classes with
coefficients in $\mathbb{N}_0=\{0,1,2,\cdots\}.$ The corresponding
Hopf quiver $Q=Q(G,R)$ is defined as follows: the set of vertices
$Q_0$ is $G,$ and for each $x \in G$ and $c \in C$ and each $C\in\mathfrak{C}$ there are $R_C$
arrows $x\longrightarrow xc.$  We will say that a Hopf quiver is {\em coloured} if these arrows have been enumerated $1,\cdots,R_C$ for every $c$ in every $C$. 

For a given Hopf quiver $Q(G, R),$ the isoclasses of graded Hopf structures on $kQ$ is in one-to-one correspondence with the isoclasses of $kG$-Hopf bimodule structures on $kQ_1.$ The graded Hopf structures are obtained from Hopf bimodules via the quantum shuffle product of Rosso \cite{cr2}. Equivalently, a finite quiver $Q$ is a Hopf quiver if and only if the path algebra $kQ$ admits a graded Hopf structure 
and the isoclasses of graded Hopf structures on $kQ$ are in one-to-one correspondence with the isoclasses of $k(G)$-Hopf bimodule structures on $k Q_1.$

We will also need super versions. Thus, a super-quiver is a quiver with $\Z_2=\{\bar{0},\bar{1}\}$-grading on $Q_1$. Naturally, the path coalgebra $kQ$ has a $\Z_2$-grading by assigning $|p|=\sum_{i=1}^n|\alpha_i|$ to any path $p=\alpha_1\alpha_2\cdots\alpha_n$ and hence becomes a super-coalgebra, so we call $kQ$ path super-coalgebra. Then a Hopf super-quiver is a super-quiver whose path super-coalgebra admits a length-graded Hopf super-algebra. It can be given by super ramification datum of groups, which is a pair of formal sums $(R_{\bar{0}}=\sum_{C\in\mathfrak{C}}R_{C,\bar{0}}C,\,R_{\bar{1}}=\sum_{C\in \mathfrak{C}}R_{C,\bar{1}}C)$ of conjugacy classes with non-negative integer coefficients.
The associated Hopf super-quiver $Q(G,R_{\bar{0}},R_{\bar{1}})$ is defined as follows: the set of vertices is $G,$ for each $x\in G$ and $c\in C$, there are $R_{C,\bar{0}}$ even arrows and $R_{C,\bar{1}}$ odd arrows from $x$ to $xc.$ In fact, the isoclasses of graded Hopf super algebra structures on $kQ(G,R_{\bar{0}},R_{\bar{1}})$ are in one-to-one correspondence with the isoclasses of pairs of $k(Q_0)$-Hopf bimodules on $kQ_{1,\bar{0}}$ and $kQ_{1,\bar{1}}.$ For our purpose, we only consider Hopf super-quiver $Q(G, R_{\bar 0}, R_{\bar 1})$ with $R_{\bar 0}=0,$ i.e. all arrows are odd. In this case, the isoclasses of the graded Hopf super-algebra structures on $kQ$ is in one-to-one correspondence with the isoclasses of $k(Q_0)=kG$-Hopf bimodule structures on $kQ_1.$
Equivalently, on a finite group, we consider graded Hopf super-algebra structures on the path algebra $kQ$ with all arrows are odd,  whose isoclasses are equivalent to the isoclasses of $k(Q_0)=k(G)$-Hopf bimodule structures on $kQ_1.$

\section{(Co)differentials on group algebras and group function algebras}

 It is shown in \cite{MaTao1} that given a pair $(\bar Q,Q)$ of a digraph $\bar Q$ contained in a quiver $Q$, with vertex set $X$, there is an inner first order calculus \cite[Example 3.2]{MaTao1}  $\Omega^1(\bar Q, Q)$ on $A=k(X)$ given by 
\begin{equation}\label{quivcalc}\Omega^1=k Q_1=\bigoplus_{x,y\in X}k {}^xQ_1{}^y,\quad \extd=[\theta,\ ]\end{equation}
where $\theta$ is the sum of all arrows in $\bar{Q}$. Thus the space of 1-forms has a basis given by all arrows of the quiver, with left and right module structures given by the source and target maps $s,t$ (so $a.\beta=a(s(\beta))\beta$ and $\beta.a=\beta a(t(\beta))$ for every $a\in k(X)$ and every arrow $\beta$). Moreover, every first order generalised differential calculus on $k(X)$ is isomorphic to such a  `quiver differential calculus' canonical form  (in this direction \cite[Example 2.2]{MaTao1}, any calculus must have $\Omega^1 $ an $X$-bigraded space and be inner; we then define $\bar Q$ as having an arrow $x\to y$ whenever the component $\theta_{x,y}\ne 0$). One can think of the the data $\bar Q$ here as the choice of a distinguished arrow (we will later label it by `1') in $Q$ between any two distinct vertices for which arrows exist. Different choices clearly give isomorphic calculi, so isomorphism classes of first order calculi are given by data $(\bar Q,R)$ where the digraph $\bar Q$ is supplemented  by a non-negative integer  $R_{x,y}=\dim k{}^xQ^y$ with $R_{x,y}\ge 1$ whenever $x\to y$ in $\bar Q$ and zero otherwise. It is also shown \cite[Corollary 2.3]{MaTao1} that this first order calculus extends to an inner differential graded algebra $\Omega_\theta(X)=kQ/J_\theta$ on $k(X),$ where $kQ$ is the path algebra and $J_\theta$ is the graded ideal generated by $\theta^2a-a\theta^2, \theta^2\omega-\omega\theta^2$ for all $a\in A,\omega\in kQ_1$.  Every inner differential graded algebra on $k(X)$ that is generated algebraically by its degree $0,1$ components  is isomorphic to a quotient of this $\Omega_\theta(X)$. 

The dual version of these results apply to a coalgebra $C=kX$ where $\Delta x=x\tens x$, $\epsilon (x)=1$ for all $x\in X$ and $X$ is not assumed to be finite (but is assumed to be nonempty). Then dual to \cite[Example~2.2]{MaTao1},  first order codifferential calculi on $C$ are in one-to-one correspondence with pairs $(\Omega^1,\vartheta),$ where
\begin{itemize}
    \item[1)] $\Omega^1=\oplus_{x,y\in X}{}^x\Omega^1{}^y$ is a $X$-$X$-bigraded vector space, and
    \item[2)] $\vartheta=\sum_{x,y\in X}\vartheta_{x,y}$ is a formal sum of linear functions in the graded dual $\oplus_{x,y\in X} ({}^x \Omega^1 {}^y)^*$ of $\Omega^1$ with components $\vartheta_{x,y}\in ({}^x\Omega^1{}^y)^*$ and $\vartheta_{x,x}=0$ for any $x,y\in X.$
\end{itemize}
and any codifferential calculus is necessarily coinner. Next, to avoid unnecessarily difficulty, we assume that each homogeneous space ${}^x\Omega^1{}^y$ is finite dimensional. Then a linear function $\vartheta_{x,y}$ can be taken as $\theta_{x,y}^*$ for some element $\theta_{x,y}\in {}^x\Omega^1{}^y$ for all $x,y$ in $X$. As result we can classify the first order codifferential calculi on $kX$ by the same data pairs $(\bar Q,R)$ as above except that $\bar Q_0=X$ is not assumed to be finite. Moreover, we can realise such data as we did before as a digraph-quiver pair $\bar Q\subseteq Q$; associated to such a pair  is a first order codifferential calculus on $kX$ given by 
\begin{equation}\label{codiffX}\Omega^1=k Q_1=\oplus_{x,y} {}^x kQ_1 {}^y,\quad i(x\rightarrow y)=\<\vartheta, x\rightarrow y\>y-x \<\vartheta, x\rightarrow y\>\end{equation} with $\vartheta(\alpha)=1$ if $\alpha\in \bar Q_1$ and zero otherwise, extended linearly.  
Every first order codifferential calculus on $kX$ is isomorphic to such a `quiver codifferential calculus' canonical form.

For higher orders, dual to $\Omega_\theta(X)$ above, we have from the same data $\bar Q\subseteq Q$ a coinner codifferential graded coalgebra $(\Omega_\vartheta(kX), i)$, where 
\begin{align*}
\Omega_\vartheta(kX)=\big\{w\in kQ\,|\,\<\vartheta,w\o\>\<\vartheta,w\t\>w\th\tens w\four&=w\o\<\vartheta,w\t\>\<\vartheta,w\th\>\tens w\four\\
&=w\o\tens w\t \<\vartheta,w\th\>\<\vartheta,w\four\>\big\}
\end{align*}
is a subcoalgebra of path coalgebra $kQ$ and  codifferential 
\begin{equation}\label{kX-subpathcoalgebra-i}
i(p)=\<\vartheta,\alpha_1\>\alpha_2\cdots \alpha_n+(-1)^{n}\alpha_1\cdots \alpha_{n-1}\<\vartheta,\alpha_n\>
\end{equation}
for all path $p=\alpha_1 \alpha_2\cdots \alpha_n$ of $Q$ with $\vartheta=\sum_{\alpha\in \bar Q_1}\delta_{\alpha}.$  Every coinner codifferential graded coalgebra on $kX$ that is `cogenerated by its degree $0,1$ components' (i.e. there is a coalgebra embeding from $\Omega^1$ to the cotensor coalgebra $\mathrm{CoT}_{kX}\Omega^1$) is isomorphic to a sub-object of such $(\Omega_\vartheta(kX),i)$ for some digraph-quiver pair $\bar Q\subseteq Q.$ 

\subsection{Differentials on $k(G)$ and codifferentials on $kG$} We now specialise to the case $X=G$ a group, which we take finite in the case of $k(G)$. Our goal is to say more in view of the Hopf algebra theory of Section~2. Again  we denote by $\mathfrak{C}$ the set of all the conjugacy classes of $G.$ If $V$ is a left $G$-module we denote by ${}_GV$ the space of invariant elements. Similarly $V_G$ in the right module case. We let $Z_c\subseteq G$ be the centraliser of any $c\in G$. 

\begin{lemma}\label{kofG-dif-lem}
Let $G$ be a finite group and $A=k(G)$. The data $(\Lambda^1,\omega)$ in (\ref{fdiff}) for a bicovariant calculus is equivalent to:
\begin{itemize}
\item[(a)] $\Lambda^1$ a $G$-graded left $G$-module $\Lambda^1=\oplus_{g\in G}\Lambda^1_g $ s.t. $h\la\Lambda^1_g=\Lambda^1_{hgh^{-1}}$ for any $g,h\in G,$ and 
\item[(b)] a set of pairs $\{(c, \omega_c)\ |\ c\in C,\  \omega_c\in{}_{Z_c}\!\Lambda^1_c\,\}_{C\in\mathfrak{C},C\neq\{e\}}.$
\end{itemize}
Moreover, the calculus here is always inner with  $\theta=\theta_e+\sum_{g\in G\setminus\{e\}}\omega_g\in\Lambda^1$ where $\omega_g=h\la\omega_c$ if $g=hch^{-1}$ for some $h\in G, c\in C$ and any $\theta_e\in\Lambda^1_e$. 
\end{lemma}
\proof It is well known that a vector space is a right $k(G)$-crossed module if and only if it is a left $kG$-crossed module. So the right $k(G)$-crossed module $\Lambda^1$ is equivalent to the data a). Note here $\Lambda^1_g:=\Lambda^1\ra\,\delta_g$ and $\Delta_R(v)=\sum_{h\in G}h\la v\tens \delta_h.$ Also, the right $k(G)$-module map $\omega:k(G)^+\to \Lambda^1$ is uniquely defined by $\{\omega_g=\omega(\delta_g)\in\Lambda^1_g|\,g\in G\setminus\{e\}\}$ since the right-module structure corresponds to the grading. We also need $h\la \omega_g=\omega_{hgh^{-1}}$ for all $h\in G$ for $\omega$ to be a right comodule map where $k(G)^+$ has the right adjoint coaction. This is equivalent to the data stated in b). Indeed, given the collection $\{\omega_g\in \Lambda^1_g\}_{g\ne e}$ such that $h\la \omega_g=\omega_{hgh^{-1}}$ for all $h\in G$, clearly $\omega_g\in {}_{Z_g}\Lambda^1_g$. We can choose an element $c\in C$ and its associated $\omega_c$ for each nontrivial $C\in \mathfrak{C}$. Conversely, suppose we are given the data b) then for any $g\in G\setminus\{e\}$ write $g=hch^{-1}$ for some  $C$ and its chosen $c\in C$ and some $h\in G.$ One can set $\omega_g=h\la\omega_c\in \Lambda^1_g$. This is well-defined because  if also $g=h'ch'^{-1},$ then $h'=hu$ for some $u\in Z_c,$ so $h'\la \omega_c=h\la(u\la \omega_c)=h\la \omega_c.$ 
For $\theta$, because $A=k(G)$ is commutative, we have $\Lambda^1\square A=\Lambda^1_A\tens A$. The elements of $\Lambda^1_A$ are $\eta=\sum_h\eta_h$ such that $\eta_h\ra\delta_g=\epsilon(\delta_g)\eta_h=\delta_{g,e}\eta_h$ for all $h\in G,$ which implies $\Lambda^1_A=\Lambda^1_e$.  So we need $\Delta_R\theta-\theta\tens 1\in \Lambda^1_e\tens k(G)$ as the condition in (\ref{fdiff}) for an inner bicovariant calculus. This means  $\theta=\sum_{g\in G}\theta_g\in\Lambda^1$, where $\theta_g=\theta\ra \delta_g\in\Lambda^1_g$ such that $h\la\theta_g=\theta_{hgh^{-1}}$ for any $h\in G$ and $g\in G\setminus\{e\}.$ Such $\{\theta_g\}_{g\in G\atop g\neq e}$ are the same data as $\omega$, and we also have a free choice of  $\theta_e$.  Note that the data in (b) are equivalent if they define the same map $\omega$, meaning that for  each $C$ we have $(c,\omega_c)\sim (c',\omega'_{c'})$ in the sense that $c'=kck^{-1}$, $\omega'_{c'}=k\la\omega_c$ for some $k\in G$. \endproof

Recall that  a digraph is a {\em Cayley digraph} if it is of  the form $\bar Q=Q(G, \bar C)$ where $Q_0=G$ is a group, $\bar C\subseteq G\setminus\{e\}$ is an ad-stable subset  and the digraph has an arrow $x\to y$ iff $x^{-1}y\in\bar C$.  The set of arrows of a Cayley digraph has canonical and mutually commuting left and right action $h\ast(x\to y)=(xh^{-1}\to y h^{-1})$ and  $(x\to y)\ast h=h^{-1}x\to h^{-1}y$ for all $h\in G$. If $Q(G,R)$ is a Hopf quiver we let $\bar C$ be the sum of conjugacy classes where $R_C\ne 0$ and clearly obtain a Cayley digraph. If $Q(G,R)$ is a coloured Hopf quiver (so the arrows are enumerated  $x\xrightarrow[]{(i)}xc$ where $i=1,\cdots,R_C$ for $c\in C$) then there is a canonical inclusion $\bar Q\subseteq Q$ as the arrows $x\xrightarrow[]{(1)}y$ and moreover $\ast$ from the right extends canonically  $(x \xrightarrow[]{(i)}y)\ast h=h^{-1}x\xrightarrow[]{(i)} h^{-1}y$. We denote these various canonical actions by $\ast$ and we define  a {\em Hopf digraph-quiver triple} $(\bar Q\subseteq Q,\ast)$ to be  coloured Hopf quiver with a left action $\ast$ of $G$ on $kQ_1$ such that
\begin{enumerate}
\item $h\ast k {}^xQ_1{}^y=k {}^{xh^{-1}}Q_1{}^{yh^{-1}}$ for all $h,x,y\in G$.
\item $\ast$ restricts on $\bar Q_1$ to the canonical left action.
\item $\ast$ commutes with the canonical right action on $k Q_1$.
\end{enumerate}

\begin{proposition}\label{kofG-dif-can} Let $(\bar{Q}\subseteq Q,\ast)$ be a Hopf digraph-quiver triple on a finite group $G.$ The associated `quiver calculus' $\Omega^1(\bar Q,Q)$ is bicovariant and inner with
\begin{gather*}
\Omega^1=k Q_1,\quad x\xrightarrow[]{(i)} y.f=x\xrightarrow[]{(i)} y f(y),\quad f.x\xrightarrow[]{(i)} y=f(x)x\xrightarrow[]{(i)} y,\quad \theta=\sum_{x^{-1}y\in \bar{C}}x\xrightarrow[]{(1)} y,\\
\Delta_L(x\xrightarrow[]{(i)} y)=\sum_{h\in G}\delta_h\tens (h^{-1}x\xrightarrow[]{(i)} h^{-1}y),\quad
\Delta_R(x\xrightarrow[]{(i)} y)=\sum_{h\in G}h\ast (x\xrightarrow[]{(i)} y) \tens \delta_h,
\end{gather*}
where $i=1,2,\dots, R_C$ for $x,y$ such that $x^{-1}y\in C.$
\end{proposition}
\proof
To show $(k Q_1,\extd=[\theta,\ ])$ as recalled the start of the section is bicovariant, it suffices to show that $\Delta_L$ and $\Delta_R$ are $k(G)$-bimodule map and $\extd=[\theta,\ ]$ is $k(G)$-bicomodule map. Consider $\Delta_R(\delta_k.(x\xrightarrow[]{(i)} xg))=\delta_{k,x}\Delta_R(x\xrightarrow[]{(i)} xg)=\delta_{k,x}\sum_{h\in G}h\ast(x\xrightarrow[]{(i)} xg)\otimes\delta_h$, then $\Delta_R$ is left module map iff the last expression equals to $\sum_{h\in G}\delta_{kh^{-1}}.h\ast(x\xrightarrow[]{(i)} xg)\otimes\delta_h,$ i.e. $h\ast(x\xrightarrow[]{(i)} xg)$ is a linear combination of arrows in $Q$ starting from $xh^{-1}.$ Similarly, $\Delta_R$ is right module map iff $h\ast (x\xrightarrow[]{(i)} xg)$ is a linear combination of arrows in $Q$ ending with $xgh^{-1}.$ Therefore, $\Delta_R$ is bimodule map iff $h\ast k {}^xQ_1{}^{y}\subseteq k{}^{xh^{-1}}Q_1{}^{yh^{-1}}$, which is the case under our assumptions. Similarly $\Delta_L$ is a bimodule map. Both $\Delta_{L,R}$ are coactions as they correspond to actions of $G$.  Next, $\Delta_R(\extd\delta_x)=\Delta_R(\theta\delta_x-\delta_x\theta)=\Delta_R\sum_{a\in\bar C}\big((xa^{-1}\xrightarrow[]{(1)}x)-(x\xrightarrow[]{(1)}xa)\big)=\sum_{a,h}\big((xa^{-1}h^{-1}\xrightarrow[]{(1)}xh^{-1})-(xh^{-1}\xrightarrow[]{(1)}xah^{-1})\big)\tens\delta_h$. On the other hand, $(\extd\tens\id)\Delta\delta_x=\sum_{h}(\theta\delta_{xh^{-1}}-\delta_{xh^{-1}}\theta)\tens\delta_h=\sum_{a,h}\big((xh^{-1}a^{-1}\xrightarrow[]{(1)}xh^{-1})-(xh^{-1}\xrightarrow[]{(1)}xh^{-1}a)\big)\tens\delta_h$. The two expressions agree after a change of variables $h^{-1}ah\mapsto a$, hence $\extd$ is a right comodule map.  Similarly for $\Delta_L$. 
\endproof

We say two Hopf digraph-quiver triples $(\bar{Q}\subseteq Q,\ast)$ and $(\bar{Q'}\subseteq Q',\ast')$ are isomorphic if the data $R_C$ and $\bar C$ are the same in the two cases and there exists a linear isomorphism $\varphi:k Q_1\to k Q'_1$ that $\varphi(k{}^xQ_1{}^y)=k{}^x{Q'}_1{}^y$ and $\varphi$ intertwines $\ast,\ast'$ and such that $\varphi(x\xrightarrow[]{(1)}xa)=x\xrightarrow[]{(1)'}xa$ for any $h,x,y\in G,\,a\in\bar{C}.$

\begin{theorem}\label{kofG-dif-thm}
Let $G$ be a finite group. Every bicovariant differential calculus on $k(G)$ is isomorphic to a Hopf digraph-quiver calculus of the form in Proposition~\ref{kofG-dif-can}.  
There is a bijection between the set of isomorphism classes of generalised bicovariant differential calculi on $k(G)$ and the set of isomorphism classes of Hopf digraph-quiver triples.
\end{theorem}
\proof
Suppose $(\Omega^1,\extd)$ is any bicovariant differential calculus on $k(G).$ By (\ref{fdiff}) and Lemma~\ref{kofG-dif-lem}, without loss of generality, we can assume $\Omega^1=k(G)\tens \Lambda^1$ for some given data $(\Lambda^1=\oplus_{g\in G}\Lambda^1_g,\theta=\sum_{g\in G\setminus\{e\}}\theta_g)$ with $\theta_g\in\Lambda^1_g$ and $\theta_e=0$, where $\extd (\delta_x)=\sum_{g\in G\setminus\{e\}}\delta_{xg^{-1}}\tens\theta_g$ for any $x\in G.$ 

We construct a Hopf digraph-quiver triple $(\bar Q\subseteq Q,*)$ associated to the data $(\Lambda^1,\theta=\sum_{g\in G\setminus\{e\}}\theta_g)$ in Lemma~\ref{kofG-dif-lem} as follows. Firstly we take $\bar C$ to be the union of nontrivial conjugacy classes where $\theta_g\ne 0$, which defines the Cayley digraph $\bar Q=Q(G,\bar C)$.  Next, for any $C\in \mathfrak{C},$ we take $R_C=\dim(\Lambda^1_g)$ whenever some $g\in C$ and let $Q=Q(G,R=\sum_{C\in\mathfrak{C}}R_C)$. We then enumerate the arrows of $kQ_1$ by index $i=1,2,\dots, R_C$ for each $C\in\mathfrak{C}.$ After that,  we choose a basis $\{e^{(i)}_g\}_{i=1,\cdots,R_C}$ for each $\Lambda^1_g$ such that $e_g^{(1)}=\theta_g$ whenever $\theta_g\neq 0$ (or $g\in\bar C$). Finally, the $k(G)$-bimodule map $\varphi: k(G)\tens\Lambda^1\to k Q_1$ sending $\delta_x e^{(i)}_g$ to $x \xrightarrow[]{(i)} xg$ for all $i=1,\cdots,R_C$ transfers $k(G)$-bicomodule, or equivalently, $G$-bimodule structure to $kQ_1.$ More precisely, the right $G$-action $(\delta_x e_g^{(i)})\ast h=\delta_{h^{-1}x}e_g^{(i)}$ induces the canonical right $G$-action on $kQ_1,$ while the left $G$-action of $\Lambda^1$ induces the left $G$-action on $kQ_1,$ i.e., $h\ast (x\xrightarrow[]{(i)} xg)=h\ast\varphi^{-1}(\delta_x e_g^{(i)})=\varphi^{-1}(h\ast\delta_x e_g^{(i)})=\varphi^{-1}(\delta_{xh^{-1}} h\la e_g^{(i)}).$ In particular, $h\ast (x\xrightarrow[]{(1)}xa)
=xh^{-1}\xrightarrow[]{(1)}xah^{-1},$ as $h\la e_a^{(1)}=h\la \theta_a=\theta_{hah^{-1}}=e_{hah^{-1}}^{(1)}$ for any $h\in G,\,a\in\bar C.$ Therefore we have a Hopf digraph-quiver triple $(\bar Q\subseteq Q,\ast)$ associated to a given $(\Lambda^1,\theta).$ One may then verify all the requirements in detail. 
Clearly $k(G)\tens\Lambda^1$ is isomorphic to Hopf quiver calculus $k Q_1$ in Proposition~\ref{kofG-dif-can} under the $k(G)$-Hopf bimodule isomorphism $\varphi$ above.

If two bicovariant differential calculi on $k(G)$ are isomorphic, then the corresponding data in Lemma~\ref{kofG-dif-lem} are isomorphic, say $(\Lambda^1,\theta)\cong(\Lambda^1{}',\theta'),$ this means there is a $G$-graded left $G$-module isomorphism $\phi:\Lambda^1\to \Lambda^1{}'$ such that $\phi(\theta)-\theta'\in\Lambda^1_A=\Lambda^1_e,$ namely $\phi(\theta)=\theta'$ as $\theta_e=0=\theta'_e$.  In fact, the map $\phi$ is induced by the $k(G)$-Hopf bimodule isomorphism $\tilde{\phi}:k(G)\tens\Lambda^1\to k(G)\tens\Lambda^1{}'$ that sends $\delta_x e_g^{(1)}$ to $\delta_x e_g^{(1)}{}'.$ 
Let $(\bar Q\subseteq Q,\ast)$ and $(\bar Q'\subseteq Q',\ast')$ denote the associated triple to data $(\Lambda^1,\theta)$ and $(\Lambda^1{}',\theta')$ respectively. 
By construction, we know $(\bar Q\subseteq Q,\ast)$ and $(\bar Q'\subseteq Q',\ast')$ have equal data $R=\sum_{C\in \mathfrak{C}}R_C$ and $\bar C.$ 
From previous discussion, there are $k(G)$-Hopf bimodule isomorphisms $\varphi:k(G)\tens\Lambda^1\to k Q_1$ and $\varphi':k(G)\tens\Lambda^1{}'\to k Q'_1$
Then $\psi=\varphi'\circ\tilde{\phi}\circ\varphi^{-1}:k Q_1\to k Q'_1$ is a $k(G)$-Hopf bimodule isomorphism and sends inner data to inner data. 
In particular, $\psi$ is a left $G$-module isomorphism and 
$\psi(x\xrightarrow[]{(1)}xa)=\varphi'\circ\phi(\delta_{x} e_a^{(1)})
=\varphi'(\delta_{x}\phi(e_a^{(1)}))
=\varphi'(\delta_x e_a^{(1)}{}'))
=x\xrightarrow[]{(1)'} xa$ for any $x\in G,\,a\in \bar C.$
\endproof

Note that the association of a generalised differential calculus to a Hopf quiver calculus depends on the choice of the basis element of each $\Lambda^1_g.$ But different choices lead to isomorphic differential calculi. Also, if $\Omega^1=k Q_1$ is a bicovariant calculus of  the `digraph-quiver' form  the right $G$-action on arrows of $k Q_1$ may not be a canonical $\ast$. However, our theorem tells it is isomorphic to a canonical one. In particular, the vectors $(e\xrightarrow[]{(i)}x^{-1}y)\ast x^{-1}$ provide a basis of $k{}^xQ_1{}^y$ including $x\xrightarrow[]{(1)}y.$ 
There is a linear transformation that sends that basis of $k{}^x Q_1 {}^y$ to $\{x\xrightarrow[]{(i)} y\}$ our labeled basis.
These linear maps together constitute a $G$-bimodule isomorphism on $kQ_1$ that respects $\theta,$ hence provides the isomorphism to a canonical Hopf quiver calculus in Proposition~\ref{kofG-dif-can}.

The dual theory for codifferentials on $kG$ is strictly parallel except that we do not need $G$ to be finite. In view of this we omit some details as proofs as essentially the arrow-reversal of the ones already given. Recall that the dual  $V^*$ of a right $G$-module $V$ is naturally a left $G$-module via $\<h\la f,v\>=\<f,v\ra h\>$ for any $f\in V^*, v\in V, h\in G.$

\begin{lemma}\label{kG-codif-lem}
Let $G$ be a group and $kG,$ the group Hopf algebra. The first order bicovariant codifferential calculus data $(\Lambda^1,i)$ in (\ref{fcodiff}) is equivalent to 
\begin{itemize}
     \item[(a)] $\Lambda^1$ a $G$-graded right $G$-module $\Lambda^1=\oplus_{g\in G}\Lambda^1_g$ such that $\Lambda^1_g\ra h=\Lambda^1_{h^{-1}gh}$ for any $g,h\in G,$
     \item[(b)] a set of pairs $\{(c,\iota_c)|\, c\in C,\, \iota_c\in {}_{Z_c}(\Lambda^1_c)^*\}_{C\in\mathfrak{C},\,C\neq \{e\}}.$
\end{itemize}
The calculus is always coinner where one can take $\vartheta=\iota_e+\sum_{g\in G\setminus\{e\}} \iota_g\in (\Lambda^1)^*$ with $\iota_g=h\la\iota_c$ whenever $g=hch^{-1}$ for some $h\in G,\,c\in C$ and arbitrary $\iota_e\in(\Lambda^1_e)^*.$
\end{lemma}
\proof It is easy to see that right $kG$-crossed module $\Lambda^1$ corresponds to (a). For the right $G$-crossed module map $i:\Lambda^1\to (kG)^+,$ we write $i(\eta)=\sum_{h\neq e} l(\eta)_h (h-e).$ Suppose $\eta\in\Lambda^1_g$ then that $i$ being a right $G$-comodule map requires $i(\eta)\tens g=\sum_{h\neq e} l(\eta)_h (h-e)\tens h$ or $\sum_{h\neq e}l(\eta)_h(h-e)\tens (g-h)=0,$ this implies $i(\eta)=l(\eta)_g(g-e),$ namely
$i(\eta)=\<\iota_g,\eta\> (g-e)$
for some $\iota_g\in (\Lambda^1_g)^*.$ In particular $\iota_e\in(\Lambda^1_e)^*$ is arbitrary. Hence in general,
\[i(\eta)=\sum_g\<\iota_g,\eta\>(g-e)\]
as $\iota_g\in (\Lambda^1_g)^*$ picks out the $\Lambda^1_g$ component of $\eta.$ The module map property $i(\eta\ra h)=h^{-1} i(\eta) h$ requires $h\la \iota_g=\iota_{hgh^{-1}}$ for any $g,h\in G.$ As in the proof of Lemma~\ref{kofG-dif-lem}, the set $\{\iota_g\}_{g\in G}$ is specified by a subset of values, one for each nontrivial conjugacy class as stated in (b). Let $\vartheta=\sum_{g\in G}\iota_g\in (\Lambda^1)^*$ be the formal sum of linear functions $\iota_g,$ it is an element in the graded dual of $\Lambda^1=\oplus_{g\in G}\Lambda^1_g$. It is clear that $i(\eta)=\<\vartheta,\eta\z\>\eta\o-\<\vartheta,\eta\>e,$ so this codifferential calculus is coinner by (\ref{fdiff}).
\endproof

Here the data $(\Lambda^1,\{(c,\iota_c)\})$ on $kG$ in Lemma~\ref{kG-codif-lem} is equivalent to the data $(\Lambda^{1*},\{(c,\iota^*_c)\})$ on $k(G)$ in Lemma~\ref{kofG-dif-lem} when the group $G$ is finite.  Being dual spaces we canonically have a left-right reversal of group actions. Similarly, let $\bar Q$ be a Cayley graph in form $\bar Q=Q(G,\bar C)$ for some ad-stable subset $\bar C$ of $G$ and $(\bar Q\subseteq Q,\ast)$ a Hopf digraph-quiver triple as above. This time we consider that $G$ acts from the right by $\cdot h=h^{-1}\ast$ (which now restricts to a canonical right action on $\bar Q$) and we remind ourselves by saying that $(\bar Q\subseteq Q,\cdot)$ is right handed. %

\begin{proposition}\label{kG-codif-can}
Let $(\bar{Q}\subseteq Q,\cdot)$ be a Hopf digraph-quiver triple on group $G$ as before, but viewed acting from the right. The associated `quiver codifferential calculus' is bicovariant with
\begin{gather*}
\Omega^1=kQ_1,\quad\Delta_L(x\xrightarrow[]{(i)}y)=x\tens (x\xrightarrow[]{(i)}y),\quad\Delta_R(x\xrightarrow[]{(i)}y)=(x\xrightarrow[]{(i)}y)\tens y\\
g\cdot (x\xrightarrow[]{(i)}y)=gx\xrightarrow[]{(i)} gy,\quad (x\xrightarrow[]{(i)}y)\cdot g= \sum_j \mu_{ij}(x^{-1}y,g)(xg\xrightarrow[]{(j)}yg),
\end{gather*}
with codifferential 
\begin{displaymath} i(x\xrightarrow[]{(i)} y)= \left\{
\begin{array}{ll}
y-x, & \textrm{if $i=1$ and $x^{-1}y\in \bar C$,}\\
0,    & \textrm{otherwise.}
\end{array} \right.
\end{displaymath}
and $\mu_{ij}(c,g)$ are the coefficients of the right action $\cdot$, i.e. $(e\xrightarrow[]{(i)}c)\cdot g=\sum_j\mu_{ij}(c,g)g\xrightarrow[]{(j)}cg.$
\end{proposition}
\proof It is easy to see that under the stated $G$-actions, $kQ_1$ is $kG$-bimodule and $\Delta_L$ and $\Delta_R$ are $kG$-bimodule maps, so $kQ_1$ is $kG$-Hopf bimodule. It suffices to show that $i$ is $kG$-bimodule map. Indeed, $i(g\cdot x\xrightarrow[]{(i)}y)=i(gx\xrightarrow[]{(i)}gy)=gy-gx=g(y-x)=g i(x\xrightarrow[]{(i)}y)$ and $i(x\xrightarrow[]{(i)}y\cdot g)=i(xg\xrightarrow[]{(i)}yg)=yg-xg=(y-x)g=i(x\xrightarrow[]{(i)}y)g$ if $i=1$ and $x^{-1}y\in\bar C.$ Otherwise, $i(g\cdot x\xrightarrow[]{(i)}y)=0=g\cdot i(x\xrightarrow[]{(i)}y)$ and $i(x\xrightarrow[]{(i)}y\cdot g)=0=i(x\xrightarrow[]{(i)}y)\cdot g.$
\endproof

Then analogous to Theorem~\ref{kofG-dif-thm}, we have 

\begin{theorem}\label{kG-codif-thm}
Let $G$ be a group. 
Every first order bicovariant codifferential calculus over $kG$ is isomorphic to one of the canonical forms in Proposition~\ref{kG-codif-can} for some right-handed Hopf digraph-quiver triple. There is a bijection between the set of isomorphism classes of bicovariant codifferential calculi over $kG$ and the set of isomorphism classes of right-handed Hopf digraph-quiver triples.
\end{theorem}
\proof Suppose $(\Omega^1,i)$ is a bicovariant codifferential calculus over $kG$ which, without loss of generality, we take as $\Omega^1=kG\tens\Lambda^1$ for some data $(\Lambda^1,\vartheta=\sum_g\iota_g)$ with $\iota_e=0$ in Lemma~\ref{kG-codif-lem}, from (\ref{fdiff}). We construct a right-handed Hopf digraph-quiver triple $(\bar Q\subseteq Q,\cdot)$ associated to such data as follows. Firstly we take $\bar C$ to be the union of nontrivial conjugacy classes where $\iota_g\neq 0$, take $R_C=\dim\Lambda^1_g$ for some $g\in G$ for each conjugacy class $C$, and let $\bar Q=Q(G,\bar C),\, Q=Q(G,R).$ We colour $Q$ and embed $\bar Q\subseteq Q$ by marking all arrows in $\bar Q$ with $1$. We then choose a basis $\{e_g^{(i)}\}$ for each $\Lambda^1_g$ such that $(e_g^{(1)})^*=\iota_g$ whenever $\iota_g$ is nonzero. Similar to the proof of Theorem~\ref{kofG-dif-thm}, the $kG$-bicomodule map $\varphi:kG\tens\Lambda^1\to kQ_1$ sending $x\tens e_g^{(i)}$ to $x\xrightarrow[]{(i)} xg$ transfers $kG$-bimodule structure on $kG\tens\Lambda^1$ (or right $G$-module structure on $\Lambda^1$) to a $kG$-bimodule structure on $kQ_1.$ Precisely, one can see the left $G$-action is the canonical one and $(x\xrightarrow[]{(i)}xg) \cdot h$ is defined to be $\varphi^{-1}(xh\tens e_g^{(i)}\ra h).$ In particular, $(x\xrightarrow[]{(1)}xa) \cdot h=xh\xrightarrow[]{(1)}xah$ as $e_a^{(1)}\ra h=e_{h^{-1}ah}^{(1)}$ for any $a\in\bar C.$ One may then verify all the requirements on $(\bar Q\subseteq Q,\cdot)$ being a desired triple and see that calculus $kG\tens\Lambda^1$ is isomorphic to quiver codifferential calculus $kQ_1$ under the map $\varphi$ above.

If two bicovariant codifferential calculi on $kG$ are isomorphic, then the corresponding data in Lemma~\ref{kG-codif-lem} are isomorphic, say $(\Lambda^1,\vartheta)\cong(\Lambda^1{}',\vartheta'),$ this means there is a $G$-graded right $G$-module isomorphism $\phi:\Lambda^1\to \Lambda^1{}'$ such that $\vartheta=\vartheta'\circ \phi.$  In fact, the map $\phi$ is induced by the $kG$-Hopf bimodule isomorphism $\tilde{\phi}:kG\tens\Lambda^1\to kG\tens\Lambda^1{}'$ that sends $x\tens e_g^{(1)}$ to $x\tens e_g^{(1)}{}'.$ 
Let $(\bar Q\subseteq Q,\cdot)$ and $(\bar Q'\subseteq Q',\cdot')$ denote the associated triple to data $(\Lambda^1,\vartheta)$ and $(\Lambda^1{}',\vartheta')$ respectively. 
By construction, we know those triples have equal data $R=\sum_{C\in \mathfrak{C}}R_C$ and $\bar C.$ 
Also, the previous discussion shows that there are $kG$-Hopf bimodule isomorphisms $\varphi:kG\tens\Lambda^1\to kQ_1$ and $\varphi':kG\tens\Lambda^1{}'\to kQ{}'_1$
Then $\psi=\varphi'\circ\tilde{\phi}\circ\varphi^{-1}:kQ_1\to kQ{}'_1$ is a $kG$-Hopf bimodule isomorphism and sends coinner data to coinner data.  
In particular, $\psi$ is a left $G$-module isomorphism and 
$\psi(x\xrightarrow[]{(1)}xa)=\varphi'\circ\phi(x\tens e_a^{(1)})
=\varphi'(x\tens\phi(e_a^{(1)}))
=\varphi'(x\tens e_a^{(1')})
=x\xrightarrow[]{(1')} xa$ for any $x\in G,\,a\in \bar C.$ This shows that the associated triples are isomorphic.
\endproof

The association of a bicovariant codifferential calculus to a  Hopf digraph-quiver depends on the choice of the basis of each $\Lambda^1_g$, but different choices lead to isomorphic triples. If $kQ_1$ is a bicovariant codifferential calculus over $kG$ that is of quiver form. Though the right $G$-action on arrows of $kQ_1$ may not be the canonical one, our theorem tells that it is isomorphic to a canonical one. In particular, the vectors $x\cdot (e\xrightarrow[]{(i)} x^{-1}y)$ provide a basis of ${}^x kQ_1 {}^y$ including $x\xrightarrow[]{(1)}y,$ this means there is a linear isomorphism sends that basis of ${}^x kQ_1 {}^y$ to some $\{x\xrightarrow[]{(i)'} y\}$ basis.

\subsection{Differentials on $kG$ and codifferentials on $k(G)$}

In this section we focus on differentials on $kG$, and codifferentials on $k(G)$ when $G$ is finite. 

\begin{lemma}\label{kG-dif-lem} Let $G$ be a group and $A=kG$ the group algebra. The data $(\Lambda^1,\omega)$ in (\ref{fdiff}) for a bicovariant calculus is equivalent to 
\begin{itemize}
\item[(a)] $\Lambda^1$ is $G$-graded $\Lambda^1=\oplus_{g\in G}\Lambda^1_g$ and such that $\Lambda^1_g\ra h=\Lambda^1_{h^{-1}gh}$ for any $g,h\in G$ and
\item[(b)] $\zeta\in Z^1(G,\Lambda^1_e)$ a group cocycle in the sense $\zeta_{gh}=\zeta_g\ra h+\zeta_h$
\end{itemize}
which is inner when $\zeta$ is exact, i.e.. $\zeta_g=\theta\ra g-\theta$. Inner calculi are given by $[\theta]\in \Lambda^1_e/(\Lambda^1_e)_G.$ If the group $G$ is finite and $|G|$ is invertible in $k$, then any bicovariant calculus $\Omega^1$ on $kG$ is inner with $\theta=-|G|^{-1}\sum_{g\in G}\zeta_g$.
\end{lemma}
\proof It was already shown in \cite[Proposition~2.11]{MaTao1} that a left covariant calculus $\Omega^1(kG)$ corresponds to a right $G$-module $\Lambda^1$ and $\zeta\in Z^1(G,\Lambda^1)$ where  $\zeta_g:=\omega(g-e).$  The additional information for the bicovariant case is that $\Lambda^1$ is a crossed moduele and $\zeta$ has its image in $\Lambda^1_e$ for $\omega$ to be a $kG$-module map (because the crossed module coaction on $(kG)^+$ is trivial as the Hopf algebra is cocommutative.  From the discussion after (\ref{fdiff}), we know the data for an inner calculus is $\theta=\sum_{h\in G} \theta_h$ where $\theta_h\in \Lambda^1_h$ such that $\theta_h=\theta_{ghg^{-1}}\ra g$ for all $g\in G,\,h\in G\setminus\{e\}$ for the bicovariance condition. This means that it is given by a set of pairs $\{(c,\theta_c)|\,c\in C,\,\theta_c\in (\Lambda^1_c)_{Z_c}\}_{C\in\mathfrak{C},\,C\neq \{e\}}$ and a free value $\theta_e.$  However, any two inner data correspond to the same derivation, say  $\extd g=[\theta, g]=[\theta',g]$ for any $g\in G,$ if and only if $\theta-\theta'\in \Lambda^1_G.$ So we can take $\theta\in\Lambda^1_e$ for an inner  calculus, as $\sum_{h\ne e}\theta_h\in\Lambda^1_G$ by a change of variables in the sum so that the relevant data is $[\theta]\in \Lambda^1_e/(\Lambda^1_e)_G.$ Finally,  if the group $G$ is finite and $|G|$ is invertible in $k,$ we have $\zeta_g=\theta\ra g-\theta$ for the stated data $\theta$ due to the cocycle condition.   It is interesting to note that after a left-right reversal we have identical data $(\Lambda^1,\theta)$ for an inner bicovariant calculus on $k G$ and $k(G)$ according to the proof here and Lemma~\ref{kofG-dif-lem} except that now only the part in $\Lambda^1_e$ is significant while before only the part not in $\Lambda^1_e$ was significant. \endproof

From Lemma~\ref{kG-dif-lem}, we know that two inner generalised bicovariant differential calculi are isomorphic if and only if there is a $G$-graded right $G$-module isomorphism $\varphi:\Lambda^1\to \Lambda^1{}'$ between the corresponding data $(\Lambda^1,[\theta])$ and $(\Lambda^1{}',[\theta'])$ such that $\varphi(\theta)-\theta'\in (\Lambda^1_e)_G.$ We are now ready to interpret the inner generalised bicovariant differential calculi on $kG$ in terms of Hopf quivers. Let $Q=Q(G,R)$ be a coloured Hopf quiver. This time, we consider $kQ_1$ as having a canonical left $G$-action, namely $h\ast x\xrightarrow[]{(i)} y=hx\xrightarrow[]{(i)} hy$ for any $h\in G$ where $x\xrightarrow[]{(i)}y$ runs all arrows in of $Q_1,$
and define a right $G$-action $\ast$ on $kQ_1$ such that
\begin{enumerate}
    \item ${}^xkQ_1{}^y\ast h={}^{xh}kQ_1{}^{yh}$ for all $h,x,y\in G$;
    \item $\ast$ commutes with the canonical left $G$-action on $kQ_1$.
\end{enumerate} 
Note that the right $G$-action $\ast$ together with canonical left $G$-action on $kQ_1$ defines a right $G$-action $\ra$ on ${}^e kQ_1{}^e$ by conjugation
\[(e\xrightarrow[]{(i)} e)\,\ra h=h^{-1}\ast (e\xrightarrow[]{(i)} e)\ast h,\,i=1,\dots,R_{\{e\}},\,\forall\,h\in G.\]

\begin{proposition}\label{kG-dif-can}
Associated to data $(Q=Q(G,R),\ast)$ defined above on a group $G$ is an inner bicovariant differential calculus on $kG.$ Its structure is
\begin{gather*}
\Omega^1=kQ_1,\quad \Delta_L(x\xrightarrow[]{(i)}y)=x\tens (x\xrightarrow[]{(i)}y),\quad \Delta_R(x\xrightarrow[]{(i)}y)=(x\xrightarrow[]{(i)}y)\tens y,\\
h\ast (x\xrightarrow[]{(i)}y)= hx\xrightarrow[]{(i)} hy,\quad (x\xrightarrow[]{(i)}y)\ast h\in {}^{xh} kQ_1 {}^{yh},\\
\extd g=g\ast \left((e\xrightarrow[]{(1)} e)\ra g\right)- (g\xrightarrow[]{(1)}g),\textit{ i.e., }\extd g=[\theta,g]\textit{ with }\theta= e\xrightarrow[]{(1)}e,
\end{gather*}
for any $i=1,2,\dots,R_C$ whenever $R_C\neq 0$ and  $x^{-1}y\in C$ for some $C\in\mathfrak{C}$ and for all $g,h\in G.$
\end{proposition}
\proof It is obviously that $kQ_1$ is a $kG$-bimodule,  $\Delta_L,\Delta_R$ make $kQ_1$ a $kG$-bicomodule and $\Delta_L,\Delta_R$ are left $kG$-module maps. It remains to show that $\Delta_L,\Delta_R$ are right $kG$-module maps and $\extd$ is $G$-bicomodule map. Indeed, $\Delta_L,\Delta_R$ are right $kG$-module map if and only if ${}^xkQ_1{}^y\ast h={}^{xh}kQ_1{}^{yh}$ for all $h,x,y\in G$. As $\extd g=g\ast \left((e\xrightarrow[]{(1)} e)\ra g\right)- g\xrightarrow[]{(1)}g$ and $g\ast \left((e\xrightarrow[]{(1)} e)\ra g\right) \in {}^g kQ_1 {}^g,$ we know that $\Delta_L\extd g=g\tens g\ast \left((e\xrightarrow[]{(1)} e)\ra g\right)- g\tens (g\xrightarrow[]{(1)}g)$ $=g\tens\extd g$ and $\Delta_R\extd g=g\ast \left((e\xrightarrow[]{(1)} e)\ra g\right)\tens g- (g\xrightarrow[]{(1)}g)\tens g=\extd g\tens g$ for any $g\in G.$ This means $\extd=[\theta,\ ]$ is $kG$-bicomodule map.\endproof

We say two data pairs $(Q=Q(G,R),\ast)$ and $(Q'=Q(G,R'),\ast')$ are isomorphic if $R=R'$ and there exists a linear isomorphism $\varphi:kQ_1\to kQ'_1$ that $\varphi({}^x kQ_1 {}^y)={}^x kQ'_1 {}^y$ and $\varphi$ intertwines $\ast,\ast'$ and such that
$\varphi \left((e\xrightarrow[]{(1)} e)\ra (h-e)\right)=(e\xrightarrow[]{(1)'} e)\ra'\, (h-e)$ for any $h\in G.$

\begin{theorem}\label{kG-dif-thm}
Let $G$ be a group. Any inner  bicovariant differential calculus on $kG$ is isomorphic to a canonical form in Proposition~\ref{kG-dif-can} for some data $(Q=Q(G,R),\ast)$. There is a bijection between the set of isomorphism classes inner  bicovariant differential calculi and the set of isomorphism classes of data $(Q=Q(G,R),\ast).$
\end{theorem}
\proof Without loss of generality, we can assume any inner bicovariant differential calculus on $kG$ is in form $\Omega^1=kG\tens \Lambda^1$ for data $(\Lambda^1,[\theta])$ given in Lemma~\ref{kG-dif-lem}, i.e. $\Lambda^1=\oplus_{g\in G}\Lambda^1_g$ a $G$-graded right $G$-module such that $\Lambda^1_g\ra h=\Lambda^1_{h^{-1}gh}$ for any $g,h\in G$ and $[\theta]\in \Lambda^1_e/(\Lambda^1_e)_G.$ 

We now construct a data pair $(Q=Q(G,R),\ast)$ from $(\Lambda^1,[\theta])$ as follows. Firstly, we take $R=\sum_{C\in \mathfrak{C}}R_C C$ with $R_C=\dim \Lambda^1_g$ for some $g\in C,\,C\in\mathfrak{C}$ and define coloured Hopf quiver $Q=Q(G,R=\sum_{C\in \mathfrak{C}}R_CC).$ The data $R_C$ are well-defined as the dimensions of $\Lambda^1_g$'s are constant on each conjugacy class. 
Secondly, for each $g\in C,\,C\in \mathfrak{C}$ with $R_C\neq 0,$ we choose a basis $\{f_g^{(i)}\}_{i=1}^{R_C}$ for $\Lambda^1_g.$ In particular, the basis $\{f_e^{(i)}\}_{i=1}^{R_{\{e\}}}$ of $\Lambda^1_e$ is chosen in the way that $f_e^{(1)}=\theta,$ any fixed representative for $[\theta]\in \Lambda^1_e/(\Lambda^1_e)_G.$ Thirdly, noting that choice of basis elements provides a $kG$-bicomodule isomorphism $\varphi:kG\tens\Lambda^1\to kQ_1$ by sending $x\tens f_g^{(i)}$ to arrow $x\xrightarrow[]{(i)}xg,$ the map $\varphi$ then transfers the $kG$-bimodule structure on $\Omega^1=kG\tens\Lambda^1$ to $kQ_1.$ More precisely,  the left $G$-action $h.(x\tens f_g^{(i)})=hx\tens f_g^{(i)}$ determines the canonical left $G$-action on $kQ_1$, i.e., $h\ast (x\xrightarrow[]{(i)}xg)=hx\xrightarrow[]{(i)}hxg,$ while the right $G$-action on $kG\tens\Lambda^1$, or equivalently, the right $G$-action $\ra$ on $\Lambda^1$ determines the right $G$-action on $kQ_1$ by
$(x\xrightarrow[]{(i)}xg)\ast h=\varphi^{-1}(xh\tens (f_g^{(i)}\ra h)).$ Therefore one can verify that we now have a desired data pair $(Q=Q(G,R),\ast)$. 
Clearly one can check that $\Omega^1=kG\tens \Lambda^1$ is isomorphic to the inner generalised bicovariant differential calculus $kQ_1$ associated to this data $(Q=Q(G,R),\ast)$ in Proposition~\ref{kG-dif-can} under the map $\varphi.$
The bijection on isoclasses is clear as two data are defined to be isomorphic in the aim of the associated calculi are isomorphic. 
\endproof

The association of an inner differential calculus on $kG$ to a Hopf quiver calculus in Theorem~\ref{kG-dif-thm}  depends on the choice of the basis elements of each $\Lambda^1_g,$ hence it is not in a unique way. However, different choices lead to isomorphic differential calculi. And as we remarked before, the standard sub-calculus $\bar\Omega^1$ of a generalised differental calculus $\Omega^1$ in Proposition~\ref{kG-dif-can} is not always inner, hence may not be associated to a coloured Hopf quiver. The following  general construction on a finite group algebra provides a canonical example.

\begin{lemma}\label{CG} Let $G$ be any group and $V$  any representation, with structure map  $\rho: G\to \End(V)$ extended linearly to $\kk G$. Let $\Lambda^1=\End(V)$ be a right-module by right multiplication via $\rho$. Then $\zeta(x)=\rho(x)-\id$ is a cocycle and we have an inner calculus with $\theta=\id\in\Lambda^1$. 
\end{lemma}
\proof Explicitly, $\omega.x=x\omega\rho(x)$ and $\extd x=x(\rho(x)-\id)$ for all $x\in G$ and $\omega\in \Lambda^1$. If we write $e_x=x^{-1}\extd x=\zeta(x)$ then the commutation relations among these elements are $e_x.y=y.(e_{xy}-e_y)$ for all $x,y\in G$. Note that the $\{e_x\}$ are not typically linearly independent nor do they necessarily span. Also note that $e_e=0$. There is only one graded component $\Lambda^1_e=\End(V)$ which is closed under $\ra$ given by $\omega\ra h=\omega\rho(h)$ which ensures by Lemma~\ref{kG-dif-lem} that the calculus is bicovariant.   \endproof

Here the Hopf quiver has vertices $G$ and $\dim(V)^2$-squared loops on each vertex according to a basis of $\End(V)$ and $*$ given by $\rho$. The calculus is standard whenever the action of $kG^+$ generates all arrows from $\theta=(e\xrightarrow[]{(1)}e)$.  The calculus is also covariant under conjugation by $G$ acting on either side of $\End(V)$ via $\rho$. 

We now consider codifferentials on $k(G)$ with $G$ finite. 

\begin{lemma}\label{kofG-codif-lem}
Let $G$ be a finite group and $C=k(G)$. Bicovariant codifferential calculi data $(\Lambda^1,i)$ in (\ref{fcodiff}) are equivalent to \begin{itemize}
     \item[(a)] $\Lambda^1$ a $G$-graded left $G$-module $\Lambda^1=\oplus_{g\in G}\Lambda^1_g$ such that $h\la\Lambda^1_g=\Lambda^1_{hgh^{-1}};$
     \item[(b)] a group cocycle $\zeta\in Z^1(G,(\Lambda^1_e)^*)$ 
in the sense $\zeta_{gh}=\zeta_g\ra h+\zeta_h$ for all $g,h\in G.$  
\end{itemize}
Moreover, the codifferential calculus is coinner if $\zeta$ is a coboundary, such calculi corresponding to pairs $(\Lambda^1,[\vartheta])$ where $[\vartheta]\in (\Lambda^1_e)^*/ (\Lambda^1_e)^*_G.$  If $|G|$ is invertible in $k,$  then any bicovariant codifferential calculus on $k(G)$ is coinner.
\end{lemma}
\proof
1) It is clear that $\Lambda^1$ being right $k(G)$-crossed module is equivalent to $\Lambda^1$ being a left $G$-crossed module, which is the data in a). In particular, $\Lambda^1_g=\Lambda^1\ra\delta_g$ and $\Delta_R(\eta)=\sum_{h\in G}h\la\eta\tens\delta_h.$ For b), $\Lambda^{1*}$ is canonically a $G$-graded right $G$-module via $\<f\ra h,\eta\>=\<f,h\la\eta\>$ for all $h\in G,\,\eta\in\Lambda^1,\,f\in \Lambda^{1*}$ and $(\Lambda^1_e)^*$ is clearly a submodule of $\Lambda^{1*}.$ For the codifferential structure $i:\Lambda^1\to k(G)^+,$ we let
\[ i(\eta)=\sum_{g\in G}\delta_g\<\zeta_g,\eta\>,\] where $\zeta_g\in \Lambda^{1*}$ and $\zeta_e=0$. That this is a module map means $i(\eta_h)=\delta_{h,e}i(\eta)$ for all $h\in G$ as the right $k(G)$-action on $k(G)^+$ is $\delta_g\ra\delta_h=\delta_{h,e}\delta_g$ for any $g,h\in G,\,g\neq e,$ from which we deduce that $\zeta_g\in (\Lambda^1_e)^*$. That $i$ is a comodule map means $i(h\la \eta)=i(\eta)(\  h)-1i(\eta)(h)$ for all $h\in G$, which means $\sum_{g\in G}\delta_g\<\zeta_g\ra h,\eta\>=\sum_{g\in G}\delta_{gh^{-1}}\<\zeta_g,\eta\>-\<\zeta_h,\eta\>$ which after a change of variables is the condition stated in b). The cocycle condition stated entails that $\zeta_e=0$. 

2) From (\ref{fcodiff}), a codifferential $i$ being coinner means $i(\eta)=\<\vartheta,\eta\z\>\eta\o-\<\vartheta,\eta\>1_{k(G)}=\sum_{h\in G} \<\vartheta,h\la\eta\>\delta_h-\sum_{h\in G}\<\vartheta,\eta\>\delta_h=\sum_{h\in G} \<\vartheta, (h-e)\la\eta\>\delta_h$ for some $\vartheta\in \Lambda^{1*}.$ 
This shows that $\zeta_g=\vartheta\ra (g-e)$ for some $\vartheta\in \Lambda^{1*}.$  Note that $\vartheta,\vartheta'\in \Lambda^{1*}$ define the same codifferential $i$ iff $\vartheta-\vartheta'\in (\Lambda^{1*})_G$. For a bicovariant codifferential, the corresponding $\zeta_g\in (\Lambda^1_e)^*$ for any $g\in G.$ Putting these two facts together, we know only component of $\vartheta$ in $(\Lambda^1_e)^*$ matters and $[\vartheta]\in (\Lambda^1_e)^*/ (\Lambda^1_e)^*_G.$

3) Check that $i(\eta)=\<\vartheta,\eta\z\>\eta\o-\<\vartheta,\eta\>1_{k(G)}$ with $\vartheta=-|G|^{-1}\sum_{h\in G}\zeta_h\in(\Lambda^1)^*_e$ agrees with $i(\eta)=\sum_{g\in G}\delta_g \<\zeta_g,\eta\>.$ \endproof

As an example, the dual to Lemma~\ref{CG} gives us a bicovariant coinner codifferential calculus on $k(G)$ associated to every finite-dimensional representaion $\rho$. 
Moreover, similarly to the differential side of $k(G)$, we can illustrate coinner bicovariant codifferential calculi on $k(G)$ by Hopf quivers. Let $(Q=Q(G,R),\cdot)$ be a coloured Hopf quiver together with a left $G$-action $\cdot$ on $kQ_1$ such that
\begin{enumerate}
    \item  $h\cdot k({}^x Q_1 {}^y)=k({}^{xh^{-1}} Q_1 {}^{yh^{-1}})$ for all $h,x,y\in G,$
    \item  $\cdot$ commutes with the canonical right $G$-action on $kQ_1,$ where $(x\xrightarrow[]{(i)} y)\cdot g=g^{-1}x\xrightarrow[]{(i)} g^{-1}y$ for any $g\in G.$
\end{enumerate}
Clearly, these are the exactly the same data (1),(3) for a Hopf digraph-quiver triple defined in Section~3.1. For all $h\in G$ and $g$ in some conjugacy class $C$ with ramification $R_C$, denote $h\cdot (e\xrightarrow[]{(i)} g)=\sum_j\lambda_{ij}(h,g)\, (h^{-1}\xrightarrow[]{(j)} gh^{-1})$ by some coefficients $\lambda_{ij}(h,g)$ in $k$ with $i,j=1,2,\dots,R_C.$ Since left $G$-action commutes with right canonical right $G$-action, we have \[h\cdot (x\xrightarrow[]{(i)} y)=\sum_j\lambda_{ij}(h,x^{-1}y)\, (xh^{-1}\xrightarrow[]{(j)} yh^{-1}).\]

\begin{proposition}\label{kofG-codif-can}
Associated to data $(Q=Q(G,R),\cdot)$ defined above on a finite group $G$ is a coinner bicovariant codifferential calculus on $k(G).$ Its structure is 
\begin{gather*}
\Omega^1=kQ_1,\quad f. x\xrightarrow[]{(i)} y =f(x)x\xrightarrow[]{(i)} y,\quad x\xrightarrow[]{(i)} y. f= x\xrightarrow[]{(i)} y f(y),\\
\Delta_L (x\xrightarrow[]{(i)} y)=\sum_{h\in G}\delta_h\tens (h^{-1}x\xrightarrow[]{(i)} h^{-1}y),\quad \Delta_R(x\xrightarrow[]{(i)} y)=\sum_{h\in G} h\cdot (x\xrightarrow[]{(i)} y)\tens \delta_h,\\
i(x\xrightarrow[]{(i)} y)=\delta_{x,y}(\lambda_{i1}(x)-\delta_{i,1})\delta_x,\quad\textit{with }\vartheta=\sum_{x\in G} (x\xrightarrow[]{(1)} x)^\ast.
\end{gather*}
where $\lambda_{ij}(h):=\lambda_{ij}(h,e)$
\end{proposition}
\proof As shown in the proof of Proposition~\ref{kofG-dif-can}, the bimodule $kQ_1$ together with $\Delta_L,\Delta_R$ is a $k(G)$-Hopf bimodule. It suffices to show that $i:kQ_1\to k(G)$ is a $k(G)$-bimodule map such that $\Delta\circ i=i\tens\id\circ \Delta_R+\id\tens i\circ \Delta_L.$ Since $i(x\xrightarrow[]{(i)} y. f)=i(x\xrightarrow[]{(i)} y f(y))=f(y) \delta_{x,y}(\lambda_{i1}(x)-\delta_{i,1})\delta_x=f(x)\delta_{x,y}(\lambda_{i1}(x)-\delta_{i,1})\delta_x=i(x\xrightarrow[]{(i)} y)f,$ thus $i$ is a left $k(G)$-module map. Similarly $i$ is a right $k(G)$-module map. Next $(i\tens\id\circ\Delta_R+\id\tens i\circ \Delta_L)(x\xrightarrow[]{(i)} y)=\sum_{h\in G}\delta_{x,y}\sum_j\lambda_{ij}(h)(\lambda_{j1}(xh^{-1})-\delta_{j,1})\delta_{xh^{-1}}\tens \delta_h+\sum_{h\in G}\delta_h\tens \delta_{x,y}(\lambda_{i1}(h^{-1}x)-\delta_{i,1})\delta_{h^{-1}x},$ after a change of variables, which equals to 
$\sum_{h\in G}\delta_{x,y}(\sum_j\lambda_{ij}(h)\lambda_{j1}(xh^{-1})-\delta_{i,1}) \delta_{xh^{-1}}\tens\delta_h$. This equals to $\sum_{h\in G}\delta_{x,y}(\lambda_{i1}(x)-\delta_{i,1})\delta_{xh^{-1}}\tens\delta_h=\Delta\circ i(x\xrightarrow[]{(i)} y)$ due to $\lambda_{i1}(x)=\sum_j\lambda_{ij}(h)\lambda_{j1}(xh^{-1})$ since $\lambda$ is the matrix representation of $\cdot.$  
Therefore $\Delta\circ i=i\tens \id \circ \Delta_R+\id\tens i\circ \Delta_L$ and $(kQ_1,i)$ is a coinner bicovariant codifferential.
\endproof

\begin{theorem}\label{kofG-codif-thm}
Let $G$ be a finite group. Any coinner bicovariant codifferential calculus on $k(G)$ is isomorphic to one of the canonical forms in Proposition~\ref{kofG-codif-can} for some pair $(Q(G,R),\cdot).$ There is a bijection between the set of isomorphism classes of coinner bicovariant codifferential calculi and the set of isomorphism classes of such pairs.
\end{theorem}
\proof Here $(Q=Q(G,R),\cdot)$ and $(Q'=Q(G,R'),\cdot')$ are isomorphic if $R=R'$ and there exists a linear isomorphism $\varphi: kQ_1\to kQ'_1$ such that $\sum_j\lambda_{ij}(h,g)\varphi_{jk}(hgh^{-1})=\sum_j\varphi_{ij}(g)\lambda'{}_{jk}(h,g)$ and $\sum_j\varphi_{ij}(e)(\lambda'_{j1}(x)-\delta_{j,1})=\lambda_{i1}(x)-\delta_{i,1}$ for all $h,g\in G$ where $\varphi(x\xrightarrow[]{(i)} y) =\sum_j \varphi_{ij}(x^{-1}y) x\xrightarrow[]{(j)'} y.$ The assumption entails that $\sum_{j}\varphi_{ij}(e)\lambda'_{jk}(h)=\sum_j\lambda_{ij}(h)\varphi_{jk}(e)$ and $\sum_j\lambda_{ij}(h)\varphi_{j1}(e)=\lambda_{i1}(h)+\varphi_{i1}(e)-\delta_{i,1}.$ Now suppose $(\Omega^1,i)$ is any coinner bicovariant codifferential calculus on $k(G).$ By (\ref{fcodiff}) and Lemma~\ref{kofG-codif-lem} 2)-3), without loss of generality, we can assume $\Omega^1=k(G)\tens \Lambda^1$ for some data $(\Lambda^1,[\vartheta])$ where $\Lambda^1=\oplus_{g\in G}\Lambda^1_g$ a $G$-graded left $G$-module such that $h\la\Lambda^1_g =\Lambda^1_{hgh^{-1}}$ for any $g,h\in G$ and $[\vartheta]\in (\Lambda^1_e)^*/(\Lambda^1_e)^*_G.$ Clearly, $i(\delta_x\eta)=\<\vartheta\ra (x-e),\eta\>\delta_x.$

We now construct a data pair $(Q=Q(G,R),\cdot)$ from a given data $(\Lambda^1,[\vartheta])$ as follows. Firstly, we take $R=\sum_{C\in \mathfrak{C}}R_C C$ with $R_C=\dim \Lambda^1_g$ for some $g\in C,\,C\in\mathfrak{C}$ and define coloured Hopf quiver $Q=Q(G,R=\sum_{C\in \mathfrak{C}}R_CC).$ 
Secondly, for each $g\in C,\,C\in \mathfrak{C}$ with $R_C\neq 0,$ we choose a basis $\{e_g^{(i)}\}_{i=1}^{R_C}$ for $\Lambda^1_g.$ In particular, the basis $\{e_e^{(i)}\}_{i=1}^{R_{\{e\}}}$ of $\Lambda^1_e$ is chosen in the way that $(e_e^{(1)})^*=\vartheta,$ any fixed representative for $[\vartheta]\in (\Lambda^1_e)^*/(\Lambda^1_e)^*_G.$ Thirdly, noting that choice of basis elements provides a $k(G)$-bimodule isomorphism $\varphi:k(G)\tens\Lambda^1\to kQ_1$ by sending $\delta_x\tens e_g^{(i)}$ to arrow $x\xrightarrow[]{(i)}xg,$ the map $\varphi$ then transfers the $k(G)$-bicomodule structure on $\Omega^1=k(G)\tens\Lambda^1$ to $kQ_1.$ More precisely, the left $k(G)$-coaction $\Delta_L(\delta_x\tens e_g^{(i)})=\sum_{h\in G}\delta_h\tens\delta_{h^{-1}x}\tens e_g^{(i)}$ determines the canonical right $G$-action on $kQ_1$, namely $x\xrightarrow[]{(i)} xg\cdot h=h^{-1}x\xrightarrow[]{(i)} h^{-1}xg,$ while the right $k(G)$-coaction $\Delta_R(\delta_x\tens e_g^{(i)})=\sum_{h\in G}\delta_{xh^{-1}}\tens h\la e_g^{(i)}\tens \delta_h$ determines the left $G$-action on $kQ_1$, namely $h\cdot x\xrightarrow[]{(i)} xg=\lambda_{ij}(h,g) xh^{-1}\xrightarrow[]{(j)}xgh^{-1},$ where $h\la e_g^{(i)}=\sum_{j}\lambda_{ij}(h,g) e_g^{(j)}.$ 
Therefore one can verify that we now have a desired data pair $(Q=Q(G,R),\cdot)$. 
Under the map $\varphi,$ one can check that $\Omega^1=k(G)\tens \Lambda^1$ is isomorphic to the coinner bicovariant codifferential calculus $kQ_1$ associated to this data $(Q=Q(G,R),\cdot)$ in Proposition~\ref{kofG-codif-can}, since $i (\delta_x\tens e^{(i)}_g)=\<(e^{(1)}_e)^*, (x-e)\la e^{(i)}_g\>\delta_x=\delta_{g,e}(\lambda_{i1}(x)-\delta_{i,1})\delta_x=i(x\xrightarrow[]{(i)}xg)$ and thus $i=i_{kQ_1}\circ\varphi.$ The proof for the bijection on isoclasses is straightforward as two Hopf quiver data are defined to be isomorphic in order that the associated calculi are isomorphic. 
\endproof

 The association of a coinner bicovariant codifferential calculus on $k(G)$ to a `quiver codifferential calculus' in Proposition~\ref{kofG-codif-can} depends on the choice of the basis elements, but different choices lead to isomorphic codifferential calculi.
Also note that if the given quiver $Q=Q(G,R)$ has no loops, then the associated codifferential $i$ in Proposition~\ref{kofG-codif-can} is zero.

\subsection{Augmented exterior algebras on $k(G)$ and $kG$} In this section, we study when the first order bicovariant (co)differential calculi on $k(G)$ and $kG$ in the preceding two subsections extend to strongly bicovariant (co)differential graded algebras and when these are  augmented. 

For the `universal inner' construction in d1) on $k(G)$ we have
\begin{corollary}\label{kofG-univ-aug}
Let $A=k(G)$ and $(\Lambda^1,\theta,\vartheta)$ define an augmented first order inner and coinner bicovariant calculus with $\Delta_R\theta=\theta\tens 1$ (or $\theta_e\in{}_G\Lambda^1_e$) and $\vartheta\in(\Lambda^1_e)^*$ as in Lemmas~\ref{kofG-dif-lem} and~\ref{kofG-codif-lem}. The inner strongly bicovariant differential exterior algebra $\Omega_\theta(G)$ in d1) is augmented with codifferential $i$ automatically extended from first order. 
\end{corollary}
\begin{proof}
From $\Delta_R\theta=\theta\tens 1$ or $\theta_e\in {}_G\Lambda^1_e$, we know $(g-1)\la \theta_e=0$ for all $g\in G$.  This means $i(\theta)=\sum_{g\in G}\delta_g\<\vartheta,(g-1)\la\theta_e\>=0,$ so $\theta\ra i(\theta)=0.$
\end{proof}

We can also illustrate inner strongly bicovariant differential graded algebra  $\Omega_\theta(G)$ in d1) in terms of Hopf quivers.

\begin{corollary}\label{kofG-qpathalgebra} Let $(\bar Q\subseteq Q,\ast)$ be a Hopf digraph-quiver triple on a finite group $G$. The first order calculus in Proposition~\ref{kofG-dif-can} extends to $\Omega_\theta(G)$ as a quotient of the path super-Hopf algebra the relation that the element $\sum_{x\in G, a,b\in\bar C}x \xrightarrow[]{(1)} xa\xrightarrow[]{(1)}xab$ is central in the path algebra. Moreover, any inner strongly bicovariant calculus on $k(G)$ with inner data right-invariant that is generated by its degrees 0,1 and has the canonical form for degree one is isomorphic to a quotient of $\Omega_\theta(G)$ for some Hopf digraph-quiver triple.
\end{corollary}
\proof We construct $\Omega_\theta(G)$ from d1) noting that $\theta$ in Proposition~\ref{kofG-dif-can} is invariant under $\Delta_R$, the element stated being $\theta^2$. Here the path super-Hopf algebra is isomorphic to $k(G)\rbiprod T_-\Lambda^1$ as explained. The second part follows from the universal property of the construction $\Omega_\theta(G).$
\endproof

For the well-known `Woronowicz' construction in d2) on $k(G)$, we have
\begin{corollary}\label{kofG-wor}
Let $A=k(G)$ and $(\Lambda^1,\theta)$ define an inner generalised bicovariant differential calculus as in Lemma~\ref{kofG-dif-lem}. This extends to an inner strongly bicovariant differential graded algebra $(\Omega_w(G),\extd)$ in d2) iff $\Delta_R\theta=\theta\tens 1$, i.e. $\theta_e\in{}_G\Lambda^1_e$, where $\Omega_w(G)=k(G)\rbiprod B_-(\Lambda^1)$ is generated by $k(G),\Lambda^1$ with relations, coproduct and exterior derivative
\[ v\delta_g=\delta_{g|v|^{-1}}v, \quad \Delta\delta_g=\sum_{h\in G}\delta_{gh^{-1}}\tens\delta_h,\quad \Delta v=1\tens v+\sum_{h\in G}h\la v\tens \delta_h,\quad \extd=[\theta,\ \}\]
 for all homogeneous $v\in\Lambda^1$ of $G$-degree $|v|$ and all $g\in G$. \end{corollary}
\proof  The braiding in our case on $\Lambda^1$ is $\Psi(\eta\tens \zeta)= \sum g\la \zeta\tens \eta \delta_g(|\eta|)=|\eta|\la \zeta \tens \eta$. If $\theta_e\neq 0,$ the condition $\Psi(\eta\otimes\theta)=\theta\otimes\eta$ for all $\eta$ means $h\la\theta_e=\theta_e$ for $h$ where $\Lambda^1_h\neq0.$ The condition $\{\Delta_R\theta-\theta\otimes1,\Delta_R(\eta)\}=0$ implies $\{\Delta_R(\theta_e)-\theta_e\otimes1,\Delta_R(\eta)\}=0$ for all $\eta\in\Lambda^1.$ Choose $\eta=\theta_e,$ we have $2(h\la\theta_e)^2=\theta_e(h\la\theta_e)+(h\la\theta_e)\theta_e.$ Since $h\la\theta_e\in\Lambda^1_e,$ we can extend $\theta_e$ to a basis of $\Lambda^1_e$ to prove that $h\la\theta_e=\theta_e$ for all $h\in G,$
 which means $\Delta_R(\theta_e)=\theta_e\otimes1$ and thus $\Delta_R\theta=\theta\otimes1.$ The rest of $(\Omega_w(G),\extd)$ is an elaboration of the general construction of d2). \endproof

For the tensor exterior algebra construction c3) on $k(G)$  we have
\begin{corollary}\label{kofG-tens}
Let $(\Omega^1(Q(G,R),\cdot),i)$ be the coinner bicovariant codifferential calculus on $k(G)$ in Proposition~\ref{kofG-codif-can}. The associated path super-Hopf algebra $kQ$ automatically becomes  a coinner strongly bicovariant codifferential graded algebra with $i$ extends to higher degrees as a super-derivation on paths, i.e.
\[i(p)=i(\alpha_1\cdots\alpha_{n-1})\alpha_n+(-1)^n\alpha_1\cdots\alpha_{n-1}i(\alpha_n),\]
for any path $p=\alpha_1\cdots\alpha_n.$
\end{corollary}
\proof It is clear that $\Omega_{tens}(G)=kQ.$ It suffices to check that $\<\vartheta,\xi\>\eta=\eta\z\<\vartheta,\xi\ra\eta\o\>$ for any $\xi,\eta\in \Lambda^1.$ Using the notations in the proof of Theorem~\ref{kofG-codif-thm}, let $\xi=e^{(i)}_g=\sum_{x\in G}x\xrightarrow[]{(i)}xg$ and $\eta=e^{(j)}_h=\sum_{x\in G}x\xrightarrow[]{(i)}xh$ be any two basis element in $\Lambda^1.$ Clearly, $\eta\z\tens\xi\ra\eta\o=g\la e^{(j)}_h\tens e^{(i)}_g=\sum_k\lambda_{jk}(g,h)e^{(k)}_h\tens e^{(i)}_g,$ thus $\<\vartheta,\xi\>\eta=\delta_{g,e}\delta_{i,1}e^{(j)}_h$ while $\eta\z\<\vartheta,\xi\ra\eta\o\>=\lambda_{jk}(g,h)e^{(k)}_h\delta_{g,e}\delta_{i,1}.$ These two equations agree with each other if and only if $\lambda_{jk}(e,h)=\delta_{j,k},$ which is always true as $\lambda_{jk}(e,h)$ is the matrix representation of the action of identity $e$. 
\endproof

\begin{example}
Let $G=\mathbb{Z}_2=\langle g \rangle,$ with $Q=Q(\Z_2,R)$ given by $R=2\{g\}$ and $\bar Q=e{\to\atop\leftarrow}g$. 
Consider the Hopf digraph-quiver triple $(\bar{Q}\subset Q,\ast)$ with $\ast$  the canonical left action. Denote the arrows $\alpha_i:e\xrightarrow[]{(i)}g$ and $\beta_i: g\xrightarrow[]{(i)}e,\ i=1,2,$ as shown.
\[\begin{tikzpicture}[descr/.style={fill=white},text height=1.5ex, text depth=0.25ex]
\node (a) at (0,0) {${\circ\atop e}$};
\node (b) at (2.5,0) {${\circ\atop g}$};
\path[->,font=\scriptsize,>=angle 90]
([yshift= 18pt]a.east) edge node[above] {$\alpha_1$} ([yshift= 18pt]b.west)
([yshift= 10pt]a.east) edge node[descr] {$\alpha_2$} ([yshift= 10pt]b.west)
([yshift= -2pt]b.west) edge node[descr] {$\beta_1$} ([yshift= -2pt]a.east)
([yshift=-10pt]b.west) edge node[below] {$\beta_2$} ([yshift=-10pt]a.east);
\end{tikzpicture}\] 
The path super-Hopf algebra is $kQ$ is $k\< 1,\delta_e,\alpha_i,\beta_i\>$ modulo the  the relations
\[ \delta_e^2=\delta_e,\quad \delta_e\alpha_i=\alpha_i,\quad \alpha_i\delta_e=\delta_e\beta_i=0,\quad \beta_i\delta_e=\beta_i,\quad \alpha_i\alpha_j= \beta_i\beta_j=0,\quad\forall i,j\]
with grading $|\alpha_i|=|\beta_i|=1$ and super-coproduct defined on generators by
\[  \Delta\delta_e=\delta_e\tens\delta_e+\delta_g\tens \delta_g,\quad \Delta\alpha_i=\delta_e\tens\alpha_i+\delta_g\tens\beta_i+\alpha_i\tens\delta_e+\beta_i\tens\delta_g,\]
\[ \Delta\beta_i=\delta_e\tens\beta_i+\delta_g\tens\alpha_i+\beta_i\tens\delta_e+\alpha_i\tens\delta_g\]
where $\delta_g=1-\delta_e$. The counit is $\epsilon(\delta_e)=1,\ \epsilon(\alpha_i)=0,\ \epsilon(\beta_i)=0.$ The left-invariant 1-forms are $\Lambda^1=\Lambda^1_g=k\text{-}\mathrm{span}\{e^{(i)}\}$ where $e^{(i)}=\alpha_i+\beta_i,$ with (co)action given by $e^{(i)}\ra\delta_e=0$ and $\Delta_R(e^{(i)})=e^{(i)}\tens 1$. Then \[ kQ\isom k(\Z_2)\rcross T_-\Lambda^1=k(\Z_2)\rcross k\langle e^{(1)},e^{(2)}\rangle\] with cross relations $e^{(i)}\delta_e=\delta_g e^{(i)}$ for all $i$, and the tensor product coalgebra as the coaction is trivial. Hence $\Delta e^{(i)}=e^{(i)}\tens 1+1\tens e^{(i)}$ and $\eps(e^{(i)})=0$.
From Corollary~\ref{kofG-tens}, we know $kQ$ is a strongly bicovariant differential graded algebra with $i=0.$

Next we compute $\Omega_\theta(\Z_2)$ in d2). Here $\theta=e^{(1)}=\alpha_1+\beta_1$ so we have $\theta^2\ra\delta_g=(\theta\ra\delta_e)(\theta\ra\delta_g)+(\theta\ra\delta_g)(\theta\ra\delta_e)=0$ and $[\theta^2,e^{(1)}]=0$, so
\[ \Omega_\theta(\Z_2)=k(\Z_2)\rbiprod k\langle e^{(1)},e^{(2)}\rangle/{\langle e^{(1)}e^{(1)}e^{(2)}-e^{(2)}e^{(1)}e^{(1)}\rangle}.\] Equivalently, as $\delta_e(e^{(1)}e^{(1)}e^{(2)}-e^{(2)}e^{(1)}e^{(1)})=\alpha_1\beta_1\alpha_2-\alpha_2\beta_1\alpha_1$ and $\delta_g(e^{(1)}e^{(1)}e^{(2)}-e^{(2)}e^{(1)}e^{(1)})=\beta_1\alpha_1\beta_2-\beta_2\alpha_1\beta_1,$ it follows that $\Omega_\theta(\Z_2)$ is $kQ$ modulo the additional relations
\[ \alpha_1\beta_1\alpha_2=\alpha_2\beta_1\alpha_1,\quad  \beta_1\alpha_1\beta_2=\beta_2\alpha_1\beta_1.\]
Here $\theta^2=\alpha_1\beta_1+\beta_1\alpha_1$ is central and requiring this is equivalent to imposing these relations as in Corollary~\ref{kofG-qpathalgebra}. The new feature not present for the path algebra is the super-derivation  $\extd=[\theta,\ \}$. Thus
\[ \extd \delta_e=\beta_1-\alpha_1,\quad \delta \theta=2\theta^2,\quad \delta e^{(2)}=e^{(1)}e^{(2)}+e^{(2)}e^{(1)}\] or
\[ \extd\alpha_1=\beta_1\alpha_1+\alpha_1\beta_1,\quad \extd\alpha_2=\beta_1\alpha_2+\alpha_2\beta_1,\quad
\extd\beta_1=\alpha_1\beta_1+\beta_1\alpha_1,\quad \extd\beta_2=\alpha_1\beta_2+\beta_2\alpha_1\]
 extended as a super-derivation with $\extd^2=0$.

Finally, since the braiding is trivial on $\Lambda^1,$  $B_{-}(\Lambda^1)=B_{-}^{quad}(\Lambda^1)=\Lambda(e^{(1)},e^{(2)})$, the usual Grassmann algebra on generators $\{e^{(i)}\}$ with anticommutative relations and basis $\{1,e^{(1)},e^{(2)},e^{(1)}\wedge e^{(2)}\}.$ Thus the canonical `minimal' calculus as in d1) is $k(\Z_2)\rcross B_-(\Lambda^1)=k(\Z_2)\rcross\Lambda(e^{(1)},e^{(2)})$ with cross relations as above.  Equivalently, as $\delta_e e^{(i)}=\alpha_i$ and $\delta_g e^{(i)}=\beta_i,$ it follows that $\Omega_\theta(\Z_2)$ is a quotient of the path algebra by the further relations \[ \alpha_2\beta_1=-\alpha_1\beta_2,\quad \beta_2\alpha_1=-\beta_1\alpha_2,\quad \alpha_i\beta_i=0,\quad \beta_i\alpha_i=0,\quad i=1,2.\]
Here $\theta^2=\alpha_1\beta_1+\beta_1\alpha_1=0$ agrees with $\theta^2$ is graded central in d1). In this quotient we see that $\delta e^{(i)}=0$ or equivalently \[ \extd\alpha_1=\extd\beta_1=0,\quad \extd\alpha_2=-\extd\beta_2=\beta_1\alpha_2-\alpha_1\beta_2.\]\end{example}

\bigskip We now turn to when a first order bicovariant (co)differential calculus on $kG$ extends to higher orders and when such strongly bicovariant (co)differential grade algebras on $kG$ are augmented. 

We start with the coinner strongly bicovariant codifferential graded algebra $\Omega_\vartheta(kG)=kG\rbiprod B_\vartheta(\Lambda^1)$ in c1) for some bicovariant data $(\Lambda^1,\vartheta)$ in Lemma~\ref{kG-codif-lem}.
Choose a basis $\{e^{(i)}_g\}_{i=1}^{R_C}$ for each $\Lambda^1_g$ with $g$ in some conjugacy class $C$ such that $(e^{(1)}_g)^*=\iota_g$ whenever $\iota_g\neq 0$. Then the braided-super Hopf algebra
\[ B_{\vartheta}(\Lambda^1)=\big\{\sum_{g_1,\dots,g_n\in G \atop i_1,\dots,i_n}\lambda_{g_1,\dots,g_n}^{i_1,\dots,i_n}e^{(i_1)}_{g_1}\tens e^{(i_2)}_{g_1}\tens\cdots \tens e^{(i_n)}_{g_n}\in \Sh_-(\Lambda^1)\ |\ 
\lambda_{g_1,\dots,g_n}^{i_1,\dots,i_n} \text{obey (A) and (B)}\big\},
\]
where the conditions (A) are \[\sum_{a,b\in\bar{C}\atop ab=g}\lambda_{a,b,g_3,\dots,g_n}^{1,1,i_3\dots,i_n}=
\sum_{a,b\in\bar{C}\atop ab=g}\lambda_{g_1,a,b,g_4,\dots,g_n}^{i_1,1,1,i_4\dots,i_n}=\dots =
\sum_{a,b\in\bar{C}\atop ab=g}\lambda_{g_1,g_2,\dots,g_{n-2},a,b}^{i_1,i_2\dots,i_{n-2},1,1}=0,\]
for any $g_1,\dots g_{n},g\in G$ with $g\neq e$, and the conditions (B) are \[
\sum_{a,b\in\bar{C}\atop ab=e}\lambda_{a,b,g_1,\dots,g_{n-2}}^{1,1,i_1\dots,i_{n-2}}=
\sum_{a,b\in\bar{C}\atop ab=e}\lambda_{g_1,a,b,g_2,\dots,g_{n-2}}^{i_1,1,1,i_2\dots,i_{n-2}}=\dots =
\sum_{a,b\in\bar{C}\atop ab=e}\lambda_{g_1,g_2,\dots,g_{n-2},a,b}^{i_1,i_2\dots,i_{n-2},1,1},\] for any fixed $g_1,\dots g_{n-2}\in G$ and their fixed indices $i_1,\dots,i_{n-2}.$ The coinner codifferential is given by
\[ i(h\tens v_1\tens\cdots\tens v_n)=\<\vartheta,v_1\>hg\tens(v_2\tens\cdots \tens v_n)+(-1)^n h\tens(v_1\tens\cdots \tens v_{n-1})\<\vartheta,v_n\>,\]
for all $h\in G$, $v_1\tens\cdots\tens v_n\in B_{\vartheta}(\Lambda^1)$ and $v_1\in \Lambda^1_g$ for some $g\in G$. 
\begin{corollary}\label{kG-max-aug}
Let $A=kG$ and $(\Lambda^1,\theta,\vartheta)$ define an augmented inner first order bicovariant differential calculus with $\theta\in \Lambda^1_e,\,\vartheta\in \Lambda^{1*}$ as in Lemmas~\ref{kG-dif-lem} and~\ref{kG-codif-lem}. 
Suppose $\<\vartheta,v\ra g\>=\<\vartheta,v\>$ for any $v\in \Lambda^1,\,g\in G,$ i.e. $\iota_e\in {}_G (\Lambda^{1}_e)^*,$ then the coinner bicovariant codifferential calculus $\Omega_\vartheta(G)=kG\rbiprod B_{\vartheta}(\Lambda^1)$ is augmented with $\extd=[\theta,\ \}.$ 
\end{corollary}
\proof It is suffices to show that the conditions stated for a3) 
are satisfied. If $\iota_e=0$ there is nothing to prove. If $\iota_e\neq 0.$
First, we note that for any $v\in \Lambda^1,\,g\in G$, $\<\vartheta,v\ra g\>=\<\vartheta,v\>$ is equivalent to $\<\iota_e,v\ra g\>=\<\iota_e,v\>$ i.e. $\<(g-1)\la\iota_e,v\>=0,$ which means $\iota_e\in {}_G(\Lambda^1_e)^*.$ For any $v\in \Lambda^1,$ without loss of generality, say $v\in \Lambda^1_g$ for some $g.$ Then $\<\vartheta\tens\vartheta,v\z\tens\tilde{\omega}(v\o)\>=\<\vartheta\tens\vartheta,v\tens\tilde{\omega}(g)\>=\<\vartheta\tens\vartheta,v\tens\theta\ra(g-1)\>=\<\iota_g,v\>\<\iota_e,\theta\ra(g-1)\>=\<\iota_g,v\>\<(g-1)\la\iota_e,\theta\>=0.$ This completes the proof.\endproof

Corollary~\ref{kofG-univ-aug} and Corollary~\ref{kG-max-aug} taken together imply
\begin{corollary}
Let $(\Lambda^1,\theta,\vartheta)$ define an augmented first order bicovariant differential calculus on $k(G)$ for a finite group $G$ with $\theta\in\Lambda^1$ right invariant and $\vartheta\in{\Lambda^{1*}_e}.$ (So $(\Lambda^{1*},\vartheta,\theta)$ on $kG$ is in the setting of Corollary~\ref{kG-max-aug}.) Then $\Omega_\theta(k(G))=k(G)\rbiprod \Lambda_\theta(\Lambda^1)$ in Corollary~\ref{kofG-univ-aug} and $\Omega_\vartheta(kG)=kG\rbiprod B_{\vartheta}({\Lambda^1}^*)$ in Corollary~\ref{kG-max-aug} are mutually dually super-Hopf-algebras with the differential on one side dual to the codifferential on the other side.
\end{corollary}

For the `Woronowicz' construction d2) on $kG$, we have

\begin{proposition}\label{kG-wor} Let $A=kG$ and $(\Lambda^1,\theta)$ define an inner bicovariant calculus in Lemma~\ref{kG-dif-lem} 3). The bimodule relations and exterior derivative are
\[  \eta. g= g (\eta\ra g),\quad \extd g= g(\theta\ra g-\theta),\quad\forall\eta\in\Lambda^1,\ g\in G.\]
The conditions for a differential exterior algebra require $\Delta_R\theta=\theta\tens 1$ if the calculus is connected. The super-Hopf algebra structure of $\Omega_w(kG)=kG\rbiprod B_-(\Lambda^1)$ and exterior derivative  are
\[ \Delta \eta=1\tens\eta+\eta\z\tens\eta\o,\quad \Delta g=g\tens g,\quad \forall \eta\in \Lambda^1,\ g\in G,\quad \extd=[\theta,\ \}.\]
\end{proposition}
\proof The condition $\Psi(\theta\tens\eta)=\eta\tens\theta$ for all $\eta$ means $\sum_g \theta_g\tens (\eta\ra g-\eta)=0$ for all $\eta$. This requires the action of $g$ to be the identity whenever $\theta_g\ne 0$. This is a strong condition and among other things requires $g$ where $\theta_g\ne 0$ to commute with all $h$ where $\Lambda^1_h\ne 0$. It also needs that such $g$ commute with all $\eta$ in $\Omega^1$ as stated. Finally, setting $\eta=\theta$ it also requires $\sum_g \theta_g\tens\extd g=0$ which for a connected calculus (where $\ker\,\extd=k\{1\}$) means $\theta=\theta_e$.  In this case we have an exterior super-Hopf algebra $\Omega_w(kG)=kG\rbiprod B_-(\Lambda^1)$,  where we extend
the above with the relations of $B_-(\Lambda^1)$, the  super homomorphism property of $\Delta$ and the graded-derivation property of $\extd$.\endproof

Note that the standard part $\bar\Lambda^1\subseteq\Lambda^1_e$ in this case and hence $\bar\Lambda$ is the usual Grassmann algebra on $\bar\Lambda^1$ in keeping with the known theory of standard bicovariant calculi on $kG$.

\begin{corollary} Both $\Omega_w(G)=k(G)\rbiprod B_-(\Lambda^1)$ constructed by $(\Lambda^1,\theta)$ in Corollary~\ref{kofG-wor} and $\Omega_w(kG)=kG\rbiprod B_-(\Lambda^1{}^*)$ constructed by $(\Lambda^{1*},\vartheta)$ in Proposition~\ref{kG-wor} with $\Lambda^1$ dual to $\Lambda^{1*}$, are augmented and are mutually dual as graded super-Hopf algebras.
\end{corollary}
\proof  This is now clear from \cite[Proposition 4.9 and Lemma 4.10]{MaTao1}.
As both sides of $\Omega^1$ are inner they both extend to a $\Omega_w$ from d1)  and hence both sides are augmented.  \endproof

For the shuffle construction in d3) on $kG$ we have

\begin{corollary}\label{kG-shu}
Let $(\Omega^1(Q(G,R),\ast),\extd=[\theta,\ ])$ be the inner bicovariant differential calculus on $kG$ in Proposition~\ref{kG-dif-can}. The associated path super-Hopf algebra $kQ$ becomes an inner strongly bicovariant differential graded algebra with differential given by $\extd=[\theta,\ \},$ namely
\[\extd(p)=\theta\cdot (\alpha_1\alpha_2\cdots\alpha_n)-(-1)^n (\alpha_1\alpha_2\cdots\alpha_n)\cdot \theta,\]
for all path $p=\alpha_1\alpha_2\cdots\alpha_n$ in $Q.$
\end{corollary}
\proof It is clear that the shuffle construction in d3) is path super-Hopf algebra with product given by super-quantum shuffle, namely $\Omega_{sh}(kG)=kQ.$ The product between the arrows in $kQ_1$ in our notation can be computed by \[\alpha\cdot \beta=[\alpha\ast s(\beta)][t(\alpha)\ast\beta]-[s(\alpha)\ast\beta][\alpha\ast t(\beta)],\] where $s(\alpha),t(\alpha)$ denotes the source and target vertices of each arrow $\alpha$ and $[\ ]$'s connected by concatenation. The formulae for higher orders are refer to~\cite[(3.1)]{Hua}. It suffices to show that  braiding $\Psi$ on $\Lambda^1$ such that $\Psi(\eta\tens\theta)=\theta\tens\eta$ for all $\eta\in\Lambda^1$ in d3). In fact, for any basis element $\eta\in \Lambda^1, $ say $\eta:e\xrightarrow[]{(i)} g$ for some $i$ and $g\in C$ with $R_C\neq 0,$ then $\Psi(\eta\tens\theta)=(e\xrightarrow[]{(1)}e
)\tens(e\xrightarrow[]{(i)}g)\ra e=\theta\tens\eta.$
\endproof

We illustrate the construction of Corollary~\ref{kG-max-aug}, Corollary~\ref{kG-shu} in terms of colored Hopf quivers by the following example.

\begin{example}\label{example}
Let $G=\mathbb{Z}_2=\<g\>$ with $Q=Q(\Z_2,R)$ and $\bar{Q}=Q(\Z_2,r)$ , where ramification data are given by $R=\{e\}+2\{g\}$ and $r=\{g\}.$ Denote the arrows $\alpha_i:e\xrightarrow[]{(i)} g,\ \beta_i:g\xrightarrow[]{(i)} e,\ \gamma:e\to e$ and $\rho:g\to g,\ i=1,2,$ as below.

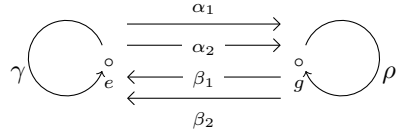
\[\begin{tikzpicture}[descr/.style={fill=white},text height=1.5ex, text depth=0.25ex, arrow/.style={->}]
\node (a) at (0,0) {${\circ\atop e}$};
\node (b) at (2.5,0) {${\circ\atop g}$};
\path[->,font=\scriptsize,>=angle 90]
([yshift= 18pt]a.east) edge node[above] {$\alpha_1$} ([yshift= 18pt]b.west)
([yshift= 10pt]a.east) edge node[descr] {$\alpha_2$} ([yshift= 10pt]b.west)
([yshift= -2pt]b.west) edge node[descr] {$\beta_1$} ([yshift= -2pt]a.east)
([yshift=-10pt]b.west) edge node[below] {$\beta_2$} ([yshift=-10pt]a.east);
\draw [arrow] (a.west) ++(+0.15cm,10pt) arc [start angle=20, end angle=340, radius=0.5cm] node[left]{$\gamma\,\ \qquad$};
\draw [arrow] (b.east) ++(-0.15cm,10pt) arc [start angle=160, end angle=-160, radius=0.5cm] node[right]{$\,\ \qquad\rho$};
\end{tikzpicture}\] 

Consider the right-handed Hopf digraph-quiver triple $(\bar{Q}\subseteq Q,\cdot)$ with $\cdot$ given by \[\alpha_i\cdot g=(-1)^i\beta_i,\ \beta_i\cdot g=(-1)^i\alpha_i,\ \gamma\cdot g=-\rho,\text{ and }\ \rho\cdot g=-\gamma.\]The path super-Hopf algebra $kQ$  is the $k$-space with grading $|\alpha_i|=|\beta_i|=|\gamma|=|\rho|=1$ spanned by all the paths (e.g. $\alpha_2\rho\rho\beta_1$) of $Q$ with comultiplication given by de-concatenation
\begin{gather*}
\Delta e=e\tens e,\quad \Delta g=g\tens g,\quad \Delta\alpha_i=e\tens\alpha_i+\alpha_i\tens g,\quad \Delta\beta_i=g\tens\beta_i+\beta_i\tens e,\\
\Delta\gamma=e\tens \gamma+\gamma\tens e,\quad \Delta\rho=g\tens\rho+\rho\tens g,\quad \Delta(\beta_i\gamma)=g\tens\beta_i\gamma+\beta_i\tens\gamma+\beta_i\gamma\tens e,\ {\rm etc}\\
\epsilon(e)=1,\quad\epsilon(g)=1,\quad\epsilon(p)=0, \text{ for any nontrivial path $p$.}
\end{gather*}
The multiplication of $kQ$ is given by the quantum shuffle product. Between arrows in $kQ_1$, we have
\begin{gather*}
\alpha_1\cdot\alpha_1=0,\quad\alpha_1\cdot\alpha_2=\alpha_1\beta_2-\alpha_2\beta_1,
\quad\alpha_2\cdot\alpha_1=\alpha_1\beta_2+\alpha_2\beta_1,
\quad\alpha_2\cdot\alpha_2=2\alpha_2\beta_2,\\
\quad\gamma\cdot\gamma=0,\quad\gamma\cdot\alpha_i=\alpha_i\rho+\gamma\alpha_i,
\quad\alpha_i\cdot\gamma=\alpha_i\rho-\gamma\alpha_i,\\
\quad\gamma\cdot\beta_i=(-1)^i(\rho\beta_i+\beta_i\gamma),
\quad\beta_i\cdot\gamma=\beta_i\gamma-(-1)^i\rho\beta_2,\ {\rm etc}.
\end{gather*}
The left-invariant $1$-forms are $\Lambda^1=k\textrm{-span}\{\gamma,\alpha_i\}$ where $\Lambda^1_e=k\{\gamma\}$ and $\Lambda^1_g=k\{\alpha_1,\alpha_2\}$ with the coaction $\gamma\ra g=-\gamma,\ \alpha_i\ra g=(-1)^i\alpha_i,\ i=1,2.$ Then $kQ\cong k\Z_2\rbiprod \Sh_-(\Lambda^1)$ has cross relation $\gamma.g=-g.\gamma$ and $\alpha_i.g=(-1)^i g.\alpha_i.$ Let $\theta=\gamma\in\Lambda^1_e,$ then $(\Omega_{sh}(kG)=kQ,\extd=[\theta,\ \})$ is the inner strongly bicovariant differential graded algebra in Corollary~\ref{kG-shu}. In particular,
\begin{gather*}
\extd e=0,\quad\extd g=-2\rho\\
\extd\gamma=\extd\rho=0,\quad\extd \alpha_i=2\alpha_i\rho,\quad \extd\beta_i=(-1)^i2\rho\beta_i,\ \text{etc.}
\end{gather*}

Choose $\vartheta=\alpha_1^*+\beta_1^* \in  (\Lambda^1_g)^*,$ we can compute $\Omega_\vartheta(k\Z_2)=k\Z_2\rbiprod B_{\vartheta}(\Lambda^1)$ from the analysis before Corollary~\ref{kG-max-aug}. Note that $\bar{C}=\{g\}$ and $\bar{C}\bar{C}=\{e\},$ so
\begin{align*}
B_{\theta^*}(\Lambda^1)=\big\{&\sum_{g_1,\dots,g_n\in \Z_2 \atop i_1,\dots,i_n}\lambda_{g_1,\dots,g_n}^{i_1,\dots,i_n}e^{(i_1)}_{g_1}\tens e^{(i_2)}_{g_1}\tens\cdots \tens e^{(i_n)}_{g_n}\in \Sh_-(\Lambda^1)\ |\ & \\
\lambda_{g,g,g_1,\dots,g_{n-2}}^{1,1,i_1\dots,i_{n-2}} &=\lambda_{g_1,g,g,g_2,\dots,g_{n-2}}^{i_1,1,1,i_2\dots,i_{n-2}}=\dots =\lambda_{g_1,g_2,\dots,g_{n-2},g,g}^{i_1,i_2\dots,i_{n-2},1,1},\
\forall g_1,\dots g_{n-2},  i_1,\dots,i_{n-2}.\ \big\}
\end{align*} with $e^{(i_l)}_{g_l}\in\{\gamma,\alpha_1,\alpha_2\}.$ One can see that $\Omega^2=kQ_2,$ $\Omega^n\subsetneq kQ_n$ for any $n\ge 3$ and write down a specific basis for each degree. For instance, $\Omega^3_\vartheta(k\Z_2)$ has basis
\begin{align*}
(Q_3\setminus\{\alpha_1\beta_1\gamma,\,\gamma\alpha_1\beta_1,\,\alpha_1\beta_1\alpha_2,\,\alpha_2\beta_1\alpha_1,\linebreak \beta_1\alpha_1\rho,\,\rho\beta_1\alpha_1,\,\beta_1\alpha_1\beta_2,\,\beta_2\alpha_1\beta_2\})\\
\cup \{\alpha_1\beta_1\gamma+\gamma\alpha_1\beta_1,\,\alpha_1\beta_1\alpha_2+\alpha_2\beta_1\alpha_1,\,\beta_1\alpha_1\rho+\rho\beta_1\alpha_1,\linebreak \beta_1\alpha_1\beta_2+\beta_2\alpha_1\beta_2\}
\end{align*}
and $\dim \Omega^3_\vartheta(k\Z_2)=50<54=\dim kQ_3.$  Similarly, 
$\Omega^4_\vartheta(k\Z_2)$ is spanned by basis elements like
\begin{gather*}
\alpha_1\beta_1\gamma\gamma+\gamma\alpha_1\beta_1\gamma+\gamma\gamma\alpha_1\beta_1,\quad \alpha_1\beta_1\gamma\alpha_1+\gamma\alpha_1\beta_1\alpha_1,\quad \alpha_1\beta_1\gamma\alpha_2+\gamma\alpha_1\beta_1\alpha_2+\gamma\alpha_2\beta_1\alpha_1,\\
\alpha_1\beta_1\alpha_1\rho+\alpha_1\rho\alpha_1\beta_1,\quad\alpha_1\beta_1\alpha_2\rho+\alpha_2\beta_1\alpha_1\rho+\alpha_2\rho\beta_1\alpha_1,\\
\alpha_1\beta_1\alpha_1\beta_1,\quad \alpha_1\beta_1\alpha_1\beta_2+\alpha_1\beta_2\alpha_1\beta_1,\quad \alpha_1\beta_1\alpha_2\beta_1+\alpha_2\beta_1\alpha_1\beta_1,\\
\alpha_1\beta_1\alpha_2\beta_2+\alpha_2\beta_1\alpha_1\beta_2+\alpha_2\beta_2\alpha_1\beta_1, \text{ etc,}
\end{gather*}
and $\dim\Omega^4_\vartheta(k\Z_2)=138<162=\dim kQ_4.$
Thus $\Omega_\vartheta(k\Z_2)=k\Z_2\rbiprod B_{\vartheta}(\Lambda^1)$ is an infinite-dimensional coinner strongly bicovariant  codifferential graded algebra on $k\Z_2$ with the codifferential $i$ given by
\begin{gather*}
i(\alpha_1)=g-e,\quad i(\beta_1)=e-g,\quad i(\gamma)=i(\alpha_2)=i(\beta_2)=0\\
i(\alpha_1\beta_1)=\beta_1+\alpha_1,\quad i(\alpha_1\beta_2)=\beta_2,\quad i(\alpha_1\rho)=\rho,\ {\rm etc.}
\end{gather*}
According to Corollary~\ref{kG-max-aug}, $(\Omega_\vartheta(k\Z_2),i)$ is augmented with
inner differential given by $\extd=[\theta,\ \}$ with $\theta=\gamma.$ 
\end{example}

\section{Noncommutative  differential geometry on quivers}

In this section we look at elements of noncommutative Riemannian geometry for generalised differentials, with examples from the digraph-quiver pairs, generalising \cite{Ma:gra} to the quiver case. We also explain how quiver geometries naturally arise from discrete finite group quantum principal bundles in standard noncommutative differential geometry and we explore the noncommutative geometry of group algebras with the example of the symmetric group $S_3$, which is new even with a normal (not generalised) differential calculus and is dual to the previously known noncommutative Riemannian geometry of $k(S_3)$ from \cite{Ma:rieq}.  

\subsection{Metrics and connections} Here we recap the standard framework\cite{BegMa:starcomp, Ma:gra} with the small changes needed to extend to the generalised csse. Let $A$ be a unital algebra and $(\Omega^1,\extd)$ a generalised differential calculus over $A.$ A \textit{linear (left) connection} on a left $A$-module $E$ is a linear map $\nabla:E\to\Omega^1\tens_AE$ such that \[\nabla(a\omega)=\extd a\tens_A\omega+a\nabla\omega\]
for all $\omega\in E$, $a\in A.$  A connection $\nabla$ is called a \textit{(left) bimodule connection} if there exists a bimodule map $\sigma:E\tens_A\Omega^1\to \Omega^1\tens_AE$ such that 
\[\nabla(\omega a)=(\nabla\omega) a+\sigma(\omega\tens_A\extd a)\]
for all $\omega\in E$, $a\in A.$ Unlike for a standard calculus, the map $\sigma$ here is no longer fully determined by $\nabla$, thus the map $\sigma$ if it exists is additional data rather than a property of $\nabla$.  When $\Omega^2$ is defined, the \textit{curvature} of a connection $\nabla$ are defined by
\begin{gather}
R_\nabla:E\to\Omega^2\tens_AE,\quad R_\nabla=(\extd\tens\id-(\wedge\tens\id)(\id\tens\nabla))\nabla
\end{gather} 

\begin{lemma}\label{quiverrep} Let $X$ be a finite set and $\Omega^1(\bar Q,Q)$ the calculus on $A=k(X)$ associated to a digraph-quiver pair with vertex set $X$. 

(i) A left module $E$ means an $X$-graded space $E=\oplus_{x\in X} {}_xE$ and a left module with left connection $(E,\nabla)$ means a {\em quiver representatation} of $Q$ in the sense of a collection of maps $L_\beta:{}_{s(\beta)}E\to {}_{t(\beta)}E$ for all $\beta\in Q_1$. 

(ii) A bimodule $E$ means an $X$-$X$ bigraded space $E=\oplus_{x,y\in X}{}_xE_y$ and a bimodule with left bimodule connection $(E,\nabla,\sigma)$ means a left connection and $\sigma:E\tens_A \Omega^1\to \Omega^1\tens_A E$ a bimodule map satisfying $\sigma(v\tens_A \alpha)=-\sum_{\beta\in Q_1}\beta\tens_A L_\beta({}_{s(\beta)}v_{s(\alpha)})_{t(\alpha)}$ for all arrows $\alpha$ in the digraph $\bar{Q}$.
\end{lemma}
\proof (i) Recall that a representation of a quiver $Q$ means an assignment of vector spaces ${}_xE$ to each vertex $x\in Q_0=X$ and of a linear map $L_{\beta}:{}_{s(\beta)}E\to {}_{t(\beta)}E$ for each arrow $\beta\in Q_1$. We identify this information with 
\[\nabla v= \sum_{\alpha\in \bar{Q}_1}\alpha\tens_A {}_{t(\alpha)}v+\sum_{\beta\in Q_1}\beta\tens_A L_\beta ({}_{s(\beta)}v).\]
where ${}_xv$ is the component of $v$ in ${}_xE$. We check that
\[ \nabla (av)=\sum_{\alpha\in \bar{Q}_1}\alpha\tens_A a(t(\alpha)){}_{t(\alpha)}v+\sum_{\beta\in Q_1}\beta\tens_A L_\beta (a(s(\beta)){}_{s(\beta)}v)=a\nabla v+\sum_{\alpha\in \bar{Q}_1}(a(t(\alpha))-a(s(\alpha)))\alpha\tens_A v\]
which is $a\nabla v+\extd a\tens_A v$ as required. (ii) To check the bimodule connection property we assume a bigrading. Then
\[ \nabla (va)-(\nabla v)a=\sum_{\beta\in Q_1}\sum_{z\in X} a(z) \beta\tens_A \left( L_\beta({}_{s(\beta)}\!v_z)-L_\beta({}_{s(\beta)}\! v)\delta_z \right)=\sigma(v\tens_A\extd a) \]
needs to be well-defined and requires that $\sigma({}_x\! v_y\tens_A \alpha)=\sigma({}_x\! v_y\tens \extd \delta_z)=-\sum_{\beta\in Q_1 \atop s(\beta)=x}\beta\tens L_{\beta}({}_x\! v_y)\delta_z$ for all ${}_x\! v_y$ in ${}_xE_{y}$ and $\alpha:y\to z$ in $\bar{Q}_1.$
\endproof

If we fix $\Omega^2$ then having zero curvature of a connection can be interpreted as a certain composition property of a some quiver representations according to this lemma. We are particularly interested in the case of a so-called `linear connection' where $E=\Omega^1$ in which case the \textit{torsion} is defined by
\[ T_\nabla:\Omega^1\to\Omega^2,\quad T_\nabla=\wedge\nabla-\extd.\]
Note that $T_\nabla$ and $R_\nabla$ are both left $A$-module maps by construction. A \textit{metric} is defined to be an element $g\in\Omega^1\tens_A\Omega^1$ together with a bimodule map $(\ ,\ ):\Omega^1\tens_A\Omega^1\to A$ such that
\begin{equation}\label{centralmetric}
g\bo\tens_A (g\bt,\omega)=\omega,\quad (\omega,g\bo)\tens_A g\bt=\omega,\quad \forall\,\omega\in\Omega^1,
\end{equation}
where $g=g\bo\tens g\bt.$ This is equivalent to saying that $\Omega^1$ is a left and right self-dual in the monoidal category of $A$-bimodules. As for standard differential calculi, this requires that $g$ be \textit{central} in $\Omega^1\tens_A\Omega^1,$ i.e. $ag=ga$ for any $a\in A.$
When there is a metric $g,$ we say that $\nabla$ is \textit{skew-metric-compatible} or \textit{cotorsion-free} if 
\begin{equation}\label{s-m-c}
\left(\extd\tens\id-(\wedge\tens\id)(\id\tens\nabla)\right)g=0.
\end{equation}
A torsion-free and skew-metric-compatible connection is called a \textit{weak Levi-Civita connection}. Such a connection does not necessarily exist and if it does may not be unique.

\begin{lemma}
The torsion of a bimodule connection on $\Omega^1$ is a bimodule map if  
\begin{equation*}
\im (\id+\sigma)\subseteq \ker (\wedge:\Omega^1\tens_A\Omega^1\to\Omega^2). 
\end{equation*} 
(we say that $\nabla$ is {\em torsion compatible} when this inclusion holds). 
\end{lemma}
\proof These are both constructed to be left module maps. Light computation shows
$T_\nabla(\omega a)-T_\nabla(\omega)a=\wedge(\id+\sigma)(\omega\tens_A\extd a)$ so we require the condition 
at least on exact forms $\extd a$. 
\endproof

Imposing this on all of $\Omega^1\tens_A\Omega^1$ is a little stronger on a generalised calculus compared to the standard case and hence and torsion-free (where zero is obviously a bimodule map) no longer implies torsion-compatible in the stronger sense stated.  Next, a bimodule connection $\nabla$ naturally extends to $1$-$1$-forms, namely
\begin{equation}\label{m-c}
\nabla g=(\nabla g\bo)\tens_A g\bt+(\sigma\tens\id)(g\bo\tens_A\nabla g\bt),
\end{equation}
for $g=g\bo\tens g\bt\in\Omega^1\tens_A\Omega^1,$ and we say that $(\nabla,\sigma)$ is \textit{metric-compatible} in the case of a metric $g$ if $\nabla g=0.$ A bimodule connection $(\nabla,\sigma)$ will be called \textit{Levi-Civita} if it is torsion-free, torsion-compatible and metric-compatible. If it exists it may not be unique and will typically have curvature. A bimodule connection $(\nabla,\sigma)$ will be called \textit{Maurer-Cartan} if it is metric-compatible and flat (has zero curvature). It this exists it may not be unique and will typically have torsion. If $\nabla$ is torsion-free and torsion-compatible, then (\ref{s-m-c}) reduces to $(\wedge\tens\id)\nabla g=0$ 
which is weaker than being metric-compatible and justifies the term `skew-metric-compatible' for cotorsion-freeness in the notion of a  weak Levi-Civita connection.

Let $A$ be an algebra and now let $\Omega^1$ be an inner generalised differential structure on $A$. The corresponding inner data may not unique in $\Omega^1$ but it is uniquely determined by some element $[\theta]\in\Omega^1/Z(\Omega^1),$ where 
$Z(\Omega^1)=\{\eta\in\Omega^1\big|\,a\eta=\eta a,\,\forall\,a\in A\}.$

\begin{lemma}\label{inner-con}
Let $A$ be an algebra and let $\Omega^1$ be an inner differential structure on $A$. Fix a representative $\theta
\in\Omega^1$ for the inner data.
\begin{itemize}
\item[1)] Bimodule connections $(\nabla,\sigma)$ are in one-to-one correspondence with pairs $(\sigma,\alpha)$
\[\sigma:\Omega^1\tens_A\Omega^1\to \Omega^1\tens_A\Omega^1,\quad\alpha:\Omega^1\to\Omega^1\tens_A\Omega^1\]
of bimodule maps and take the form
\begin{equation*}\label{innernabla}
\nabla\omega=\theta\tens_A\omega-\sigma(\omega\tens_A\theta)+\alpha\omega.
\end{equation*}
\item[2)] Given a pair $(\sigma,\alpha)$ as in 1), for $g\in\Omega^1\tens_A\Omega^1,$ we have
\[\nabla g=\theta\tens_A g-\sigma_{12}\sigma_{23}(g\tens_A\theta)+(\alpha\tens\id+(\sigma\tens\id)(\id\tens\alpha))g.\]
\item[3)] If $\Omega^1$ extends to $\Omega^2$ with $\theta\wedge\theta\in Z(\Omega^2)$ and $\extd \omega=\theta\wedge\omega+\omega\wedge\theta$ for all $\omega\in\Omega^1,$ then 
\begin{align*}
T_\nabla\omega&=-\wedge (\id+\sigma)(\omega\tens_A\theta)+\wedge\alpha\omega,\\
R_\nabla\omega&=\theta\wedge\theta\tens_A\omega+(\wedge\tens\id) \tilde{R}_\nabla\omega,\nonumber\\
\tilde{R}_\nabla\omega&=-\sigma_{23}\sigma_{12}(\omega\tens_A\theta\tens_A\theta)\nonumber\\
&{}\qquad+\left(\sigma_{23}(\alpha\tens\id)+(\id\tens\alpha)\sigma\right)(\omega\tens_A\theta)-(\id\tens\alpha)\alpha\omega.\nonumber
\end{align*}

\item[4)] In the case of 3), $(\nabla,\sigma)$ is torsion-free if $\wedge(\id+\sigma)\big|_{\Omega^1\tens_A\theta}=0$ and $\wedge\alpha=0.$ When $\theta\in\bar{\Omega}^1,$ the converse is also true.

\item[5)] Suppose the characteristic of the ground field is not 2, $\sigma(\theta\tens_A\theta)=\theta\tens_A\theta,$ $\alpha=0$ and $\theta\in\bar{\Omega}^1$. Then $(\nabla,\sigma)$ in part 1)  torsion-free and with $\sigma$ obeying the braid relations implies $R_\nabla=0$.
\end{itemize}
\end{lemma}
\proof This lemma is more or less \cite[Theorem 2.1]{Ma:gra}. For completeness, we still provide the sketch of the proof. 1) Given $(\nabla,\alpha),$ one can check that $(\nabla,\sigma)$ defined in 1)  indeed form a bimodule connection. Conversely, let $(\nabla,\sigma)$ be a bimodule connection. Clearly, $\nabla^0\omega=\theta\tens_A\omega-\sigma(\omega\tens_A\theta)$ is a connection and $\nabla-\nabla^0$ is a bimodule map, which is taken as $\alpha.$
2) The formulae displayed are straightforward computed from 1)  by definition. In fact,
\begin{align*}
\nabla g&=\nabla(g\bo)\tens_A g\bt+(\sigma\tens\id)(g\bo\tens_A\nabla g\bt)\\
&=\left(\theta\tens_A g\bo-\sigma(g\bo\tens_A\theta)+\alpha g\bo\right)\tens_A g\bt\\
&{}\quad +(\sigma\tens\id)\left( g\bo\tens_A\left(\theta\tens_A g\bt-\sigma(g\bt\tens_A\theta)+\alpha g\bt\right)\right)\\
&=\theta\tens_A g-\sigma_{12}\sigma_{23}(g\tens_A\theta)+\left(\alpha\tens\id+(\sigma\tens\id)(\id\tens\alpha)\right)g.
\end{align*}

3) It is straightforward to check that 
\begin{align*}
T_\nabla(w)&=\wedge(\theta\tens_A\omega-\sigma(\omega\tens_A\theta)+\alpha\omega)-\extd\omega\\
&=-\omega\wedge\theta-\wedge\sigma(\omega\tens_A\theta)+\wedge\alpha\omega\\
&=-\wedge(\id+\sigma)(\omega\tens_A\theta)+\wedge\alpha\omega,
\end{align*}
and
\begin{align*}
R_\nabla&=(\extd\tens\id-(\wedge\tens\id)(\id\tens\nabla))(\theta\tens_A\omega-\sigma(\omega\tens_A\theta)+\alpha\theta)\\
&=\theta\wedge\theta\tens_A\omega-\sigma^1\wedge\sigma(\sigma^2\tens_A\theta)+\sigma^1\wedge\alpha(\sigma^2)
+\alpha^1\wedge\sigma(\alpha^2\tens_A\theta)-\alpha^1\wedge\alpha(\alpha^2)\\
&=\theta\wedge\theta\tens_A\omega-(\wedge\tens\id)\sigma_{23}\sigma_{12}(\omega\tens_A\theta\tens_A\theta)\\
&{}\qquad+(\wedge\tens\id)\left(\sigma_{23}(\alpha\tens\id)+(\id\tens\alpha)\sigma\right)(\omega\tens_A\theta)-(\wedge\tens\id)(\id\tens\alpha)\alpha\omega,
\end{align*}
where we denote $\sigma^1\tens_A\sigma^2=\sigma(\omega\tens_A\theta)$ and $\alpha^1\tens\alpha^2=\alpha\omega$ is a shorthand notation.

4) The `if' statement is obvious from 3). We know that if $\nabla$ torsion-free, then $\wedge(\id+\sigma)\big|_{\Omega^1\tens_A\bar{\Omega}^1}=0.$ So when $\theta\in\bar{\Omega}^1$, we have $\wedge(\id+\sigma)\big|_{\Omega^1\tens_A\theta}=0.$  This together with the formula in 3) implies that $\wedge\alpha=0.$ Therefore the `only if' part is also true.

5) Clearly, $0=\wedge(\id+\sigma)(\theta\tens_A\theta)=2\theta\wedge\theta$ implies $\theta\wedge\theta=0, $ if the characteristic is not $2$. Therefore
\begin{align*}
R_\nabla\omega&=(\wedge\tens\id)\tilde{R}_\nabla\omega=-(\wedge\tens\id)\sigma_{23}\sigma_{12}(\omega\tens_A\theta\tens_A\theta)\\
&=-(\wedge\tens\id)\sigma_{23}\sigma_{12}\sigma_{23}(\omega\tens_A\theta\tens_A\theta)\\
&=-(\wedge\tens\id)\sigma_{12}\sigma_{23}\sigma_{12}(\omega\tens_A\theta\tens_A\theta)\\
&=(\wedge\tens\id)\sigma_{23}\sigma_{12}(\omega\tens_A\theta\tens_A\theta)\\
&=-R_\nabla\omega
\end{align*}
and hence vanishes under our assumptions. 
\endproof

According to the analysis above, we know that if $\theta'$ is another representative for the same inner calculus then $\Delta\theta=\theta'-\theta\in Z(\Omega^1)$ and $\alpha_{\Delta \theta}(\omega)=\Delta\theta\tens_A\omega-\sigma(\omega\tens_A\Delta\theta)$ is an $A$-bimodule map. It is easy to see that $(\theta,\sigma,\alpha)$ and $(\theta',\sigma,\alpha-\alpha_{\Delta\theta})$ provide the same bimodule connection, thus the inner form in Lemma~\ref{inner-con}  of a bimodule connection does not depend on the choice of represetative of inner data. Moreover, if $\Delta\theta\wedge\omega+\omega\wedge\Delta\theta=0$ for all $\omega\in\Omega^1$, then $\extd\omega=\theta\wedge\omega+\omega\wedge\theta=\theta'\wedge\omega+\omega\wedge\theta'.$ In this case, the bimodule connections associated to $(\theta,\sigma,\alpha)$ and $(\theta+\Delta\theta,\sigma,\alpha-\alpha_{\Delta\theta})$ have the same torsion and curvature. 

Finally, the geometric Laplace-Beltrami operator given a connection $\nabla$ on $\Omega^1$ and an `inverse metric' bimodule map $(\ ,\ )$ is defined to be $\Delta=(\ ,  )\nabla\extd$. This may or may not coincide with a 2nd order operator $\Delta_\theta a=2(\theta,\extd a)$ defined for any inner calculus and obeying $\Delta_\theta(ab)=(\Delta_\theta a)b+2 (\extd a,\extd b)+a\Delta_\theta b$. These notions are the same as the standard case in \cite{Ma:gra}.

\begin{example} Let $X$ be a finite set and $\Omega^1(\bar Q,Q)$ the calculus on $A=k(X)$ associated to a digraph-quiver pair with vertex set $X$ which is inner as in (\ref{quivcalc}). A bimodule connection on $\Omega^1$ by the above is given by (i) a bimodule map $\alpha$ which we can write as $\alpha=\sum_{\gamma\in Q_1}\gamma\tens_A N_\gamma$ for $N_\gamma:{}_{s(\gamma)}\Omega^1_y\to {}_{t(\gamma)}\Omega^1_y$ for all $y\in X$ and (ii) a bimodule map $\sigma$ which we can write as $\sigma((\ )\tens_A\beta)=\sum_{\gamma\in Q_1}\gamma\tens_A N_\gamma^\beta$ for $N_\gamma^\beta:{}_{s(\gamma)}\Omega^1_{s(\beta)}\to {}_{t(\gamma)}\Omega^1_{t(\beta)}$ for all $\beta\in Q_1$. Then the associated quiver representation is $L_\gamma=N_\gamma - \sum_{\alpha\in \bar Q_1}N_\gamma^\alpha$. 
One can check that we recover $\sigma((\ )\tens_A\alpha)$ from the formula in Lemma~\ref{quiverrep}. 
\end{example}

\subsection{Geometry of $\C S_3$}

We specialise the general theory of metrics and connections above to the case of $\C S_3$  with generators $u,v$ with $u^2=v^2=e$ and $uvu=vuv$, based on the natural construction in Lemma~\ref{CG}. We construct calculi using Lemma~\ref{CG} from representations of the group: (i) the trivial representation of $S_3$ gives the zero calculus; (ii) the sign representation $\rho(u)=\rho(v)=\rho(w)=-1$ where $w=uvu$ (and others $+1$) gives a 1-dimensional calculus with generator $\theta$ and relations
\[ \theta x= \rho(x)x\theta,\quad \extd u=-2  u\theta,\quad\extd v=-2v\theta,\quad \extd w=-2 w\theta\]
and zero otherwise. This is a standard calculus but not connected as $\extd(uv)=\extd(vu)=0$. There is a quantum metric $g=\theta\tens\theta$ which is clearly central and nondegenerate. The associated $\Delta_\theta x=2(\theta,\theta x-x\theta)=2(1-\rho(x)x$ so is diagonalised by conjugacy classes with eigenvalues 4 for $u,v,w$ and 0 for $uv,vu$ and $e$; (iii) finally we have the 2-dimensional representation which we take in the form
\[ \rho(u)=\begin{pmatrix}1&0\cr 0 & -1\end{pmatrix},\quad \rho(v)={1\over 2}\begin{pmatrix}-1  &\sqrt{3}  \cr \sqrt{3}& 1\end{pmatrix}. \]
To get a feel for this differential calculus we let
\[ e_u=u^{-1}\extd u=\begin{pmatrix}0&0\cr 0 & -2\end{pmatrix},\quad e_v=v^{-1}\extd v={1\over 2}\begin{pmatrix}-3&\sqrt{3}\cr \sqrt{3} & -1\end{pmatrix}\]
\[ e_{uv}=(uv)^{-1}\extd(uv)={1\over 2}\begin{pmatrix}-3&\sqrt{3}\cr -\sqrt{3} & -3\end{pmatrix},\quad e_{vu}=(vu)^{-1}\extd (vu)={1\over 2}\begin{pmatrix}-3&-\sqrt{3}\cr \sqrt{3} & -3\end{pmatrix},\]
which form a basis, so this is a standard calculus. Here $e_u+e_v+e_w=e_{uv}+e_{vu}=-3\theta$ where $e_w$ is defined similarly. One has relations
\[ e_u u=-u e_u,\quad e_u v=v( e_{uv}-e_v),\quad e_v u=u(e_{vu}-e_u),\quad e_v v=-v e_v\]
\[ e_{uv}u=u(e_w-e_u),\quad e_{uv}v=v(e_u-e_v),\quad e_{vu}u=u(e_v-e_u),\quad e_{vu}v=v(e_w-e_v).\]
Here the adjoint action of $S_3$ on $\C S_3$ is by permutation of $u,v,w$ and the calculus is covariant under this, in addition to the bicovariance with respect to the Hopf algebra coproduct.  Indeed, the calculus from Lemma~\ref{CG} is necessarily covariant under the adjoint action which here is generated by $\Ad_w$ which swaps $e_u,e_v$ and $e_{uv},e_{vu}$) and by $\Ad_{uv}$ which cyclically rotates $e_u\to e_v\to e_w\to e_u$. We look at the noncommutative geometry in more detail in this calculus, with focus on ad-invariant geometries. The exterior algebra $\Omega$ is generated by the above and the $\{e_i\}$ anticommutative amongst themselves.

\begin{proposition} There is a quantum metric on $\Omega^1(\C S_3)$ given by
\[ g=e_u\tens e_u+e_v\tens e_v+e_w\tens e_w+e_{uv}\tens e_{uv}+e_{vu}\tens e_{vu}- 6\theta\tens\theta,\]
which is also invariant under the adjoint action of the group. There is a unique ad-invariant quantum Levi-Civita connection, namely $\sigma={\rm flip}$ and $\nabla e_i=0$. More generally, there is a 1-parameter family of quantum torsion free cotorsion free ad-invariant connections with $\sigma={\rm flip}$, namely 
\[ \nabla e_u=\lambda(\begin{pmatrix}e_u & e_v\end{pmatrix}\begin{pmatrix}-1 & 1 \cr 1 & 2\end{pmatrix}\begin{pmatrix}e_u\cr e_v\end{pmatrix} +3(e_v\tens\theta+\theta\tens e_v)+e_{uv}\tens e_{vu}+e_{vu}\tens e_{uv})\]
\[ \nabla e_v=\lambda(\begin{pmatrix}e_u & e_v\end{pmatrix}\begin{pmatrix}2 & 1 \cr 1 & -1\end{pmatrix}\begin{pmatrix}e_u\cr e_v\end{pmatrix} +3(e_u\tens\theta+\theta\tens e_u)+e_{uv}\tens e_{vu}+e_{vu}\tens e_{uv})\]
\[ \nabla e_{uv}=\lambda((e_{vu}-e_{uv})\tens e_{uv}+ e_{uv}\tens e_{vu}),\quad \nabla e_{vu}=\lambda((e_{uv}-e_{vu})\tens e_{vu}+e_{vu}\tens e_{uv}).\]
This connection has Riemann curvature
\[ R_\nabla e_u=\lambda^2(e_{uv}e_{vu}\tens(e_{vu}-e_{uv})+3(e_ue_v+(e_u-e_v)\theta)\tens (e_v-e_w))\]
\[ R_\nabla e_v=\lambda^2(e_{uv}e_{vu}\tens(e_{vu}-e_{uv})+3(e_ue_v+(e_u-e_v)\theta)\tens (e_w-e_u))\]
\[ R_\nabla e_{uv}=\lambda^2 e_{uv}e_{vu}\tens(2e_{vu}-e_{uv}),\quad R_\nabla e_{vu}=-\lambda^2e_{uv}e_{vu}\tens (2e_{uv}-e_{vu}).\]
and $\Delta=\Delta_\theta$ for all $\lambda$. Allowing general $\sigma$, there is a 16-parameter moduli of torsion and cotorsion free connections of which a 3-parameter part is ad-invariant. \end{proposition}
\proof Let 
\[ g_{u,v}:=(e_{uv}-e_u)\tens(e_{uv}-e_u)+(e_{vu}-e_v)\tens(e_{vu}-e_v)+e_{uvu}\tens e_{uvu}\in \Lambda^1\tens\Lambda^1\]
It is clear that $g_{u,v}$ is central and quantum symmetric.  Then the element $g={1\over 3}(g_{u,v}+g_{v,w}+g_{w,u})$ is a quantum metric and comes out as stated and is ad-invariant. We then check it is nondegenerate. Indeed, its inverse  bimodule inner product on basis elements in the order $\{ e_u,e_v,e_{uv},e_{vu}  \}$ is
\[ (e_.,e_.)=
 \begin{pmatrix}{4\over 3}& {1\over 3}&1&1\cr {1\over 3}&{4\over 3}&1&1\cr 1&1&2&1\cr 1&1&1&2\end{pmatrix}
\]

We next define the matrices $\phi, \psi$ for right action of $u,v$ respectively (by multiplication by $\rho(u),\rho(v)$) which on our basis $\{e_i\}$ come out as
\[ \phi=\begin{pmatrix}-1&0&0&0\cr -1 & 0 & 0 & 1\cr -2 & -1& 1 & 1\cr -1 & 1 & 0&0\end{pmatrix},\quad 
\psi=\begin{pmatrix}0&-1&1&0\cr 0 & -1 & 0 & 0\cr 1 & -1 & 0 & 0\cr -1 & -2 & 1&1\end{pmatrix}\]
and we define $\tau$ by $\sigma(e_i\tens e_j)=e_j\tens e_i+ e_m\tens e_n \tau{}_{ij}{}^{mn}$. We solve for $[(\phi\tens\phi),\tau]=0, [(\psi\tens\psi),\tau]=0$ where $(\phi\tens\phi)_{ij}{}^{mn}=\phi_i{}^m\phi_j{}^n$ and $\tau$ are regarded as 16$\times$16 matrices for $ij, mn$ as multiindices. Similarly we require $\phi.\alpha=\alpha.(\phi\tens\phi)$  as matrices (and similary for $\psi$) where $\alpha$ is regarded as $4\times 16$. In each case we suppose $\alpha$ and $\tau$ are symmetric in their last two indices since we are interested in the torsion free case. This gives our moduli of torsion free connections as 28-dimensional and 6-dimensional respectively for $\tau$ and $\alpha$. The former are conveniently parametrized by the 15 values of $\tau_{ii}{}^{jj}$ other than $\tau_{33}{}^{11}$ (say) and a further 13 values. When we further demand the cotorsion equations that $\tau.\theta$ and $\alpha$ are totally symmetric in their three indices, these
get cut down to 12 and 4-dimensional spaces respectively. The 12-dimensional moduli space is conveniently parametrized by $\tau_{ii}{}^{jj}$ except for the four values $\tau_{33}{}^{11}, \tau_{44}{}^{11},\tau_{44}{}^{22},\tau_{44}{}^{33}$. It suffices for full metric compatibility to restrict to this torsion-free cotorsion free moduli space. Next, we similarly impose invariance under matrices $P,Q$ for the adjoint action of $w, uv$ respectively, i.e. we look for Ad-invariant solutions, giving a 2-parameter moduli for $\tau$ and 1-parameter for $\alpha$ of torsion free cotorsion free invariant connections. Setting $\tau=0$ focusses on the 1-dimensional moduli space stated and one can check that this is not fully metric compatible unless $\lambda=0$. One can check that there is no other fully metric compatible connection in the full 16 parameter moduli either. The quadratic conditions on $\tau$ here were checked using MATHEMATICA and apply over $\C$ for this reason. 

We now study this 1-dimensional parameter space further. In fact the parameter enters as a scaling of $\alpha$ which we separate off, so we write $\nabla e_i=\lambda\alpha(e_i)$ where $\alpha$ denotes the displayed right hand side of $\nabla$ corresponding to $\lambda=1$. Note that while the equations for $\alpha$ are linear in $\lambda$, the connection is not, for example $\nabla(ue_i)=\extd u\tens e_i + u\lambda\alpha(e_i)$. As an illustrative check let us verify that this indeed gives a bimodule connection. Thus for example
\begin{eqnarray*}&&\kern-15pt\nabla(e_u v)=\nabla( v (e_{uv}-e_v))=\extd v\tens (e_{uv}-e_v)+v\nabla(e_{uv}-e_v)\\
&&=\sigma(e_u\tens\extd v)+v(-e_{uv}\tens e_{uv}-2e_u\tens e_u+e_v\tens e_v-e_v\tens e_u-e_u\tens e_v-3e_u\tens\theta-3\theta\tens e_u)\end{eqnarray*}
and we want this to equal $\sigma(e_u\tens \extd v)+ (\nabla e_u).v$ for a bimodule connection. Here $\sigma={\rm flip}$ on the basis means $\sigma(e_u\tens \extd v)=-\sigma(e_u\tens e_v)v=-e_v\tens e_u v=ve_v\tens (e_{uv}-e_v)=\extd v\tens (e_{uv}-e_v)$ and meanwhile
\begin{eqnarray*} (\nabla e_u).v&=&v( -(e_{uv}-e_v)\tens(e_{uv}-e_v)+2e_v\tens e_v -e_v\tens(e_{uv}-e_v)-(e_{uv}-e_v)\tens e_v\\
&&-3e_v\tens\theta-3\theta\tens e_v-6e_v\tens e_v+(e_u-e_v)\tens (e_w-e_v)+(e_w-e_v)\tens (e_u-e_v))\end{eqnarray*}
using the form of $\nabla e_u$, the bimodule relations for the calculus and $\theta v=v\theta+v e_v$. Expanding out $-3\theta=e_{uv}+e_{vu}$ and 
$e_w=e_{uv}+e_{vu}-e_u-e_v$, this equates to  the previous expression from $\nabla(e_{uv}-e_v)$. One can similarly check this for $u$ and for the other 3 basis elements of $\Lambda^1$ as well as metric compatibility for the stated quantum metric. Finally, $\tau=0$ means that $\Delta=\Delta_\theta$ on functions. Since $\tau=0$ we can also write the latter as $\Delta_\theta=g_{ij}\del^i\del^j$ where $\extd a=\sum_i (\del_ia)e_i$ for all $a\in \C S_3$ and $g=\sum_{i,j}g_{ij}e_i\tens e_j$.
\endproof

One can canonically lift the 2-form values of $R_\nabla$ by lifting products of the $e_i$ antisymmetrically,  and then take a trace via the metric and inverse metric to obtain the {\rm Ricci} curvature. This comes out for the stated 1-parameter moduli of connections as 
\[ {\rm Ricci}=0.\]
So this noncommutative geometry is Ricci-flat but not flat when $\lambda\ne 0$. 

Moreover, the  `Laplacian'  $\Delta=\Delta_\theta$ is diagonal on the given basis of $\C S_3$ with $u,v,w$ having eigenvalue 4/3 and $uv,vu$ eigenvalue 2 and $e$ having eigenvalue 0, i.e. the eigenspaces are the spans of conjugacy classes. This is the reverse of the situation for the noncommutative geometry of $\C(S_3)$ where a choice of conjugacy class provides a natural 3D calculus and with respect to it  $\Delta_\theta$ is diagonalised by matrix elements of the irreducible representations. Here the trivial class gives the zero calculus $\{e\}$, $\{uv,vu\}$ gives a calculus which is not connected and $\{u,v,w\}$ gives the standard 3D calculus afforded by the Bruhat graph. In the latter case one again has a natural metric $g=e_u\tens e_u+e_v\tens e_v+e_w\tens e_w$ and (2-parameter) moduli of torsion free cotorsion free connections $\nabla$ in \cite{Ma:rieq} for which one can check that $\Delta=\Delta_\theta$ is diagonalised by the irreducibles of $S_3$  with the 2D representation matrix elements having eigenvalue  6, the sign representation having eigenvalue  12   and the trivial representation eigenvalue 0. There is a similar picture for the quantum geometry defined by the other nontrivial class $\{uv,vu\}$ where the eigenvalues are 0 on the trivial and sign representation and 8 on the 2D representation matrix elements (the general formula for the eigenvalues of $\Delta_\theta$ for a finite group is $2(1-\chi(C))|C|$ where $C$ is a conjugacy class and $\chi$ is the normalised character of an irreducible representation evaluated on an element of $C$, cf. \cite{Ma:gra}). This indicates a remarkable duality between the noncommutative geometry of $k G$ and $k(G)$  extending the classical duality between characters and conjugacy classes and illustrated above for $S_3$. We have focussed on standard calculi but expect a similar duality to hold at the quiver level as well.

\subsection{Quivers from finite homogeneous spaces}

A nice thing about quiver calculi on  sets (as opposed to standard ones given by digraphs) is that this category is closed under quotients by a group action. If $(X=Q_0, Q_1,s,t)$ is a quiver then an action of a group $G$ means an action 1on both $Q_0,Q_1$ with $s,t$ covariant, i.e. for every arrow $x\to y$, $(x\to y)^g$ is some arrow $x^g\to y^g$ where $(\ )^g$ denotes the action of $g\in G$.  Even if $X$ has a digraph calculus and $G$ acts freely, this in general still gives a quiver calculus on $X/G$  (we will give an example below). To fix this one can define the quotient graph on $X/G$ as one where there is an arrow between distinct orbits if and only if there exists any arrow between a representative of one to a representative of the other. 

The Hopf algebra version of this is that if $P$ is a right $H$-comodule algebra with coaction $\Delta_R$ and $\Omega^1(P)$ a right-covariant standard differential calculus and we set $A=P^H$ (the coinvariant subalgebra) then $\Omega^1(P)^H$ is a not-necessarily standard differential calculus on $A$. This is because $\extd: P\to \Omega^1(P)$ restricts to the invariant subalgebra and clearly inherits all required properties except surjectivity. To fix this one defines the standard calculus $\Omega^1(A)\subseteq \Omega^1(P)^H$ as the $A$-subimodule generated by $\extd A$ within $\Omega^1(P)$. This inherited calculus on $A=k(X/G)$ in the set case agrees with the quotient graph because  the forms on $A$ are by definition generated by characteristic functions $\delta_i$ on $X$ where $\delta_i$ is a constant $1$ on orbit $\mathcal{O}_i$ and zero elsewhere, and their differentials
\[ \extd \delta_i=\Big(\sum_{x\to y, x\notin \mathcal{O}_i,y\in \mathcal{O}_i}x\to y\Big)-\Big(\sum_{x\to y, x\in \mathcal{O}_i, y\notin \mathcal{O}_i} x\to y\Big)=\sum_{j\ne i}e_{j\to i}-e_{i\to j}\]
where we only include $e_{i\to j}=\sum_{x\to y, x\in \mathcal{O}_i, y\in \mathcal{O}_j} x\to y$ if some arrow from $\CO_i$ to $\CO_j$ exists in the digraph on $X$.

The above data is encountered naturally in the theory of quantum principal bundles in standard noncommutative geometry. We recall that a quantum principal bundle means such a right $H$-comodule algebra $P$ with differential structure such that\cite{BrzMa:gau}
\begin{equation}\label{seqP} 0\to P\Omega^1(A)P\hookrightarrow \Omega^1(P)\xrightarrow[]{\rm ver} P\tens\Lambda^1_H\to 0\end{equation}
is exact, where $\Omega^1(H)=H\tens\Lambda^1_H$ is bicovariant and the map ${\rm ver}(u\tens v)=u\Delta_Rv$ at the level of universal calculus $\Omega^1_{univ}P\subseteq P\tens P$, which we assume descends to a given right-covariant calculus $\Omega^1(P)$. Such data can be obtained first by solving at the level of the universal calculus on all our algebras, in which case exactness of the above is equivalent to the similarly defined map
\[ {\rm ver}:P\tens_A P\to P\tens H\]
being an isomorphism (i.e. a Hopf-Galois extension) and then requiring that our data descends correctly. A connection on the quantum principal bundle is an equivariant splitting of ${\rm ver}$, which amounts to a right comodule  map $\omega:\Lambda^1_H\to \Omega^1(P)$ such that ${\rm ver}\omega(v)=1\tens v$ for $v\in \Lambda^1_H$. Here $H$ coacts on $\Lambda^1_H$ by the right adjoint coaction when the latter is identified as a quotient of $H^+$. This gives an equivariant left $P$-module projection $\Pi_\omega$ on $\Omega^1(P)$ with kernel the horizontal forms $P\Omega^1(A)P$, splitting (\ref{seqP}). Taking the coinvariants of this at least in nice cases where $P\Omega^1(A)P=P\Omega^1(A)$ gives us $\Omega^1(P)^H=\Omega^1(A)\oplus E$ as left $A$ modules, where $E=(P\tens\Lambda^1_H)^H$ is an associated bundle according to the theory in \cite{BrzMa:gau} (classically it would be the sections of the coadjoint bundle with fibre the dual of the Lie algebra). 

Now in our case $P=k(X)$ where $X$ is a digraph and its arrows determine $\Omega^1(P)$. We take $H=k(G)$ for a finite group and $\Omega^1(H)$ given by the Cayley graph of an ad-stable subset $\bar C$, i.e. with vertices $G$ and arrows $x\to xa$ for $a\in \bar C\subseteq G\setminus\{e\}$. We work with a basis of $\Lambda^1_H$ given by $e_a=\sum_{g\in G}g\to ga$ for $a\in \bar C$. To meet our covariance condition we require that the digraph of $X$ is covariant under the action of $G$ as discussed above. 

\begin{proposition}\label{prinG} A finite group $G$ with Cayley graph given by $\bar C$ acting on a finite digraph $X$  gives a quantum principal bundle if and only if

(i) Each orbit has cardinality $|G|$  

(ii) The graph within each orbit in $X$ has valency $|\bar C|$ with arrows of the form $x\to x^a$, $a\in \bar C$. 

In this case there is a canonical connection $\omega(e_a)=\sum_x x\to x^a$ for all $a\in \bar C$ and if there is at most one arrow from any $x\in X$ to any given different orbit from the one containing $x$ then there is a splitting
\[ \Omega^1(X)^G\isom \Omega^1(X/G)\oplus E, \quad E=k(X/G,\Lambda^1_H)\]
as $A$-modules. The associated quiver then consists of the quotient digraph plus self-loops labelled by $\bar C$. 
\end{proposition}
\proof For a Hopf Galois extension or quantum bundle at the universal level we want the action of $G$ to be free, which happens if and only if each orbit $\CO_i$ has cardinality $|G|$, which in turn means each orbit is bijective with $G$ on fixing a basepoint $x_i$. In this case we have a bijection $(X/G)\times G\to X$ by $i\times g\mapsto x_i^g$ and hence at least when $X$ finite that $P\isom A\tens H$ as an algebra. This means that $P$ has a trivialisation 
\[ \Phi(\delta_g)=\sum_{i\in X/G}\delta_{x_i^g},\quad \Phi^{-1}(\delta_g)=\sum_{i\in X/G}\delta_{x_i^{g^{-1}}}\]
which one can check obeys the required conditions \cite{BrzMa:gau}.  This also provides a flat connection from the trivialisation at least at the level of the universal calculus though this is not relevant at the moment. In our case we will see that we have a quotient of the universal bundle calculus if and only if the graph on $X$ restricted within each orbit has valency $|\CC|$ with arrows $x\to x^a$ for $a\in \CC$, i.e. if and only if each orbit is not only isomorphic to $G$ as a set but as a digraph obtained by restriction from $X$ (dropping any arrows that do not lie within the orbit). Indeed, 
\[ \ver(x\to y)=\delta_x\sum_{a\in \mathcal{C}} \delta_{y^{a^{-1}}}\tens e_a=\sum_{a\in\mathcal{C}, y=x^a}\delta_x \tens e_a\]
where we use the definition in terms the coaction $\Delta_R$ projected down to $\Lambda^1_H$. By freeness, there is only one possible $a$ that can contribute so $\ver(x\to x^a)=\delta_x\tens e_a$ and zero on arrows not of this form. Meanwhile, the horizontal forms $P\,\Omega^1_AP$ are spanned by all arrows from one orbit to another (since these are picked out by $\delta_xe_{i\to j}\delta_y$ as we vary $x,y\in X$). If we want this to be exactly the kernel of $\ver$ then we need all arrows $x\to y$ where $x,y$ are in the same orbit to be exactly those of the form $x\to x^a$ for some $a\in \CC$. This proves our assertion about the required structure of the digraph on $X$. For the canonical flat connection form we let
\[ \omega(e_a)=\sum_{x\in X} x\to x^a\]
and check that  $\ver\omega(e_a)=\sum_x \delta_x\tens e_a=1\tens e_a$, and \[  \Delta_R\omega(e_a)=\sum_{g,x} x^g\to x^{ag}\tens\delta_g=\sum_{x',g}x'\to x'{}^{g^{-1}ag}\tens\delta_g=\sum_g \omega(e_{g^{-1}ag})\tens\delta_g\]
as required. This connection provides an equivariant left $P$-module map splitting of (\ref{seqP}) via the projection
\[ \Pi_\omega(x\to y)=\cdot(\id\tens\omega)\ver(x\to y)=\sum_{a\in\mathcal{C}, y=x^a}x\to y\]
which is the identity if $x,y$ are in the same orbit and zero otherwise. Meanwhile $P\Omega^1(A)$ is spanned by elements of the form $\delta_xe_{i\to j}=\sum_{y\in\CO_j}x\to y$ where $x\in\CO_i$ so this coincides with $P\Omega^1(A)P$ if and only if each of these sums has only one term which is the condition stated.  In this case taking  $G$-invariants of the splitting provided by $\Pi_\omega$ gives the stated splitting with $E\isom k(X/G)\tens\Lambda^1_H$ since the bundle is trivial\cite{BrzMa:gau}.  Here $\Lambda^1_H$ has basis $\{e_a\}_{a\in\bar C}$ from which the quiver associated to $\Omega^1(X)^G$ is then clear.  \endproof

Note that the flat connection here need not be `strong' in the sense needed for the theory of associated bundles and likewise the condition $P\Omega^1(A)P=P\Omega^1(A)$ need not hold. For example, it typically does not hold for universal calculi on $X$ and $G$.  Finally, there is an important case of a quantum bundle namely a finite homogeneous space. Let $G\subseteq X$ be a nontrivial subgroup of a finite group $X$ and $\bar C,\bar C_X$ define respectively bicovariant and left-covariant calculi (so $\bar C_X$ is any subset of $X\setminus\{e\}$ while $\bar C$ is an ad-stable subset of $G\setminus\{e\}$). Then $X\to X/G$ by the above gives a quantum principal bundle if and only if 
\begin{equation} \label{XmodG} \bar C=\bar C_X\cap G.\end{equation}
Even in this case the further condition for the splitting need not hold and one can have quiver calculi $\Omega(X)^G$ with multiple arrows between vertices in $X/G$. We give an example where the condition does hold.

\begin{example} We take $X=S_3$ with its standard 3D calculus afforded by the Cayley graph of $\bar C_X=\{u,v,w\}$ and $G=\Z_2$ generated by $u$  with its universal calculus afforded by the complete graph on two elements (here $\bar C=\{u\}$) and acting by right translation (which we can view as a right coaction of $H=k(\Z_2)$ on $P=k(S_3)$). The set $X/G$ consist of the orbits $\CO_0=\{e,u\}=\Z_2, \CO_1=\{v,vu\}=v\Z_2, \CO_2=\{uv,w\}=w\Z_2$. The graph on $X$, which is the bidirected Bruhat graph, has 18 arrows and meets the conditions of Proposition~\ref{prinG} as we can see from (\ref{XmodG}). The quotient quiver corresponding to calculus $\Omega(X)^G$ therefore has half as many, i.e. 9 arrow and in accordance with Proposition~\ref{prinG} is given by
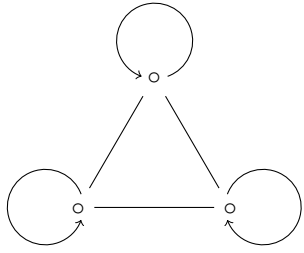
\[
\begin{tikzpicture}[descr/.style={fill=white},text height=1.5ex, text depth=0.25ex, arrow/.style={->}]
\node (1) at (0, 0) {$\circ$};
\node (2) at (2, 0) {$\circ$};
\node (3) at (1, 1.732) {$\circ$};
\path[-,font=\scriptsize,>=angle 90]
(1) edge node[]{} (2)
(2) edge node[]{} (3)
(3) edge node[]{} (1);
\draw [arrow] (1.west) ++(+0.25cm,5pt) arc [start angle=20, end angle=340, radius=0.5cm];
\draw [arrow] (2.east) ++(-0.25cm,5pt) arc [start angle=160, end angle=-160, radius=0.5cm];
\draw [arrow] (3.north) ++(5pt,-0.25cm) arc [start angle=-70, end angle=250, radius=0.5cm];
\end{tikzpicture}
\]
where the unmarked edges should be read as directed both ways. These form an equilateral triangle or complete digraph on 3 points (the universal calculus) as the standard calculus $\Omega(X/G)$ of the base. The condition for the splitting holds, for example $e\to v$ holds in $X$ and connects different orbits and the only other arrows from $e$ are $e\to u$ which lands in the same orbit as $e$ and $e\to w$ which lands in a different orbit from $v$.  \end{example}

Note that our context was quantum principal bundles with standard differentials and it was not our purpose here to study quantum principal bundles with general quiver calculi though this may certainly be done. 

\subsection{Quiver metrics} We now turn to elements of actual noncommutative geometry on quiver calculi. We start with a finite set $Q_0$ and let $Q=(Q_0,Q_1)$ be a coloured quiver. Define the index set of arrows in $Q$ by $E_Q:=\{x\to y\,|\, x\xrightarrow{(i)} y\in Q_1, \text{ for some } i\,\},$ and we call element $x\to y$ in $E_Q$ \textit{index arrow} though it is not a real arrow in $Q.$
Denote $E'_Q$ the subset of $E_Q$ consisting of all the arrows which have opposite arrow in $Q$, i.e. $E'_Q=\{x\to y\in E_Q\,|\,y\to x\in E_Q\}.$ We call a quiver is \textit{symmetric} if for arbitrary vertices $x$ and $y$ in $Q_0,$ the number of arrows from $x$ to $y$ equals the number of arrows from $y$ to $x.$ In this case, $E'_Q=E_Q.$

\begin{proposition} Let $Q=(Q_0,Q_1)$ be a colored quiver, the associated differential structure $\Omega^1=kQ_1$ on $A=k(Q_0)$ admits a central metric if and only if the quiver is \textit{symmetric}. The metric takes the form
\begin{equation}
g=\sum_{x\to y\in E_Q}\sum_{i,j=1}^{n(x,y)} g_{x\to y}^{ij} x\xrightarrow{(i)} y\xrightarrow{(j)} x,\quad (y\xrightarrow{(j)} x, x'\xrightarrow{(k)} y')=(g_{x\to y})^{-1}{}_{jk}\delta_{x,x'}\delta_{y,y'}\delta_y,
\end{equation}
where $g_{x\to y}=(g_{x\to y}^{ij})$ is an arbitrary $n(x,y)\times n(x,y)$ invertible matrices associated to index arrow $x\to y$ with $n(x,y)$ the number of arrows from $x$ to $y$ in $Q.$
\end{proposition}
\proof
For the quiver calculus $\Omega^1=kQ_1$ on $A=k(Q_0)$, we know $\Omega^1\tens_A\Omega^1=kQ_2$ is spanned by elements of form $x\xrightarrow{(i)} y\xrightarrow{(j)} z.$ A central metric is an element in $kQ_2$ commuting with functions. Therefore the general form of metric is
\[g=\sum_{x\to y\in E'_Q}\sum_{i=1}^{n(x,y)}\sum_{j=1}^{n(y,x)} g_{x\to y}^{ij}x\xrightarrow{(i)} y\xrightarrow{(j)} x,\quad g_{x\to y}^{ij}\in k,\]
where $n(x,y)=\dim k({}^xQ_1{}^y).$
Thus for each index arrow $x\to y\in E'_Q,$ we assign an $n(x,y)\times n(y,x)$ matrix  $g_{x\to y}=(g_{x\to y}^{ij}).$ 

Let $(\ ,\ ):\Omega^1\tens_A\Omega^1\to A$ be a bimodule map. Then $(\ ,\ )$ must be in form
\[(y\xrightarrow{(j)} x, x'\xrightarrow{(k)} y')=\lambda_{y\to x}^{jk}\delta_{x,x'}\delta_{y,y'}\delta_y,\quad \lambda_{y\to x}^{jk}\in k,\]
where $\lambda_{y\to x}=(\lambda_{y\to x}^{jk})$ is $n(y,x)\times n(x,y)$ matrix.

For each index arrow $x\to y$ in $E'_Q,$  (\ref{centralmetric}) forces
\[g_{x\to y}\lambda_{y\to x}=I_{n(x,y)},\quad \lambda_{y\to x} g_{x\to y}=I_{n(y,x)}.\] This shows that $n(x,y)=n(y,x),$ and matrices $g_{x\to y},$ $\lambda_{y\to x}$ are invertible and inverse to each other. However, (\ref{centralmetric}) is not obeyed for index arrow who does not have an opposite arrow, therefore $E'_Q=E_Q$ and the quiver must be symmetric.
\endproof

We now compute the Riemannian geometry of $4$D inner generalised differential calculus of $A=k(\Z_2)$ associated to the following quiver.
\[\begin{tikzpicture}[descr/.style={fill=white},text height=1.5ex, text depth=0.25ex]
\node (a) at (0,0) {${\circ\atop e}$};
\node (b) at (2.5,0) {${\circ\atop g}$};
\path[->,font=\scriptsize,>=angle 90]
([yshift= 18pt]a.east) edge node[above] {$\alpha_1$} ([yshift= 18pt]b.west)
([yshift= 10pt]a.east) edge node[descr] {$\alpha_2$} ([yshift= 10pt]b.west)
([yshift= -2pt]b.west) edge node[descr] {$\beta_1$} ([yshift= -2pt]a.east)
([yshift=-10pt]b.west) edge node[below] {$\beta_2$} ([yshift=-10pt]a.east);
\end{tikzpicture}\] 
Recall that $\Omega^1=kQ_1$ is spanned by all the arrows ${\alpha_1,\alpha_2,\beta_1,\beta_2}$ with $A$-bimodule structure and derivative given by
\begin{gather*}
\delta_e\alpha_i=\alpha_i=\alpha_i\delta_g,\quad\delta_g\alpha_i=0=\alpha_i\delta_e,
\quad\delta_g\beta_i=\beta_i=\beta_i\delta_e,\quad\delta_e\beta_i=0=\beta_i\delta_g;\\
\extd \delta_e=[\theta,\delta_e]=\beta_1-\alpha_1,\quad \extd \delta_g=[\theta,\delta_g]=\alpha_1-\beta_1,
\end{gather*}
where inner data $\theta=\alpha_1+\beta_1.$ Here $1$-forms are generated by $e^{(i)}=\alpha_i+\beta_i,\,i=1,2$ with relations
\[e^{(i)}f=(R_gf) e^{(i)},\quad i=1,2,\quad \extd f=(\partial f) e^{(1)},\] 
where $\partial=R_g-\id$ and $R_g$ is the right translation $R_g f=f(g)\delta_e+f(e)\delta_g$. The canonical `Woronowicz' differential graded algebra is the usual Grassmann algebra on $e^{(i)}.$ In particular, $\Omega^2=\mathrm{span}\{\alpha_1\beta_2=-\alpha_2\beta_1,\beta_1\alpha_2=-\beta_2\alpha_1\}$ with wedge product and differentials
\begin{gather*}
\alpha_i\wedge\beta_j=\alpha_i\beta_j=-\alpha_j\wedge\beta_i,\quad \beta_i\wedge\alpha_j=\beta_i\alpha_j=-\beta_j\wedge\alpha_i\\
\alpha_i\wedge\beta_i=0=\beta_i\wedge\alpha_i,\quad i,j=1,2,\ i\neq j.\\
\extd \alpha_1=0=\extd \beta_1,\quad\extd\alpha_2=\beta_1\alpha_2-\alpha_1\beta_2=-\extd\beta_2.
\end{gather*}

The general form of a central metric on the above quiver is clearly
\[g=\sum_{i,j=1}^2\lambda_{ij} e^{(i)}\tens e^{(j)}\ \in\Omega^1\tens_A\Omega^1=kQ_2,\]
where $\lambda_{ij}=\lambda_{ij}(e)\delta_e+\lambda_{ij}(g)\delta_g$ and  $(\lambda_{ij}(e)),\,(\lambda_{ij}(g))$ are arbitrary invertible $2\times 2$-matrices. Note that the values being invertible is equivalent to $\lambda=(\lambda_{ij})$ is invertible as an element in $M_2(A),$ as  $\lambda^{-1}=(\lambda_{ij}(e))^{-1}\delta_e+(\lambda_{ij}(g))^{-1}\delta_g$.

We now let $\sigma(e^{(i)}\tens e^{(j)})=\sum_{k,l=1}^2\sigma^{ij}_{kl}e^{(k)}\tens e^{(l)}$ be a general bimodule map with coefficients $\sigma^{ij}_{kl}$ functions on $\Z_2$ and study the associated bimodule connection $\nabla$ defined by $(\sigma,\alpha=0)$ in Lemma~\ref{inner-con} 1). We do not in general demand that $\sigma$ is invertible but this would be reasonable to also ask for. Light computation shows that  $\nabla$ is torsion-compatible (and thus torsion-free by 4) of Lemma~\ref{inner-con}  as $\alpha=0$) if and only if the coefficients $\sigma^{12}_{12}=\sigma^{12}_{21}-1$ and $\sigma^{21}_{12}=\sigma^{21}_{21}+1,$ therefore we can assume that the coefficients of $\sigma$ take the form
\begin{equation}\label{abc}
\begin{split}
(\sigma^{ij}_{11})=(a^{ij})=(\mathbf{a}_1,\mathbf{a}_2),\quad (\sigma^{ij}_{12})=(b^{ij})-\begin{pmatrix}0 & 1\\ -1 & 0 \end{pmatrix},\\
(\sigma^{ij}_{21})=(b^{ij})=(\mathbf{b}_1,\mathbf{b}_2),\quad (\sigma^{ij}_{22})=(c^{ij})=(\mathbf{c}_1,\mathbf{c}_2),
\end{split}
\end{equation}
where $\mathbf{a}_1,\mathbf{a}_2,\mathbf{b}_1,\mathbf{b}_2,\mathbf{c}_1,\mathbf{c}_2$ are six $2$-vectors with entries in $A.$ Now compute $\nabla g=e^{(1)}\tens_A g-\sigma_{12}\sigma_{23}(g\tens_A e^{(1)}).$ A lengthy computation shows that:

\begin{lemma} $\nabla g=0,$ i.e. the associated bimodule connection $\nabla$ is also metric-compatible if and only if
\begin{gather}
\lambda\cdot \mathbf{b}_1={\lambda_{21}(a^{11}-1)+\lambda_{22}a^{21}-\partial \lambda_{21}\choose
\lambda_{11}c^{11}+\lambda_{12}c^{21}+\partial\lambda_{22}},\label{bacb}\\
M\cdot\mathbf{a}_2={R_g\lambda_{11}\choose R_g\lambda_{12}}-N\cdot\mathbf{a}_1,\quad
M(\mathbf{b}_2,\mathbf{c}_2)=-N(\mathbf{b}_1,\mathbf{c}_1),\label{a2b2c2-1}
\end{gather}
where $M=(\lambda(\mathbf{b}_1,\mathbf{c}_1))^{\mathrm{Tr}}$ and $N
=\left(\lambda(\mathbf{a}_1-{1\choose 0},\mathbf{b}_1)\right)^{\mathrm{Tr}}+\lambda^{\mathrm{Tr}}.$
\end{lemma}

Now denote $\lambda (\mathbf{b}_1,\mathbf{c}_1)=\begin{pmatrix} x & y_1\\ y_2 & z \end{pmatrix}$ by some functions $x,y_1,y_2,z.$ Then (\ref{bacb}) implies 
$y_2=y_1+\partial\lambda_{22}$ and 
$ (\lambda_{21},\lambda_{22})\cdot \left(\mathbf{a}_1-{1\choose 0}\right)=x+\partial\lambda_{21}.$ If we introduce another function $w=(\lambda_{11},\lambda_{12})\cdot \left(\mathbf{a}_1-{1\choose 0}\right),$ we actually obtain a $4$-functional parameter moduli space for $\mathbf{a}_1,\mathbf{b}_1,\mathbf{c}_1:$
\begin{equation}\label{a1b1c1}
\mathbf{a}_1=\lambda^{-1}{w\choose x+\partial\lambda_{21}}+{1\choose 0},\quad (\mathbf{b}_1,\mathbf{c}_1)=\lambda^{-1}\begin{pmatrix}x & y\\ y+\partial\lambda_{22} & z \end{pmatrix},
\end{equation} 
where $x,y,z,w$ are any functions on $\Z_2$ and provided we can solve for the remaining $\bf a_2,b_2,c_2$ as follows.

\begin{proposition}  If $M$ (or $(\mathbf{b}_1,\mathbf{c}_1)$) is invertible in $M_2(A),$ i.e. $xz-y(y+\partial\lambda_{22})\neq 0,$
then  (\ref{a1b1c1}) and \begin{equation}\label{a2b2c2}
\mathbf{a}_2=M^{-1}{R_g\lambda_{11}\choose R_g\lambda_{12}}-M^{-1}N\cdot\mathbf{a}_1,\quad
(\mathbf{b}_2,\mathbf{c}_2)=-M^{-1}N(\mathbf{b}_1,\mathbf{c}_1) 
\end{equation}
with $M=\begin{pmatrix}x, & y+\partial\lambda_{22} \\ y, & z \end{pmatrix}$ and $N=\begin{pmatrix}w, & x+\partial\lambda_{21} \\ x, & y+\partial\lambda_{22} \end{pmatrix}+\lambda^{\mathrm{Tr}},$ solve (\ref{bacb}) and (\ref{a2b2c2-1}) and thus
provide a 4-functional parameter moduli space of quantum Levi-Civita bimodule connections (i.e. torsion-free, torsion-compatible and metric-compatible). 
\end{proposition}

Therefore, in view of Lemma  \ref{innernabla}, the bimodule connections $(\nabla,\sigma)$ corresponding to $\sigma$ above with $\alpha=0$ takes the form
\begin{align*}
\nabla e^{(1)}=&-\left(\lambda^{-1}_{11}w+\lambda^{-1}_{12}(x+\partial\lambda_{21})\right)e^{(1)}\tens e^{(1)}
-\left(\lambda^{-1}_{11}x+\lambda^{-1}_{12}(y+\partial\lambda_{22})\right)(e^{(1)}\tens e^{(2)}+e^{(2)}\tens e^{(1)})\\
{}&-\left(\lambda^{-1}_{11}y+\lambda^{-1}_{12}z\right)e^{(2)}\tens e^{(2)}\\
\nabla e^{(2)}=&-\left(\lambda^{-1}_{21}w+\lambda^{-1}_{22}(x+\partial\lambda_{21})\right)e^{(1)}\tens e^{(1)}
-\left(\lambda^{-1}_{21}x+\lambda^{-1}_{22}(y+\partial\lambda_{22})\right)(e^{(1)}\tens e^{(2)}+e^{(2)}\tens e^{(1)})\\
{}&-\left(\lambda^{-1}_{21}y+\lambda^{-1}_{22}z\right)e^{(2)}\tens e^{(2)},
\end{align*}
with $\sigma(e^{(1)}\tens e^{(1)})=e^{(1)}\tens e^{(1)}-\nabla e^{(1)}$ and $\sigma(e^{(2)}\tens e^{(1)})=e^{(1)}\tens e^{(1)}-\nabla e^{(2)}.$ 
We omit the tedious formulae for $\sigma(e^{(1)}\tens e^{(2)})$ and $\sigma(e^{(2)}\tens e^{(2)}),$ but one can always compute them from (\ref{a2b2c2}).

To have a concrete example, for arbitrary central metric $g$, one can always choose $x=z=1$ and $y=w=0.$ In this canonical case, 
\begin{align*}
\nabla e^{(1)}=&-(\partial\lambda_{21}+1)\lambda^{-1}{}_{12}e^{(1)}\tens e^{(1)}-(\lambda^{-1}{}_{11}+\partial\lambda_{22}\lambda^{-1}{}_{12})(e^{(1)}\tens e^{(2)}+e^{(2)}\tens e^{(1)})\\
&-\lambda^{-1}{}_{12} e^{(2)}\tens e^{(2)},\\
\nabla e^{(2)}=&-(\partial\lambda_{21}+1)\lambda^{-1}{}_{22}e^{(1)}\tens e^{(1)}-(\lambda^{-1}{}_{21}+\partial\lambda_{22}\lambda^{-1}{}_{22})(e^{(1)}\tens e^{(2)}+e^{(2)}\tens e^{(1)})\\
&-\lambda^{-1}{}_{22} e^{(2)}\tens e^{(2)}.
\end{align*}
When the given central metric is symmetric ($\lambda_{12}=\lambda_{21}$), we obtain $\sigma$ with coefficients
\begin{align*}
\sigma^{ij}_{11}&=\begin{pmatrix}(\partial\lambda_{21}+1)\lambda^{-1}{}_{12} +1, & {-\lambda^{-1}{}_{22}+\partial\lambda_{11}+\partial\lambda_{22}(\lambda^{-1}{}_{12}+\partial\lambda_{21}\lambda^{-1}{}_{12}+2)+(\partial\lambda_{22})^2(\partial\lambda_{21}\lambda^{-1}{}_{22}+\lambda^{-1}{}_{22})\atop
-2\partial\lambda_{21}\lambda^{-1}{}_{22}-(\partial\lambda_{21})^2\lambda^{-1}{}_{22},}\\ (\partial\lambda_{21}+1)\lambda^{-1}{}_{22}, & -\lambda^{-1}{}_{12}-2-\partial\lambda_{21}\lambda^{-1}{}_{12}-\partial\lambda_{22}\partial\lambda_{21}\lambda^{-1}{}_{22}-\partial\lambda_{22}\lambda^{-1}{}_{22} \end{pmatrix},\\ \sigma^{ij}_{12}&=\begin{pmatrix}\lambda^{-1}{}_{11}+\partial\lambda_{22}\lambda^{-1}{}_{12}, & {-\lambda^{-1}{}_{21}-2+\partial\lambda_{22}(\lambda^{-1}{}_{11}-\lambda^{-1}{}_{22}-\partial\lambda_{21}\lambda^{-1}{}_{22})-\partial\lambda_{21}\lambda^{-1}{}_{21}\atop (\partial\lambda_{22})^2(\lambda^{-1}{}_{12}+\lambda^{-1}{}_{21}+1)+(\partial\lambda_{22})^3\lambda^{-1}{}_{22},}\\ \lambda^{-1}{}_{21}+1+\partial\lambda_{22}\lambda^{-1}{}_{22}, & -\lambda^{-1}{}_{11}-\partial\lambda_{22}(1+\lambda^{-1}{}_{12}+\lambda^{-1}{}_{21})-(\partial\lambda_{22})^2\lambda^{-1}{}_{22} \end{pmatrix},\\
\sigma^{ij}_{21}&=\begin{pmatrix}\lambda^{-1}{}_{11}+\partial\lambda_{22}\lambda^{-1}{}_{12}, & {-\lambda^{-1}{}_{21}-1+\partial\lambda_{22}(\lambda^{-1}{}_{11}-\lambda^{-1}{}_{22}-\partial\lambda_{21}\lambda^{-1}{}_{22})-\partial\lambda_{21}\lambda^{-1}{}_{21}\atop (\partial\lambda_{22})^2(\lambda^{-1}{}_{12}+\lambda^{-1}{}_{21}+1)+(\partial\lambda_{22})^3\lambda^{-1}{}_{22},}\\ \lambda^{-1}{}_{21}+\partial\lambda_{22}\lambda^{-1}{}_{22}, & -\lambda^{-1}{}_{11}-\partial\lambda_{22}(1+\lambda^{-1}{}_{12}+\lambda^{-1}{}_{21})-(\partial\lambda_{22})^2\lambda^{-1}{}_{22} \end{pmatrix},\\
\sigma^{ij}_{22}&=\begin{pmatrix}\lambda^{-1}{}_{12}, & -\lambda^{-1}{}_{22}+\partial\lambda_{22}(1+\lambda^{-1}{}_{12})+(\partial\lambda_{22})^2\lambda^{-1}{}_{22}-\partial\lambda_{21}\lambda^{-1}{}_{22},\\ \lambda^{-1}{}_{22}, & -\lambda^{-1}{}_{12}-1-\partial\lambda_{22}\lambda^{-1}{}_{22} \end{pmatrix}.
\end{align*}

\begin{remark}  If $M$ (or $(\mathbf{b}_1,\mathbf{c}_1)$) is not invertible, i.e. $xz=y(y+\partial\lambda_{22}),$ we have other choices of bimodule connection. For instance, we can take $x=y=z=0.$ In this case
\begin{equation*}
\mathbf{a}_1={\lambda^{-1}{}_{11}w+\lambda^{-1}{}_{12}\partial\lambda_{21}+1\choose \lambda^{-1}{}_{21}w+\lambda^{-1}{}_{22}\partial\lambda_{21}},\quad
\mathbf{b}_1={\lambda^{-1}{}_{12}\partial\lambda_{22}\choose \lambda^{-1}{}_{22}\partial\lambda_{22}},\quad \mathbf{c}_1={0\choose 0},
\end{equation*}
and the condition (\ref{a2b2c2-1}) for being Levi-Civita requires
\begin{align}\label{a2b2c200w}
\partial\lambda_{22}a^{22}&=\partial\lambda_{11}+\lambda_{11}(\lambda^{-1}{}_{21}-\lambda^{-1}{}_{12})\partial\lambda_{21}-\lambda^{-1}{}_{22}\partial\lambda_{21}{}^2-(\lambda^{-1}{}_{12}+\lambda^{-1}{}_{21})\partial\lambda_{21}w\nonumber\\
&{}\quad-(1+\lambda^{-1}{}_{21}\lambda_{21}+\lambda^{-1}{}_{22}\lambda_{22})w-\lambda^{-1}{}_{11}w^2,\nonumber\\
0&=\partial\lambda_{21}-(\lambda^{-1}{}_{12}\lambda_{12}+\lambda^{-1}{}_{22}\lambda_{22})\partial\lambda_{21}-\lambda^{-1}{}_{22}\partial\lambda_{21}\partial\lambda_{22}-\lambda^{-1}{}_{21}\partial\lambda_{22}w\nonumber\\
&{}\quad+(\lambda^{-1}{}_{12}-\lambda^{-1}{}_{21})\lambda_{22}w,\\
\partial\lambda_{22}b^{22}&=\lambda_{11}(\lambda^{-1}{}_{21}-\lambda^{-1}{}_{12})\partial\lambda_{22}-\lambda^{-1}{}_{22}\partial\lambda_{21}\partial\lambda_{22}-\lambda^{-1}{}_{12}\partial\lambda_{22}w,\nonumber\\
0&=-(\lambda^{-1}{}_{12}\lambda_{12}+\lambda^{-1}{}_{22}\lambda_{22})-\lambda^{-1}{}_{22}\partial\lambda_{22}{}^2,\nonumber\\
\partial\lambda_{22}c^{22}&=0.\nonumber
\end{align}
We will come to solve these equations in the left-covariant case later. \end{remark}

\subsection{Left-covariant connection}
When working over a Hopf algebra, we are interested in left-covariant differential calculi as well as bicovariant ones. There is a notion of left-covariant connection accordingly.

\begin{definition} Let $(\Omega^1,\extd)$ be a left-covariant differential calculus over a Hopf algebra $H.$
A linear connection $\nabla:\Omega^1\to\Omega^1\tens_H\Omega^1$ is said to be \textit{left-covariant} if $\nabla$ is a left $H$-comodule map.  
A bimodule connection $(\nabla,\sigma)$ is \textit{left-covariant} if $\nabla,\sigma$ are both left $H$-comodule maps. When $(\Omega^1,\extd)$ is bicovariant, a bimodule connection $(\nabla,\sigma)$ is said to be \textit{bicovariant} if $\nabla,\sigma$ are both $H$-bicomodule maps.
\end{definition}

\begin{proposition}\label{lcon}
 Let $(\Omega^1,\extd)$ be a left-covariant differential calculus over  Hopf algebra $H.$ 
\begin{itemize} 
\item[a)] Left-covariant linear connections $\nabla:\Omega^1\to\Omega^1\tens_H\Omega^1$ are in one-to-one correspondence with linear maps $\nabla^L:\Lambda^1\to\Lambda^1\tens\Lambda^1.$ 
In one direction, $\nabla^L$ is the restriction of $\nabla$ to the left covariant part $\Lambda^1.$ In the converse direction, $\nabla$ is determined by $\nabla^L$ via
\begin{equation}\label{nablaL}
\nabla(\eta)=\extd\eta_{(-2)}\tens_H S(\eta_{(-1)})\eta_{(0)}+\eta_{(-2)}\nabla^L(S(\eta_{(-1)})\eta_{(0)}),\quad\forall\,\eta\in\Omega^1;
\end{equation}
\item[b)]  Left-covariant bimodule conections $(\nabla,\sigma)$ are in one-to-one correspondence with pairs $(\nabla^L,\sigma^L)$ where $\nabla^L:\Lambda^1\to\Lambda^1\tens\Lambda^1$ is a linear map and  $\sigma^L:\Lambda^1\tens\Lambda^1\to\Lambda^1\tens\Lambda^1$ is a right $H$-module map (under right $H$-action $\ra:\Lambda\tens H\to \Lambda ^1$) satisfying 
\begin{equation}\label{sigmaL}
\sigma^L(\xi\tens\omega\pi(h))=\omega\pi(h\t)\tens \xi\ra (S(h\o)h\th)+\nabla^L(\xi)\epsilon(h)-\nabla^L(\xi\ra Sh\o)\ra h\t
\end{equation}
for any $\xi\in\Lambda^1,\,h\in H.$
In addition to the correspondence in a), $\sigma^L$ is the restriction of $\sigma$ on $\Lambda^1\tens\Lambda^1$ in one direction. In the converse direction, $\sigma$  is determined by $\sigma^L$ via
\begin{equation}\label{sigma}
\sigma(\xi\tens_H\eta)=\xi\mt\eta\mth\sigma^L((S\xi\mo\xi\z)\ra\eta\mt\tens S\eta\mo\eta\z),\quad\forall\,\xi,\eta\in\Omega^1.
\end{equation}
\item[c)] If $(\Omega^1,\extd)$ is bicovariant ($\Lambda^1$ becomes a right $H$-crossed module), bicovariant  bimodule conections $(\nabla,\sigma)$ correspond to pairs $(\nabla^L,\sigma^L)$ in $b)$ whereas in addition $\nabla^L$ is a $H$-comdule map and $\sigma^L$ is a $H$-crossed module map.
\end{itemize}
\end{proposition}
\proof a) By definition, $\nabla$ being left-covariant means $(\id\tens\nabla)\Delta_L=\Delta\bt_L\nabla:\Omega^1\to H\tens\Omega^1\tens_H\Omega^1,$ i.e. $\xi\mo\tens\xi\z{}^1\tens_H\xi\z{}^2=\xi^1\mo\xi^2\mo\tens \xi^1\z\tens_H\xi^2\z$ for any $\xi\in\Omega^1,$ where we use $\nabla\xi=\xi^1\tens_H\xi^2$ and $\Delta_L(\xi)=\xi\mo\tens\xi\z$ denote the connection and left coaction respectively. Now let $\xi\in\Lambda^1,$ i.e. $\xi\mo\tens\xi\z=1_H\tens\xi,$ then we know $1\tens\nabla^L(\xi)=\Delta\bt_L\nabla^L(\xi),$ this means that $\nabla^L(\xi)\in {}^{\textrm{co}H}(\Omega^1\tens_H\Omega^1)=\Lambda^1\tens\Lambda^1,$ thus $\nabla^L:\Lambda^1\to\Lambda^1\tens\Lambda^1.$
Conversely, given any linear map $\nabla^L:\Lambda^1\to\Lambda^1\tens\Lambda^1,$ we can extend it to a map $\nabla:\Omega^1\to\Omega^1\tens_H\Omega^1$ by the property of a linear connection, namely
\[\nabla(\eta)=\nabla(\eta\mt. S(\eta\mo)\eta\z)=\extd\eta_{(-2)}\tens_H S(\eta_{(-1)})\eta_{(0)}+\eta_{(-2)}\nabla^L(S(\eta_{(-1)})\eta_{(0)})
\] for any $\eta\in\Omega^1.$ One can check immediately that 
$\nabla(h\eta)=\extd h\tens_H\eta+h\nabla\eta$
for any $h\in H,\,\eta\in\Omega^1$ by the property of Hopf algebra. Also, linear connection $\nabla$ defined in this way is left-covariant as
\begin{align*}
\Delta\bt_L\nabla\eta&=\Delta\bt_L\left(\extd\eta\mt\tens_H S(\eta\mo)\eta\z+\eta\mt.\nabla^L(S\eta\mo\eta\z)\right)\\
&=\eta\mth\tens\extd\eta\mt\tens_H S\eta\mo\eta\z+\eta\mth\tens\eta\mt\nabla^L(S\eta\mo\eta\z)\\
&=\eta_{(-1)}\tens\nabla\eta_{(0)},
\end{align*} noting that element $S\eta\mo\eta\z$ is left invariant. It is clear that the correspondence by restriction and (\ref{nablaL}) is one-to-one.

b) If $(\nabla,\sigma)$ is a left-covariant bimodule connection, the statement on $\nabla^L$ is clear due to a). We can focus on data $\sigma$ here. In the restriction direction, it is clear that $\sigma$ restricts to a linear map $\sigma^L:\Lambda^1\tens\Lambda^1\to\Lambda^1\tens\Lambda^1,$ as $\sigma$ is a left $H$-comodule map by definition. Since the right $H$-action $\ra$ on $\Lambda^1$ is defined by $\xi\ra h=Sh\o.\xi.h\t$, it is natural that $\sigma^L$ is a right $H$-module map under $\ra$, as $\sigma$ is a bimodule map.
For any $\xi\in\Lambda^1,\,h\in H,$ noting that $\omega\pi(h)=Sh\o\extd h\t\in\Lambda^1,$ we have 
\begin{align*}
\sigma^L(\xi\tens\omega\pi(h))&=\sigma(\xi\tens Sh\o\extd h\t)=\sigma(\xi Sh\o\tens_H\extd h\t)\\
&=\nabla(\xi Sh\o.h\t)-\nabla(\xi Sh\o).h\t\\
&=\nabla^L(\xi)\epsilon(h)-\left(\extd(Sh\th)\tens_H S^2(h\t)\xi S h\o\right).h\four\\
&{}\qquad-\left(Sh\th\tens\nabla^L((S^2 h\t)\xi Sh\o)\right).h\four\\
&=\nabla^L(\xi)\epsilon(h)-\left(\extd(Sh\t)\tens_H \xi \ra S h\o\right).h\th\\
&{}\qquad-\left(Sh\t\tens\nabla^L(\xi\ra Sh\o)\right).h\th\\
&=\nabla^L(\xi)\epsilon(h)-\extd(Sh\t)h\th\tens_H \xi \ra S h\o\ra h\four\\
&{}\qquad-Sh\t h\th\tens\nabla^L(\xi\ra Sh\o)\ra h\four\\
&=\nabla^L(\xi)\epsilon(h)-\extd(Sh\t)h\th\tens_H \xi \ra S h\o\ra h\four-\nabla^L(\xi\ra Sh\o)\ra h\t\\
&=\omega\pi(h\t)\tens \xi\ra(Sh\o h\th)+\nabla^L(\xi)\epsilon(h)-\nabla^L(\xi\ra Sh\o)\ra h\t
\end{align*}
where by (\ref{nablaL}) the fourth equality and by $\extd(S h\t)h\th=\left(S h\th\tens \omega\pi(Sh\t)\right).h\four=Sh\th.h\four\tens\omega\pi(Sh\t)\ra h_{(5)}=\omega\pi(Sh\t)\ra h\th=\omega\pi(Sh\t h\th)-\epsilon(Sh\t)\omega\pi(h\th)=-\omega\pi(h\t)$ the last equality hold. We obtain (\ref{sigmaL}) as stated.  

Conversely, given a linear map $\nabla^L$ and a right $H$-module map $\sigma^L$ such that (\ref{sigmaL}). We know already $\nabla$ defined by (\ref{nablaL}) is a left-covariant linear connection. Now consider $\sigma:\Omega^1\tens_H\Omega^1\to\Omega^1\tens_H\Omega^1$ defined by (\ref{sigma}), namely
\begin{equation*}
\sigma(\xi\tens_H\eta)=\xi\mt\eta\mth\,\sigma^L((S(\xi\mo)\xi\z)\ra\eta\mt\tens S(\eta\mo)\eta\z),\quad\forall\,\xi,\eta\in\Omega^1.
\end{equation*}
It is easy to see that $\sigma$ is well-defined by checking that $\sigma$ defined above obeys $\sigma(\xi h\tens_H\eta)=\sigma(\xi\tens_H h\eta)$ for any $h\in H$.
The pair $(\nabla,\sigma)$ constructed forms a bimodule connection if 
\begin{equation}\label{check-b}
\sigma(\eta\tens_H\extd h)=\nabla(\eta h)-\nabla(\eta)h
\end{equation}
for any $\eta\in\Omega^1,\,h\in H,$ which we now check. The left-hand side of (\ref{check-b}) is
\begin{align*}
\sigma(\eta\tens_H\extd h)&=\eta\mt h\o\tens\sigma^L\left((S(\eta\mo)\eta\z)\ra h\t\tens S(h\th)\extd h\four\right)\\
&=\eta\mt h\o\tens\sigma^L\left((S(\eta\mo)\eta\z)\ra h\t\tens \omega\pi(h\th)\right)\\
&=\eta\mt h\o\tens \omega\pi(h\t)\tens (S(\eta\mo)\eta\z)\ra h\th+\eta\mt h\o\,\nabla^L((S(\eta\mo)\eta\z)\ra h\t)\\
&{}\qquad-\eta\mt h\o\,\nabla^L(S(\eta\mo)\eta\z)\ra h\t.
\end{align*} 
The right-hand side of (\ref{check-b}) is
\begin{align*}
\nabla(\eta h)-\nabla(\eta) h&=\extd(\eta\mt h\o)\tens_H S(\eta\mo h\t)\eta\z h\th+\eta\mt h\o\nabla^L (Sh\t S\eta\mo\eta\z h\th)\\
&{}\quad -(\extd\eta\mt)h\o\tens_H (S(\eta\mo)\eta\z)\ra h\t-\eta\mt h\o\nabla^L(S(\eta\mo)\eta\z)\ra h\t\\
&=\eta\mt\extd h\o\tens (S\eta\mo\eta\z)\ra h\t+\eta\mt h\o\nabla^L ((S\eta\mo\eta\z)\ra h\t)\\
&{}\quad -\eta\mt h\o\nabla^L(S(\eta\mo)\eta\z)\ra h\t.
\end{align*}
We know two sides of (\ref{check-b}) meets as $\extd h=h\o\tens\omega
\pi(h\t)$ for any $h\in H.$ It is left to check $\sigma$ is indeed a bimodule map. The left $H$-linearity is obvious by the form of (\ref{sigma}). As to the right $H$-linearity, we check 
\begin{align*}
\sigma(\xi\tens_H\eta h)&=\xi\mo\eta\mth h\o\sigma^L\left((S\xi\mo\xi\z)\ra (\eta\mo h\t)\tens S(\eta\mo h\th)\eta\z h\four\right)\\
&=\xi\mt\eta\mth h\o \sigma^L\left(\left((S\xi\mo\xi\z)\ra\eta\mt\tens S\eta\mo\eta\z\right)\ra h\right)\\
&=\xi\mt\eta\mth h\o \sigma^L\left((S\xi\mo\xi\z)\ra\eta\mt\tens S\eta\mo\eta\z\right)\ra h\\
&=\sigma(\xi\tens_H \eta).h
\end{align*}
since $\sigma^L$ is right $H$-module map. This finishes the proof of b).

c) If $(\nabla,\sigma)$ is bicovariant bimodule connection, it is obvious that both $\nabla^L,\sigma^L$ are right $H$-comodule maps as these maps are restriction of $H$-comodule maps $\nabla,\sigma$ in which the right coaction on $\Lambda^1$ is the same right coaction on  $\Omega^1.$ The converse is straightforward, as one can check $\nabla$ defined by $\nabla^L$ via (\ref{nablaL}) (or $\sigma$ defined by $\sigma^L$ via (\ref{sigma})) commutes with right coaction if $\nabla^L$ (or $\sigma^L$) does.
\endproof

Let $(\Omega^1,\extd)$ be a left-covariant differential calculus over Hopf algebra $H.$ If $\Omega^1$ extends to $\Omega^2$ with wedge product $\wedge:\Omega^1\tens_H\Omega^1\to\Omega^2$. We say $\Omega^2$ is \textit{left-covariant} if 1) $\Omega^2$ is a $H$-bimodule with left $H$-coaction $\Delta_L$ being a $H$-bimodule map; 2) $\extd:\Omega^1\to\Omega^2$ is left $H$-comodule map; 3) The wedge product $\wedge$ is a $H$-bimodule left $H$-comodule map. Under these assumptions, we know that if $\nabla$ is a left-covariant connection given by some linear map $\nabla^L$ as in Proposition~\ref{lcon} a), then the associated torsion $T_\nabla$ and curvature $R_\nabla$ must satisfy $T_\nabla(\Lambda^1)\subseteq\Lambda^2$ and $R_\nabla(\Lambda^1)\subseteq\Lambda^2\tens\Lambda^1$, thus are determined by their restrictions 
\begin{gather}
T_\nabla^L:\Lambda^1\to\Lambda^2,\quad T_\nabla^L=\wedge\nabla^L-\delta,\label{ltor}\\
R_\nabla^L:\Lambda^1\to\Lambda^2\tens\Lambda^1,\quad R_\nabla^L=\left(\delta\tens\id-(\wedge\tens\id)(\id\tens\nabla^L)\right)\nabla^L\label{lcur}
\end{gather}
respectively, where $\wedge:\Lambda^1\tens\Lambda^1\to\Lambda^2$ is the restriction of $\wedge$ and $\delta:\Lambda^1\to\Lambda^2$ is the restriction of $\extd:\Omega^1\to\Omega^2.$ Clearly $T_\nabla(a.\xi)=a.T_\nabla^L\xi$ and $R_\nabla(a.\xi)=a.R_\nabla^L(\xi),$ due to both $T_\nabla$ and $R_\nabla$ are left $H$-module map by definition.

 We have the following elementary lemma:
\begin{lemma} Let $(\nabla,\sigma)$ be a left-covariant bimodule connection. For $g=a\xi\tens\eta\in\Omega^1\tens_A\Omega^1,$ $\nabla g=\extd a\tens\xi\tens\eta+a\nabla^L(\xi)\tens\eta+a\sigma^L(\xi\tens\eta^1)\tens\eta^2,$ where $\nabla^L\eta=\eta^1\tens\eta^2.$ When $g$ is a metric, under the above assumptions on $\Omega^2,$ the bimodule connection $(\nabla,\sigma)$ is Levi-Civita if $\wedge\nabla^L=\delta$, $\wedge(\id+\sigma^L)=0$ and $\nabla g=0.$
\end{lemma}

We now apply the above theory to our example. Indeed, the $4$D inner generalised calculus on $A=\mathbf{k}(\Z_2)$ in Section~5.4 is bicovariant with respect to coactions
\begin{gather*}
\Delta_L\alpha_i=\delta_e\tens\alpha_i+\delta_g\tens\beta_i,\quad\Delta_R\alpha_i=\alpha_i\tens\delta_e+\beta_i\tens\delta_g,\\
\Delta_L\beta_i=\delta_e\tens\beta_i+\delta_g\tens\alpha_i,\quad\Delta_R\beta_i=\beta_i\tens\delta_e+\alpha_i\tens\delta_g,
\end{gather*}
and $\{e^{(1)},\,e^{(2)}\}$ is the basis of left-covariant $1$-forms $\Lambda^1$ with $\Delta_R e^{(i)}=e^{(i)}\tens 1$ and $\delta(e^{(i)})=0$ for $i=1,2.$
From Proposition~\ref{lcon},  we know that bicovariant bimodule connections $(\nabla,\sigma)$ correspond to pairs $(\nabla^L,\sigma^L)$ where $\nabla^L,\sigma^L$ are linear maps and must be in form
\begin{gather*}
\nabla^L e^{(i)}=e^{(1)}\tens e^{(i)}-\sigma^L(e^{(i)}\tens e^{(1)}),\quad
\sigma^L(e^{(i)}\tens e^{(j)})=\sum_{k.l=1}^2\sigma^{ij}_{kl}e^{(k)}\tens e^{(l)},
\end{gather*}
for some $\sigma^{ij}_{kl}\in \mathbf{k}$ for all $i,j,k,l\in\{1,2\}.$ It is natural to ask the given metric $g=\sum_{i,j=1}^2\lambda_{ij}e^{(i)}\tens e^{(j)}$ is \textit{invariant}, i.e. left-invariant and right-invariant under the coactions. This means that the given $g$ is constant metric, i.e. $\partial\lambda_{ij}=0,$ in this example. Similar to the analysis in Section~5.4, we know that $(\nabla^L,\sigma^L)$ provides a bicovariant quantum Levi-Civita connection if and only if  the coefficients of $\sigma$ (still denoted by $\mathbf{a}_i,\mathbf{b}_i,\mathbf{c}_i$ for $i=1,2$ as in (\ref{abc})) must be given by some constants $x,y,z,w$ such that 1) $\mathbf{a}_1,\mathbf{b}_1,\mathbf{c}_1$ are given by
\begin{equation}\label{a1b1c1-c}
\mathbf{a}_1=\lambda^{-1}{w\choose x}+{1\choose 0},\quad \mathbf{b}_1=\lambda^{-1}{x\choose y},\quad \mathbf{c}_1=\lambda^{-1}{y\choose z}
\end{equation}
and 2) $\mathbf{a}_2,\mathbf{b}_2,\mathbf{c}_2$ satisfy
\begin{equation}\label{a2b2c2-c}
M\cdot \mathbf{a}_2={\lambda_{11}\choose\lambda_{12}}-N\cdot \mathbf{a}_1,\quad M\cdot \mathbf{b}_2=-N\cdot\mathbf{b}_1,\quad M\cdot\mathbf{c}_2=-N\cdot \mathbf{c}_1,
\end{equation}
where $M=\begin{pmatrix}x, & y \\ y, & z \end{pmatrix}$ and $N=\begin{pmatrix}w, & x \\ x, & y \end{pmatrix}+\lambda^{\mathrm{Tr}}.$

\begin{proposition} If $M$ is invertible, i.e. $xz-y^2\neq 0,$ (\ref{a2b2c2-c}) has a unique solution of $\mathbf{a}_2,\mathbf{b}_2,\mathbf{c}_2:$
\begin{equation*}
\mathbf{a}_2=M^{-1}{\lambda_{11}\choose \lambda_{12}}-M^{-1}N\cdot\mathbf{a}_1,\quad
\mathbf{b}_2=-M^{-1}N\cdot\mathbf{b}_1,\quad \mathbf{c}_2=-M^{-1}N\cdot\mathbf{c}_1.
\end{equation*}
\end{proposition}

The resulting formulae for $\nabla^L$ and $\sigma^L$ as follows:
\begin{align*}
\nabla^L e^{(1)}=&-\left(\lambda^{-1}{}_{11}w+\lambda^{-1}{}_{12}x\right)e^{(1)}\tens e^{(1)}
-\left(\lambda^{-1}{}_{11}x+\lambda^{-1}{}_{12}y\right)(e^{(1)}\tens e^{(2)}+e^{(2)}\tens e^{(1)})\\
{}&-\left(\lambda^{-1}{}_{11}y+\lambda^{-1}{}_{12}z\right)e^{(2)}\tens e^{(2)}\\
\nabla^L e^{(2)}=&-\left(\lambda^{-1}{}_{21}w+\lambda^{-1}{}_{22}x\right)e^{(1)}\tens e^{(1)}
-\left(\lambda^{-1}{}_{21}x+\lambda^{-1}{}_{22}y\right)(e^{(1)}\tens e^{(2)}+e^{(2)}\tens e^{(1)})\\
{}&-\left(\lambda^{-1}{}_{21}y+\lambda^{-1}{}_{22}z\right)e^{(2)}\tens e^{(2)},
\end{align*}
and $\sigma^L(e^{(1)}\tens e^{(1)})=e^{(1)}\tens e^{(1)}-\nabla^L e^{(1)}$ and $\sigma^L(e^{(2)}\tens e^{(1)})=e^{(1)}\tens e^{(2)}-\nabla^L e^{(2)}.$ 
When the given metric is symmetric ($\lambda_{12}=\lambda_{21}$), we get 
\begin{align*}
\sigma^L(e^{(1)}\tens e^{(2)})&=\left(\frac{2xy}{m}+\lambda^{-1}{}_{11}\frac{w(xy-wz)}{m}+\lambda^{-1}{}_{12}(w+\frac{x(xy-wz)}{m})-\lambda^{-1}{}_{22}x\right)e^{(1)}\tens e^{(1)}\\
&{}\quad +\left(-2+\lambda^{-1}{}_{11}\frac{x(xy-wz)}{m}-\lambda^{-1}{}_{12}(x-\frac{y(xy-wz)}{m})-\lambda^{-1}{}_{22}y\right)(e^{(1)}\tens e^{(2)}+e^{(2)}\tens e^{(1)})\\
&{}\quad+e^{(2)}\tens e^{(1)}\\
&{}\quad +\left(\lambda^{-1}{}_{11}\frac{(xy-wz)y}{m}-\lambda^{-1}{}_{12}(y-\frac{(xy-wz)z}{m})-\lambda^{-1}{}_{22}z\right)e^{(2)}\tens e^{(2)},\\
\sigma^L(e^{(2)}\tens e^{(2)})&=-\left(\frac{2x^2-wy}{m}+\lambda^{-1}{}_{11}\frac{w(x^2-wy)}{m}+\lambda^{-1}{}_{12}\frac{x(x^2-wy)}{m}\right)e^{(1)}\tens e^{(1)}\\
&{}\quad -\left(\lambda^{-1}{}_{11}\frac{x(x^2-wy)}{m}+\lambda^{-1}{}_{12}\frac{(x^2-wy)y}{m}
\right) (e^{(1)}\tens e^{(2)}+e^{(2)}\tens e^{(1)})\\
&{}\quad -\left(1+\lambda^{-1}{}_{11}\frac{(x^2-wy)y}{m}+\lambda^{-1}{}_{12}\frac{(x^2-wy)z}{m}\right)e^{(2)}\tens e^{(2)},
\end{align*}
where $m=\det M=xz-y^2.$
The curvature of corresponding Levi-Civita bimodule connection:
		\begin{align*}
			R_\nabla e^{(1)}
			&=\left((\lambda^{-1}{}_{11}y+\lambda^{-1}{}_{12}z)(\lambda^{-1}{}_{21}w+\lambda^{-1}{}_{22}x)-(\lambda^{-1}{}_{11}x+\lambda^{-1}{}_{12}y)(\lambda^{-1}{}_{21}x+\lambda^{-1}{}_{22}y)\right)e^{121}\\
			&+\left((\lambda^{-1}{}_{11}y+\lambda^{-1}{}_{12}z)(-\lambda^{-1}{}_{11}w+\lambda^{-1}{}_{22}y)+(\lambda^{-1}{}_{11}x+\lambda^{-1}{}_{12}y)(\lambda^{-1}{}_{11}x-\lambda^{-1}{}_{22}z)\right)e^{122},\\
			R_\nabla e^{(2)}
			&=\left((\lambda^{-1}{}_{21}x+\lambda^{-1}{}_{22}y)(\lambda^{-1}{}_{11}w-\lambda^{-1}{}_{22}y)-(\lambda^{-1}{}_{21}w+\lambda^{-1}{}_{12}y)(\lambda^{-1}{}_{11}x-\lambda^{-1}{}_{22}z)\right)e^{121}\\
			&+\left((\lambda^{-1}{}_{11}x+\lambda^{-1}{}_{12}y)(\lambda^{-1}{}_{21}x+\lambda^{-1}{}_{22}y)-(\lambda^{-1}{}_{11}y+\lambda^{-1}{}_{12}z)(\lambda^{-1}{}_{21}w+\lambda^{-1}{}_{22}x)\right)e^{122},
		\end{align*}
		where $e^{ijk}$ are shorthands for $e^{(i)}\wedge e^{(j)}\tens e^{(k)}$ for all $i,j,k=1,2.$

The canonical choice  (i.e. $x=1,y=0,z=1,w=0$) of bicovariant quantum Levi-Civita bimodule connection is given as follows.
\begin{align*}
\nabla^L e^{(1)}=&-\lambda^{-1}{}_{12} e^{(1)}\tens e^{(1)}
-\lambda^{-1}{}_{11}(e^{(1)}\tens e^{(2)}+e^{(2)}\tens e^{(1)})-\lambda^{-1}{}_{12}e^{(2)}\tens e^{(2)}\\
\nabla^L e^{(2)}=&-\lambda^{-1}{}_{22}e^{(1)}\tens e^{(1)}
-\lambda^{-1}{}_{21}(e^{(1)}\tens e^{(2)}+e^{(2)}\tens e^{(1)})-\lambda^{-1}{}_{22}e^{(2)}\tens e^{(2)},\\
\sigma^L(e^{(1)}\tens e^{(1)})&=(1+\lambda^{-1}{}_{12})e^{(1)}\tens e^{(1)}
+\lambda^{-1}{}_{11}(e^{(1)}\tens e^{(2)}+e^{(2)}\tens e^{(1)})
+\lambda^{-1}{}_{12}e^{(2)}\tens e^{(2)},\\
\sigma^L(e^{(2)}\tens e^{(1)})&=\lambda^{-1}{}_{22}e^{(1)}\tens e^{(1)}
+\lambda^{-1}{}_{21}(e^{(1)}\tens e^{(2)}+e^{(2)}\tens e^{(1)})+e^{(1)}\tens e^{(2)} +\lambda^{-1}{}_{22}e^{(2)}\tens e^{(2)},\\
\sigma^L(e^{(1)}\tens e^{(2)})&=-\lambda^{-1}{}_{22}e^{(1)}\tens e^{(1)}
-(2+\lambda^{-1}{}_{12})(e^{(1)}\tens e^{(2)}+e^{(2)}\tens e^{(1)})+e^{(2)}\tens e^{(1)}\\
&{}\quad -\lambda^{-1}{}_{22}e^{(2)}\tens e^{(2)},\\
\sigma^L(e^{(2)}\tens e^{(2)})&=-(2+\lambda^{-1}{}_{12})e^{(1)}\tens e^{(1)}-\lambda^{-1}{}_{11}
 (e^{(1)}\tens e^{(2)}+e^{(2)}\tens e^{(1)})\\
&{}\quad -(1+\lambda^{-1}{}_{12})e^{(2)}\tens e^{(2)}.
\end{align*}
The curvature of this canonical bicovariant quantum Levi-Civita bimodule connection is
\begin{align*}
R_{\nabla} e^{(1)}&=\lambda^{-1}{}_{12}(\lambda^{-1}{}_{22}-\lambda^{-1}{}_{11})e^{(1)}\wedge e^{(2)}\tens e^{(1)}+\lambda^{-1}{}_{11}(\lambda^{-1}{}_{11}-\lambda^{-1}{}_{22})e^{(1)}\wedge e^{(2)}\tens e^{(2)}\\
R_{\nabla} e^{(2)}&=\lambda^{-1}{}_{22}(\lambda^{-1}{}_{22}-\lambda^{-1}{}_{11})e^{(1)}\wedge e^{(2)}\tens e^{(1)}+\lambda^{-1}{}_{12}(\lambda^{-1}{}_{11}-\lambda^{-1}{}_{22})e^{(1)}\wedge e^{(2)}\tens e^{(2)}.
\end{align*}
Obviously, this canonical bimodule connection is flat when $\lambda_{11}=\lambda_{22}.$

\begin{remark} For $M$ not being invertible, i.e. $xz=y^2$, the simplest case is $M=0,$ i.e. $x=y=z=0.$ In this case, $\mathbf{a}_1={\lambda^{-1}{}_{11}w+1\choose \lambda^{-1}{}_{21}w},\,\mathbf{b}_1={0\choose 0},\,\mathbf{c}_1={0\choose 0}$. Then (\ref{a2b2c2-c}) (or (\ref{a2b2c200w}) where $\partial\lambda_{ij}=0$) reduces to $w(\lambda^{-1}{}_{11}w+2)=0.$ 

(i) If $w=0,$ we know $\mathbf{a}_1={1\choose 0}$ and thus
\begin{equation*}\label{w0}
\nabla^L=0,\quad\sigma^L(e^{(i)}\tens e^{(1)})=e^{(1)}\tens e^{(i)},
\end{equation*} and $\sigma^L(e^{(i)}\tens e^{(2)})=-\delta_{i1}e^{(1)}\tens e^{(2)}+a^{i2}e^{(1)}\tens e^{(1)}+b^{i2}(e^{(1)}\tens e^{(2)}+e^{(2)}\tens e^{(1)})+c^{i2}e^{(2)}\tens e^{(2)}$
for arbitrary constants $a^{i2},b^{i2},c^{i2}$ ($i=1,2$).
Clearly, $R_\nabla=0$ as $\nabla^L=0$. 

(ii) As $\lambda^{-1}{}_{11}\neq0,$ or $\lambda_{22}\neq 0,$ there is another solution $w=-\frac{2}{\lambda^{-1}{}_{11}}.$ In this case
\begin{gather*}
\nabla^L(e^{(1)})=0,\quad \nabla^L(e^{(2)})=-\frac{2\lambda_{21}}{\lambda_{22}}e^{(1)}\tens e^{(1)},\\
\sigma^L(e^{(1)}\tens e^{(1)})=e^{(1)}\tens e^{(2)},\quad \sigma^L(e^{(2)}\tens e^{(1)})=\frac{2\lambda_{21}}{\lambda_{22}}e^{(1)}\tens e^{(1)}+e^{(1)}\tens e^{(2)}.
\end{gather*} While $\sigma(e^{(i)}\tens e^{(2)})=-\delta_{i1}e^{(1)}\tens e^{(2)}+a^{i2}e^{(1)}\tens e^{(1)}+b^{i2}(e^{(1)}\tens e^{(2)}+e^{(2)}\tens e^{(1)})+c^{i2}e^{(2)}\tens e^{(2)}$
for arbitrary constants $a^{i2},b^{i2},c^{i2}$ ($i=1,2$).

These are some illustrative examples of bicovariant quantum Levi-Civita bimodule connections in this degenerate case. 
\end{remark}

\end{document}